\documentclass[11pt]{article}
\usepackage{latexsym,amsmath,amscd,amssymb,framed,graphics,mathrsfs}
\usepackage{enumerate}
\usepackage{graphicx}
\usepackage{subcaption}
\usepackage{color}
\usepackage{framed}
\usepackage{comment}
\usepackage{booktabs}
\usepackage{nicematrix}

\usepackage{tikz}
    \usetikzlibrary{matrix}
    \usetikzlibrary{decorations.markings}
\usepackage[authoryear,sort&compress]{natbib}
\bibpunct{[}{]}{;}{n}{,}{,}
\usepackage[colorlinks]{hyperref}
\usepackage{url}
\usepackage[all]{xy}

\newcommand{\rem}[1]{}
\newcommand\tr{\mathord{\textrm{\normalfont tr}\,}}
\newcommand\bcdot{\mathord{\,\boldsymbol{\cdot}\,}}
\newcommand\bwedge{\text{{\Large$\,\boldsymbol\wedge\,$}}}
\newcommand\nd{\mathord{{\rm d}^\nabla}}
\newcommand\blangle{\mathord{\boldsymbol\langle}}
\newcommand\brangle{\mathord{\boldsymbol\rangle}}

\newcommand{\Ad}{\textrm{\normalfont Ad}}

\newcommand{\partialnew}{\textsf{\reflectbox{6}}}
\makeatletter

\newenvironment{proof}[1][Proof]{\noindent\textbf{#1.} }{\ \rule{0.5em}{0.5em}}

\def\XXint#1#2#3{{\setbox0=\hbox{$#1{#2#3}{\int}$ }
\vcenter{\hbox{$#2#3$ }}\kern-.5\wd0}}

\definecolor{bgd}{RGB}{153,0,51}      %burgundy

\begin{document}

\newtheorem{theorem}{Theorem}[section]
\newtheorem{definition}[theorem]{Definition}
\newtheorem{lemma}[theorem]{Lemma}
\newtheorem{remark}[theorem]{Remark}
\newtheorem{proposition}[theorem]{Proposition}
\newtheorem{corollary}[theorem]{Corollary}
\newtheorem{example}[theorem]{Example}

%\def\below#1#2{\mathrel{\mathop{#1}\limits_{#2}}}

%%%%%%%%%%%%%%%%%%%%%%%%%%%%%%%%%%%%%%%%%%%%%%
%%%%%%%%%%%%%%%%%%%%%%%%%%%%%%%%%
%%%%%%%%

\title{Infinite-dimensional Lagrange--Dirac systems with\\ boundary energy flow II: Field theories with bundle-valued forms}
\author{Fran\c{c}ois Gay--Balmaz$^{1}$, Álvaro Rodríguez Abella$^{2}$ and Hiroaki Yoshimura$^{3}$}

\addtocounter{footnote}{1}
\footnotetext{Division of Mathematical Sciences, Nanyang Technological University, 21 Nanyang Link, Singapore 637371.
\texttt{francois.gb@ntu.edu.sg}
\addtocounter{footnote}{1} }

\footnotetext{Department of Applied Mathematics, ICAI School of Engineering, Comillas Pontifical University, Madrid, Spain.
\texttt{arabella@comillas.edu}
\addtocounter{footnote}{1}}

\footnotetext{Faculty of Science and Engineering, Waseda University. 3--4--1, Okubo, Shinjuku, Tokyo, Japan.
\texttt{yoshimura@waseda.jp}
\addtocounter{footnote}{1} }

\date{ }
\maketitle
\makeatother

%\begin{center} DRAFT \end{center}
%\maketitle

%|||-------------------text width----------------------|||

\begin{abstract}
Part I of this paper introduced the infinite dimensional Lagrange--Dirac theory for physical systems on the space of differential forms over a smooth manifold with boundary. This approach is particularly well-suited for systems involving energy exchange through the boundary, as it is built upon a restricted dual space - a vector subspace of the topological dual of the configuration space - that captures information about both the interior dynamics and boundary interactions. Consequently, the resulting dynamical equations naturally incorporate boundary energy flow. In this second part, the theory is extended to encompass vector-bundle-valued differential forms and non-Abelian gauge theories. To account for two commonly used forms of energy flux and boundary power densities, we introduce two distinct but equivalent formulations of the restricted dual. The results are derived from both geometric and variational viewpoints and are illustrated through applications to matter and gauge field theories. The interaction between gauge and matter fields is also addressed, along with the associated boundary conditions, applied to the case of the Yang–Mills–Higgs equations.
\end{abstract}

%\keywords{boundary energy flow, Fr\'echet space,  Lagrange--Dirac mechanics, Lagrangian density, Yang--Mills equations} 
%\subjclass{37K25, 70G45 (Primary), 35Q70, 70S15 (Secondary)}

\maketitle

\section{Introduction}

Determining the variational and geometric structures underlying the equations of continuum mechanics and field theories is fundamental not only for modeling these systems but also for guiding their structure-preserving discretization. An important class of such systems involves those that exchange energy with their surroundings through their boundary. It is crucial that these geometric and variational structures align with those of classical mechanics, specifically the Hamiltonian principle for evolution curves in the configuration manifold, and the canonical symplectic form on the phase space.

A geometric and variational framework for systems with boundary energy flow was developed in Part I \cite{GBRAYo2025I} of this paper, based on an infinite-dimensional extension of Lagrange--Dirac dynamical systems. The main advantage of this approach is that \textit{it satisfies a set of consistency requirements with the geometric and variational formulation of classical mechanics}, as detailed in Part I. In particular, the geometric structures underlying such systems are Dirac structures which are directly constructed from the canonical symplectic form on phase space. Furthermore, the solution curves of Lagrange--Dirac dynamical systems are characterized as the critical trajectories of a variational principle, the Lagrange--d'Alembert--Pontryagin principle, which consistently extends the Hamilton principle to include forces, as well as velocity and momentum as independent variables.
Finally, the geometric framework does not impose specific constraints on the form or regularity of the Lagrangian density.
Further advantages of this approach, particularly related to interconnections
and symmetry reduction, will be explored in future works.

A key step in the development of infinite dimensional Lagrange--Dirac dynamical systems with boundary energy flow is the careful choice of the dual space, which allows the description of boundary effects while keeping intact the canonical nature of the geometric and variational structures used in the finite-dimensional case. In Part I, we considered systems with a configuration manifold given by the Fr\'echet space $V=\Omega^k(M)$ of differential $k$-forms on an $m$-dimensional, compact manifold $M$ with smooth boundary. The so called \textit{restricted dual}, denoted $V^\star \subset V'$, with $V'$ the topological dual, was defined as
\[
V^\star = \Omega^{m-k}(M) \times \Omega^{m-k-1}(\partial M)
\]
with the duality paring
\[
\langle (\alpha, \alpha_\partial), \varphi\rangle = \int_M \varphi\wedge \alpha + \int_{\partial M}\iota_{\partial M}^*\varphi \wedge \alpha_\partial,\quad \varphi\in V,\;\;(\alpha, \alpha_\partial)\in V^\star,
\]
where $\iota_{\partial M}: \partial M\rightarrow M$ is the inclusion of the boundary.
Using this duality, the cotangent bundle can be defined as $T^\star V= V\times V^\star$, and its canonical symplectic form naturally includes boundary terms. Similarly, the induced canonical Dirac structure is a subbundle $D_{\rm can}\subset T(T^\star V) \oplus T^\star (T^\star V)$ which involves boundary conditions, so that the associated Lagrange--Dirac dynamical systems can describe boundary energy flow, see \cite{GBRAYo2025I}. This geometric framework was applied to nonlinear wave equations and the Maxwell equations with boundary energy flow.

In this paper, we expand the scope of the approach first by considering continuum systems with configuration manifolds given by the space $V=\Omega^k(M,E)$ of $E$-valued $k$-forms, where $\pi_{E,M}: E\rightarrow M$ is a vector bundle over $M$. Typical examples are matter fields, which take values in associated bundles. The space+time decomposition yields Lagrangian functions
\[
L_\nabla:TV\rightarrow \mathbb{R},
\]
which are expressed in terms of a Lagrangian density $\mathscr{L}$ as
\begin{equation}\label{L_nabla_1}
L_\nabla (\varphi, \dot \varphi)= \int_M \mathscr{L}\big( \varphi(x), \dot \varphi(x), {\rm d}^\nabla \varphi(x)\big).
\end{equation}
Here, $\nabla$ denotes a given linear connection on the vector bundle, and $TV=V\times V$ is the tangent bundle to $V=\Omega^k(M,E)$.

We also consider the case of gauge fields, whose configuration manifold is given by the affine space $\mathcal{C}(P)$  of principal connections of a principal $G$-bundle $\pi_{P,M}:P\rightarrow M$.
In the space+time decomposition and using the temporal gauge, the Lagrangian function for gauge fields is a map
\[
L:T \mathcal{C}(P)\rightarrow \mathbb{R},
\]
which can be expressed in terms of a Lagrangian density $\mathscr{L}$ as
\[
L_\nabla (A, \dot A)= \int_M \mathscr{L}\big( A(x), \dot A(x), {\rm d}^A A(x)\big),
\]
where $A$ is a principal connection and $B_A={\rm d}^AA$ is its curvature. Here, $T\mathcal{C}(P)=\mathcal{C}(P)\times\Omega^1(M,\tilde{\mathfrak{g}})$ is the tangent bundle of the affine space $\mathcal{C}(P)$ of connections, whose model vector space is the space of $\tilde{\mathfrak{g}}$-valued (adjoint bundle) $1$-forms on $M$.

Finally, we consider the interaction of matter and gauge fields via minimal coupling. The matter fields are sections of an associate bundle $\tilde V= P\times_G V \rightarrow M$ defined by a representation $\varrho:G\times V \rightarrow V$ of the structure group of the principal bundle on a vector space $V$. We denote by $\Omega^0(M, \tilde V)$ the space of such smooth sections. The coupling occurs through the gauge field (the principal connection) which induces a linear connection on the matter field. This results in a Lagrangian
\[
L: T Q  \rightarrow \mathbb{R},\qquad Q= \mathcal{C}(P) \times \Omega^0(M, \tilde V)
\]
expressed in terms of the Lagrangian densities as follows:
\[
L (A, \dot A, {\rm d}^AA, \varphi, \dot \varphi, {\rm d}^A\varphi)= \int_M \mathscr{L}_{\rm gau}\big( A(x), \dot A(x), {\rm d}^A A(x)\big)+ \mathscr{L}_{\rm mat}\big( \varphi(x), \dot \varphi(x), {\rm d}^A \varphi(x)\big).
\]

\medskip

In each case, the Lagrange--Dirac formulation and its associated variational principle enable the incorporation of both interior (volume) and boundary (surface) currents within a unified geometric framework. This leads to local and global energy balance laws, such as the Poynting theorem in Maxwell theory, its non-Abelian counterpart in the Yang–Mills equations, and further extensions that account for coupling with matter fields.
The key step in our geometric framework is identifying the restricted dual of the Fr\'echet configuration space as a suitable subspace of the topological dual. We provide two distinct realizations of this restricted dual, motivated by two commonly used forms of energy flux densities and boundary power densities (see, for example, Propositions \ref{prop:energybalancevectorbundle} and \ref{prop:energybalancegauge}). We refer the reader to \cite{AiSoZh2019,Ta1997,AvEs1999,CaCaFi2024,AsKoMaVa2024} for works on matter and gauge fields formulated on bounded spatial domains, and for discussion of the specific importance of such settings.
Besides these applications, the framework developed here is directly applicable to field theories in elasticity, such as geometrically exact beam and shell models, where boundary conditions play a significant role.

\subsection{Preliminaries on Lagrange--Dirac dynamical systems}

We briefly review here the construction of Lagrange--Dirac dynamical systems in the finite dimensional case. For a more detailed account, we refer to Part I, as well as to \cite{YoMa2006a,YoMa2006b} for the original references.  

Consider a system with configuration manifold $Q$, Lagrangian $L:TQ\rightarrow \mathbb{R}$, and external force $F$ given by a fiber-preserving map
\begin{equation}
\label{F}
F:TQ \rightarrow T^*Q.
\end{equation}
Given the canonical symplectic form $\Omega_{T^*Q}=d q^i\wedge dp_i$ on the phase space $T^*Q$, we construct the canonical Dirac structure as the graph of its associated flat map $\Omega^\flat_{T^*Q}:T(T^*Q)\rightarrow T^*(T^*Q)$, namely, $D_{T^*Q}=\operatorname{graph}(\Omega^\flat_{T^*Q}) \subset T(T^*Q)\oplus T^*(T^*Q)$. More explicitly, at each point $p_q\in T^*Q$, the canonical Dirac structure is given by
\begin{equation}\label{def_Dcan}
\begin{aligned}
D_{T^*Q}(p_q)&= \operatorname{graph}\left(\Omega_{T^*Q}^\flat (p_q)\right)\\
&=\left\{(v_{p_q},\alpha_{p_q})\in T_{p_q}(T^*Q)\times T^*_{p_q}(T^*Q)\;\left| \;
\alpha_{p_q}=\Omega_{T^*Q}^\flat(p_q)(v_{p_q})\right.
\right\}.
\end{aligned}
\end{equation}

From the Lagrangian $L:TQ\rightarrow \mathbb{R}$, one constructs the \emph{Dirac differential} ${\mathrm d}_DL=\gamma_{Q}\circ{\rm d}L:TQ\to T^*(T^*Q)$, where $\gamma_Q: T^*(TQ) \rightarrow T^*(T^*Q)$ is the canonical isomorphism given by $\gamma_Q(q,\delta q,\delta p,p)= (q,p,-\delta p,\delta q)$. The Dirac differential has the local expression
\[
{\mathrm d}_DL(q,v)= \left( q, \frac{\partial L}{\partial v}, - \frac{\partial L}{\partial q}, v \right),
\]
for all $(q,v)\in TQ$.

From the force $F:TQ \rightarrow T^*Q$ and the Lagrangian $L$, we construct the associated \textit{Lagrangian force field} as the map $\widetilde{F}: TQ \rightarrow T^*(T^*Q)$ defined by
\begin{equation}\label{Ftilde} 
\left\langle \widetilde{F}(q,v), W \right\rangle = \left\langle F( q,v), T_{ \mathbb{F} L(q,v)} \pi _Q(W) \right\rangle, 
\end{equation} 
for $(q,v) \in TQ$ and $W \in T_{ \mathbb{F} L(q,v)}(T^*Q)$. Here, $\mathbb{F} L:TQ \rightarrow T^*Q$ is the fiber derivative of $L$, locally given as $ \mathbb{F} L(q,v)=(q, \frac{\partial L}{\partial v}(q,v))$, and $\pi_Q:T^*Q\rightarrow Q$ is the projection map. In local coordinates, the Lagrangian force field reads
\[
\widetilde{F}(q,v) = \left( q, \frac{\partial L}{\partial v}(q,v), F(q,v),0 \right).
\]

Using ${\mathrm d}_DL$ and $\widetilde{F}$ defined above, we can formulate the 
\emph{forced Lagrange--Dirac system} as the following dynamical systems for curves $(q,v,p):[t_0,t_1]\rightarrow  TQ\oplus T^*Q$:
\begin{equation}\label{DiracSys_Mech}
\left((q,p,\dot{q},\dot{p}),{\mathrm d}_DL(q,v)- \widetilde{F}(q,v)\right)\in D_{T^*Q}(q,p).
\end{equation}
From the local expressions reviewed above and the definition of $D_{T^*Q}$, it is clear that \eqref{DiracSys_Mech} is equivalent to the system of equations:
\begin{equation}\label{LDS}
\left\{
\begin{array}{l}
\displaystyle\vspace{0.2cm}\dot q=v\\
\displaystyle\vspace{0.2cm} p=\frac{\partial L}{\partial v}(q,v)\\
\displaystyle\dot{p}=\frac{\partial L}{\partial q}(q,v)+ F(q, \dot  q).
\end{array}\right.
\end{equation}
The equations \eqref{LDS} for the curve $(q,v,p):[t_0,t_1]\rightarrow  TQ\oplus T^*Q$ in the Pontryagin bundle imply the forced Euler--Lagrange equations:
\begin{equation}\label{FEL}
\frac{d}{dt}\frac{\partial L}{\partial \dot q}(q, \dot q) - \frac{\partial L}{\partial q}(q, \dot q)=F(q,\dot q)
\end{equation}
for the curve $q:[t_0,t_1]\rightarrow Q$ in the configuration manifold. 

As we will demonstrate later, an infinite dimensional version of the equations \eqref{DiracSys_Mech}, \eqref{LDS}, and \eqref{FEL} can be developed that extends the canonical geometric structures to the continuum setting while simultaneously accommodating boundary energy flow (see, e.g., Remark \ref{energy_balance_rmk}). This is achieved thanks to an appropriate choice of the dual space, which impacts both the form of the cotangent and Pontryagin bundles, as well as the expression of the canonical symplectic form and canonical Dirac structures.
The infinite dimensional versions of \eqref{LDS} and \eqref{FEL} also incorporate a boundary condition, arising from the fact that the momentum variable
$p$ acquires both interior and boundary components (denoted $\alpha$ and $\alpha_\partial$), as will be discussed in details.

\paragraph{Associated variational structures} An important property of the Lagrange--Dirac system \eqref{DiracSys_Mech} is that its solution curves  $(q,v,p):[t_0,t_1]\rightarrow  TQ\oplus T^*Q$ are the critical trajectories for an action principle, the \textit{Lagrange--d'Alembert-Pontryagin principle}, which consistently extends Hamilton's principle to treat the velocity and momentum curves as independent variables. This action principle is given by:
\begin{equation}
\label{LDAP}
\delta\int_{t_0}^{t_1}\left(L(q,v)+\left<p, \dot q-v \right>\right)dt  + \int_{t_0}^{t_1} \left\langle F(q, \dot  q), \delta q \right\rangle dt=0,
\end{equation}
for arbitrary variations $\delta q$, $\delta v$, $\delta p$, with $\delta q$ vanishing at $t=t_0$ and $t=t_1$. The critical condition resulting from this principle is easily seen to be given by equations \eqref{LDS}. Alternatively, one can derive equation \eqref{FEL} directly by applying the \textit{Lagrange--d'Alembert principle}:
\begin{equation}
\label{LDA}
\delta\int_{t_0}^{t_1} L(q,\dot q)dt  + \int_{t_0}^{t_1} \left\langle F(q, \dot  q), \delta q \right\rangle dt=0,
\end{equation}
for arbitrary variations $\delta q$ vanishing at $t=t_0,t_1$.

\subsection{Overview of the results}

Here, we briefly summarize the infinite dimensional extensions that will be developed throughout the paper, offering an overview and illustrating the analogy with the finite dimensional case discussed above.

\paragraph{Field theories on bundle-valued forms} The first case we consider is the  configuration space $V=\Omega^k(M,E)$, consisting of $E$-valued $k$-forms, where $\pi_{E,M}:E\to M$ is a finite-dimensional vector bundle over $M$. We present two possible realizations of the restricted dual, denoted $V^\star$ and $V^\dagger$, whose most appropriate choice depends on the preferred form of the expression of the power density describing the exchange of energy with the exterior, in terms of the force and velocity variables, see Remark \ref{energy_balance_rmk}. For instance, the choice $V^{\dagger}=\Omega^{m-k}(M,E^*)\times\Omega^{m-k-1}(\partial M,E^*)$ corresponds to using the wedge product. In this context, the Lagrange--Dirac system \eqref{DiracSys_Mech} takes the form
\begin{equation}\label{LDS_intro}
\big((\varphi,\alpha,\alpha_\partial,\dot\varphi,\dot\alpha,\dot\alpha_\partial),{\rm d}_{ D}^{\dagger} L_{\nabla}(\varphi,\nu)-\widetilde{F}^{\dagger}(\varphi,\nu)\big)\in D_{T^{\dagger}V}(\varphi,\alpha,\alpha_\partial)
\end{equation}
for a curve $(\varphi,\nu, \alpha,\alpha_\partial): [t_0,t_1] \to TV \oplus T^{\dagger}V$, where $F^\dagger:TV \rightarrow T^\dagger V$ is a given force, analogous to equation \eqref{F}.

The infinite dimensional version of equations \eqref{FEL} is derived from \eqref{LDS_intro} by solving for the curve $\varphi:[t_0,t_1]\rightarrow V$. These equations are found to be:
\begin{equation}\label{equation1_intro}
\left\{\begin{array}{l}
\displaystyle\frac{\partial }{\partial t}\frac{\partialnew\mathscr L}{\partialnew\dot\varphi} =\frac{\partialnew\mathscr L}{\partialnew\varphi} -(-1)^k\nd^*  \frac{\partialnew\mathscr L}{\partialnew\zeta} +\mathcal F^{\dagger},\vspace{2mm}\\
\displaystyle\mathcal F^{\dagger}_\partial=-\iota_{\partial M}^* \frac{\partialnew\mathscr L}{\partialnew\zeta},
\end{array}\right.
\end{equation}
where $\mathscr{L}$ is the Lagrangian density (from equation \eqref{L_nabla_1}), $\zeta={\rm d}^\nabla \varphi$ is the covariant exterior derivative of $\varphi$, and $\mathcal{F}^\dagger,\mathcal{F}^\dagger_\partial \in V^\dagger$ are the interior and boundary components of $F^\dagger$.  Although the notations will be clarified later, equations \eqref{equation1_intro} show the occurrence of the interior and boundary forces in the evolution equations. These forces contribute to the energy balance as follows:
\[
\frac{d}{dt}\int_M \mathscr{E}= \underbrace{\int _M \dot\varphi \bwedge \mathcal{F}^\dagger}_{\text{spatially distributed contribution}} + \underbrace{\int_{\partial M}\iota_{\partial M}^*\dot\varphi \bwedge \mathcal{F}^\dagger_\partial}_{\text{boundary contribution}}.
\]

\paragraph{Gauge field theories} The second case we consider involves the configuration manifold $\mathcal{C}(P)$ of principal connections of a principal $G$-bundle $\pi_{P,M}:P\rightarrow M$. In this context, with the choice of the restricted dual $\mathcal{C}(P)^{\dagger}=\Omega^{m-1}(M,\tilde{\mathfrak{g}}^*)\times\Omega^{m-2}(\partial M,\tilde{\mathfrak{g}}^*)$, corresponding again to using the wedge product, the Lagrange--Dirac system \eqref{DiracSys_Mech} takes the form:
\begin{equation}\label{LDS_connection_intro}
\big((A,\varsigma,\varsigma_\partial,\dot A,\dot\varsigma,\dot\varsigma_\partial),{\rm d}_{ D}^{\dagger} L(A,\varepsilon)-\widetilde{F}^{\dagger}(A,\varepsilon)\big)\in D_{T^{\dagger}\mathcal{C}(P)}(A,\varsigma,\varsigma_\partial),
\end{equation}
for a curve $(A,\varepsilon, \varsigma,\varsigma_\partial): [t_0,t_1] \to T\mathcal{C}(P) \oplus T^{\dagger}\mathcal{C}(P)$, where $F^\dagger:T\mathcal{C}(P) \rightarrow T^\dagger \mathcal{C}(P)$ is a given force, see \eqref{F}.

When the Yang--Mills Lagrangian is used, the Lagrange--Dirac dynamical system \eqref{LDS_connection_intro} yields the equations:
\begin{equation}\label{YM_intro}
\left\{\begin{array}{l}
\displaystyle\dot E-\delta^A B_A=-J,\vspace{2mm}\\
\boldsymbol\star_\partial\left(\iota_{\partial M}^*(\boldsymbol\star B_A)\right)= j,
\end{array}\right.
\end{equation}
where $\delta^A$ is the codifferential of ${\rm d}^A$, $E=-\dot A$ and $B_A={\rm d}^AA$ represent the non-Abelian electric and magnetic fields. From these, the following relations are derived:
\begin{equation}\label{Bianchi}
{\rm d}^AB_A=0, \qquad \dot B_A=- {\rm d}^A E.
\end{equation}
Equations \eqref{YM_intro} and \eqref{Bianchi}, along with the identification $\rho=\delta^A E$ as the non-Abelian charge density, yield the Yang--Mills equations with current $J$ and surface current $j$, in the space+time decomposition, with respect to the temporal gauge. The currents $J$ and $j$ arise from the interior and boundary components of the force $F^\dagger$ as we will explain in details. In particular, a non-Abelian Poynting theorem, including spatially distributed and boundary sources, will be derived.

\paragraph{Interaction of gauge fields and matter} To describe the interaction between gauge fields and matter, we consider the configuration manifold $Q=\mathcal C(P)\times\Omega^0(M,\tilde{\mathfrak g})$, for which the Lagrange--Dirac dynamical system \eqref{DiracSys_Mech} takes the form
\begin{equation}\label{YMH_intro}
\big((q,p ,\dot q, \dot p),{\rm d}_{ D}^{\dagger} L(q,v)-\widetilde{F}^{\dagger}(q,v)\big)\in D_{T^{\dagger}V}(q,p)
\end{equation}
with $q=(A,\varphi)$, $v=(\nu, \epsilon)$, and $p=(\varsigma,\varsigma_\partial, \alpha, \alpha_\partial )$. A typical example is the Yang--Mills--Higgs Lagrangian, in which case \eqref{YMH_intro} reduces to
\begin{equation*}
\left\{\begin{array}{l}
\displaystyle\dot E-\delta^A B_A+\tilde\varrho_\varphi^*\left(({\rm d}^A\varphi)^{\flat_\kappa}\right)^{\sharp_K}=-J,\vspace{0.2cm}\\
\displaystyle\boldsymbol\star_\partial\left(\iota_{\partial M}^*\left(\boldsymbol\star B_A\right)\right)=j,\vspace{2mm}\\
\end{array}\right.\qquad\left\{\begin{array}{l}
\displaystyle\ddot\varphi+\delta^A\left({\rm d}^A\varphi\right)+\operatorname{grad}_{\kappa}\mathbf V=\beth,\vspace{0.2cm}\\
\displaystyle\boldsymbol\star_\partial\left(\iota_{\partial M}^*\left(\boldsymbol\star{\rm d}^A\varphi\right)\right)=\gimel.
\end{array}\right.
\end{equation*}
to be detailed later.
Together with the equations ${\rm d}^AB_A=0$, $\dot B_A=- {\rm d}^AE$, and $\rho=\delta^A E$, these are the Yang--Mills--Higgs equations with currents $J$, $\beth$ and surface currents $j$, $\gimel$, in the space+time decomposition, with respect to the temporal gauge. These currents arise as the components of the force $F\in T^\dagger Q$ used in \eqref{YMH_intro}.

\subsection{Plan of the paper}

The contents of this paper are organized as follows. Section \ref{sec:vectorbundlekforms} is devoted to develop the Lagrange--Dirac theory for systems on the infinite dimensional Fr\'echet space of vector bundle-valued forms on a manifold with boundary. In particular, two choices of the restricted dual are introduced and the equivalence between them is proven. After that, the restricted Tulczyjew triple is defined in this infinite dimensional setting and, after some technical results, the forced Lagrange--Dirac equations are presented. The main feature is that they naturally account for the boundary energy flow. Next, the local and global energy balances are computed for both choices of the restricted dual and, then, the variational approach is presented, leading to the dynamical equations as a critical curve condition. Lastly, the theoretical results are utilized to model matter fields, such as the Higgs field or the Klein--Gordon field in the space+time decomposition. In section \ref{sec:yangmills}, the previous theory is modified to describe (non-Abelian) gauge fields. The main difference is that the configuration space of the system is the family of principal connections on a given principal bundle, which is no longer a vector space, but an infinite-dimensional affine space. The canonical trivialization of its tangent bundle allows for developing the theory analogous to the linear case with the appropriate adjustments. As a particular instance, the Yang--Mills equations are considered. To conclude, the theories for linear and affine systems are gathered together to describe gauge fields interacting with matter, thus obtaining the Yang--Mills--Higgs equations as a particular case.

\paragraph{Notations} In the following, every manifold or map is assumed to be smooth. Given an $m$-dimensional manifold with boundary, $M$, its boundary is denoted by $\partial M$ and the natural inclusion is denoted by $\iota_{\partial M}:\partial M\to M$. Let $\pi_{E,M}:E\to M$ be a vector bundle and $\pi_{E^*,M}:E^*\to M$ be its dual bundle. The space of sections of $\pi_{E,M}:E\to M$ is denoted by $\Gamma(E)$. Given $\xi\in\Gamma(E)$ and $\eta\in\Gamma(E^*)$, their contraction is denoted by $\eta\cdot\xi\in C^\infty(M)=\Omega^0(M)$. Lastly, given another vector bundle $\pi_{E',M}:E'\to M$, the tensor product between $\pi_{E,M}:E\to M$ and $\pi_{E',M}:E'\to M$ is denoted by $\pi_{E\otimes E',M}:E\otimes E'\to M$.

\section{Lagrange--Dirac dynamical systems on the family of vector bundle-valued forms}\label{sec:vectorbundlekforms}

The aim of this section is to develop the theory of infinite-dimensional Lagrange--Dirac dynamical systems on the space of vector bundle-valued $k$-forms on a compact manifold $M$ with boundary. We begin in \S\ref{Alternative_choice} by describing two versions of restricted dual spaces, denoted $V^\star$ and $V^\dagger$, which are then used to define the restricted cotangent and Pontryagin bundles. In \S\ref{canonical}, based on the choice of dual, we define the canonical symplectic forms and associated Dirac structures, which include boundary terms that will be crucial for describing systems with energy flow through the boundary. After presenting a divergence theorem for bundle-valued forms in \S\ref{sec:divergenceVectValued_kforms}, we show in \S\ref{sec:partialderivativesvectorkforms} that the Dirac differential of Lagrangian functions defined via densities takes values in the restricted dual space. This is fundamental for defining the associated infinite dimensional Lagrange--Dirac dynamical systems, in \S\ref{sec:LDVectBundleValued_kforms}, where we compute the evolution equations, boundary conditions, and energy balance. 
In \S\ref{VP}, we provide the variational principles governing these equations and boundary conditions, analogous to \eqref{LDAP} and \eqref{LDA}. Finally, we illustrate these developments with matter fields for $k=0$ in \S\ref{sec:particlefields}.

\medskip
Consider a finite-dimensional vector bundle $\pi_{E,M}:E\to M$. We define
\begin{equation*}
V=\Omega^k(M,E)=\Gamma\big(\textstyle\bigwedge^k T^*M\otimes E\big),\qquad 0\leq k\leq m,
\end{equation*}
as the Fr\'echet space of $E$-valued $k$-forms on $M$. The space of $(r,s)$-skew-symmetric tensor fields on $M$ taking values in the dual bundle $E^*$ is denoted by
\begin{equation*}
\Lambda_r^s(M,E^*)=\Gamma\left(\textstyle\bigwedge^s TM\otimes\bigwedge^r T^*M\otimes E^*\right),\qquad 0\leq r,s\leq m.
\end{equation*}
Similarly, we write $\Lambda_r^s(M)=\Gamma\left(\textstyle\bigwedge^s TM\otimes\bigwedge^r T^*M\right)$ for the space of $(r,s)$-skew- symmetric tensor fields on $M$. 

Throughout this paper, we use bold symbols to indicate the contraction and the wedge product in both sets of indices. Namely,
\begin{align*}
& \chi\bcdot\varphi=\left(\tilde\chi\cdot\tilde\varphi\right)(\eta\cdot\xi),\quad  & &\varphi=\tilde\varphi\otimes\xi\in\Omega^k(M,E),~\chi=\tilde\chi\otimes\eta\in\Lambda_m^k(M,E^*),\\
& \chi\bwedge\varphi=\left(\tilde\chi\wedge\tilde\varphi\right)(\eta\cdot\xi), & &\varphi=\tilde\varphi\otimes\xi\in\Omega^k(M,E),~\chi=\tilde\chi\otimes\eta\in\Omega^{m-k}(M,E^*).
\end{align*}
We stress that $\bwedge$ is a bilinear map taking values in real- (not bundle)-valued forms, i.e.,
\[
\bwedge: \Omega ^k(M, E) \times \Omega ^{m-k}(M, E^*) \rightarrow \Omega ^m(M).
\]

\begin{remark}[Tensors on the boundary]\rm
When dealing with vector bundle-valued tensors on the boundary, we need to consider the pullback bundle: $\pi_{\iota_{\partial M}^* E^*,\partial M}:\iota_{\partial M}^* E^*\to\partial M$, where $\iota_{\partial M}:\partial M\to M$ is the inclusion. For brevity, we will omit the pullback notation, simply writing $E^*\equiv\iota_{\partial M}^*E^*$.
\end{remark}

\subsection{Alternative choices for the restricted dual}\label{Alternative_choice}

We now introduce the restricted dual of the Fréchet space 
$V=\Omega^k(M,E)$ as a subspace of its topological dual, $V'$, defined through a duality pairing via integration over the manifold $M$ and its boundary $\partial M$. There are two choices for this \emph{restricted dual} of $V$; namely,
\begin{enumerate}
    \item $V^\star=\Lambda_m^k(M,E^*)\times\Lambda_{m-1}^k(\partial M,E^*)$ with the pairing given by
    \begin{equation*}
    \blangle(\alpha,\alpha_\partial),\varphi\brangle_\star=\int_M\alpha\bcdot\varphi+\int_{\partial M}\alpha_\partial\bcdot\iota_{\partial M}^*\varphi,\qquad(\alpha,\alpha_\partial)\in V^\star,~\varphi\in V.
    \end{equation*}
    \item $V^\dagger=\Omega^{m-k}(M,E^*)\times\Omega^{m-k-1}(\partial M,E^*)$ with the duality pairing given by
    \begin{equation*}
    \blangle(\alpha,\alpha_\partial),\varphi\brangle_\dagger=\int_M\varphi\bwedge\alpha+\int_{\partial M}\iota_{\partial M}^*\varphi\bwedge\alpha_\partial,\qquad(\alpha,\alpha_\partial)\in V^\dagger,~\varphi\in V.
    \end{equation*}
\end{enumerate}
These two choices yield two different, yet equivalent, formulations of the equations of motion and boundary conditions. Depending on the structure of the Lagrangian, one formulation may be more convenient or advantageous to use.
Observe that both choices agree for $k=0$, i.e., when $V=\Omega^0(M,E)=\Gamma(\pi_{E,M})$.

These definitions give continuous injections of $V^\star$ and $V^\dagger$ into the topological dual:
\begin{equation}\label{Psi_star_dagger}
\Psi _\star:V^\star \rightarrow V' \quad\text{and}\quad \Psi _\dagger: V^\dagger \rightarrow V',
\end{equation} 
which are defined as
\begin{equation*}
\Psi_\star(\alpha,\alpha_\partial)(\varphi)=\langle(\alpha,\alpha_\partial),\varphi\rangle_\star\quad\text{and}\quad\Psi_\dagger(\beta,\beta_\partial)(\varphi)=\langle(\beta,\beta_\partial),\varphi\rangle_\dagger,
\end{equation*}
respectively, for each $\varphi\in V$, $(\alpha,\alpha_\partial)\in V^\star$, and $(\beta,\beta_\partial)\in V^\dagger$. Below we shall exhibit a canonical isomorphism between the two choices $V^\star$ and $V^\dagger$ of the restricted dual.

Given $x\in M$ and $U=U_1\wedge\dots\wedge U_k\in\bigwedge^k T_xM$, we define the left interior multiplication as
$$i_{U}:\textstyle\bigwedge^m T_x^*M\to\bigwedge^{m-k}T_x^*M,\quad \mu\mapsto i_U\mu,$$
where $\left(i_U\mu\right)(u_1,...,u_{m-k})=\mu(U_1,\dots,U_k,u_1,...,u_{m-k})$ for each $u_1,\dots,u_{m-k}\in T_xM$. Similarly, for each $\alpha=\alpha_1\wedge\dots\wedge\alpha_k\in\bigwedge^k T_x^*M$, the left interior multiplication is defined as
$$i_\alpha:\textstyle\bigwedge^m T_xM\to\bigwedge^{m-k}T_xM,\quad U\mapsto i_\alpha U,$$
where $(i_\alpha U)(\beta_1,\dots,\beta_{m-k})=U(\alpha_1,\dots,\alpha_k,\beta_1,\dots,\beta_{m-k})$ for each $\beta_1,\dots,\beta_{m-k}\in T_x^*M$.

\begin{lemma}\label{lemma:invPhiE}\rm
There exists a canonical vector bundle isomorphism over the identity, $\operatorname{id}_M$, given by
\begin{equation}\label{eq:PhiE}
\Phi_E:\textstyle\bigwedge^k TM\otimes\bigwedge^m T^*M\otimes E^*\to\bigwedge^{m-k}T^*M\otimes E^*,\quad U\otimes\mu\otimes\eta\mapsto i_U\mu\otimes\eta.
\end{equation}
Furthermore, with the aid of a Riemannian metric $g$ on $M$, the inverse of $\Phi_E$ may be expressed as
\begin{equation*}
\Phi_E^{-1}:\textstyle\bigwedge^{m-k}T^*M\otimes E^*\to\bigwedge^k TM\otimes\bigwedge^m T^*M\otimes E^*,\quad\alpha\otimes\eta\mapsto(-1)^{k(m-k)}(\star\alpha)^\sharp\otimes\mu_g\otimes\eta,
\end{equation*}
where $\sharp:T^*M\to TM$ and $\flat:TM\to T^*M$, respectively, denote the sharp and flat isomorphisms defined by the Riemannian metric, $\mu_g\in\Omega^m(M)$ is the Riemannian volume form and $\star:\Omega^k(M)\to\Omega^{m-k}(M)$ is the Hodge star operator, which is defined by the condition $g( \alpha , \beta ) \mu_g = \alpha\wedge\star\beta$ for each $\alpha,\beta\in\Omega^k(M)$.
\end{lemma}

\begin{proof}
The statement of \cite[Chapter 2, Ex. 14]{Wa1983}, as well as the fact that $\star\mu_g=(-1)^{m-1}$, yield $i_U\mu_g=(-1)^{(m-1)}\star U^\flat$ for $U\in TM$. This extends to multivectors straightforwardly, i.e., $i_U\mu_g=(-1)^{m-k}\star U^\flat$ for $U\in\bigwedge^k TM$. Hence, we conclude:
\begin{enumerate}[(i)]
    \item $\left(\Phi_E^{-1}\circ\Phi_E\right)(U\otimes\mu_g\otimes\eta)=(-1)^{k(m-k)}(\star(i_U\mu_g))^\sharp\otimes\mu_g\otimes\eta=(-1)^k(\star(\star U^\flat))^\sharp\otimes\mu_g\otimes\eta=( U^\flat)^\sharp\otimes\mu_g\otimes\eta=U\otimes\mu_g\otimes\eta$, where we have used that $\star\star=(-1)^{k(m-k)}$.
    \item $\left(\Phi_E\circ\Phi_E^{-1}\right)(\alpha\otimes\eta)=(-1)^{k(m-k)}i_{(\star\alpha)^\sharp}\mu_g\otimes\eta=\alpha\otimes\eta$, where we have used that, for each $\beta\in\bigwedge^{m-k}T^*M$, we have
    \begin{align*}
    g(i_{(\star\alpha)^\sharp}\mu_g,\beta) & =g(\mu_g,\star\alpha\wedge\beta)=(-1)^{k(m-k)}g(\mu_g,\alpha\wedge\star\beta)=g(\mu_g,\star\beta\wedge\alpha)\\
    & =g(i_{(\star\beta)^\sharp}\mu_g,\alpha)=(-1)^{m-k}g(\star\star\beta,\alpha)=(-1)^{k(m-k)}g(\alpha,\beta).
    \end{align*}
\end{enumerate}
\end{proof}

\begin{remark}[Multi-index notation]\label{remark:multindex}\rm
Let $(x^i)$, $1\leq i\leq m$, be local coordinates on $M$. For simplicity, we denote multi-indices by $I=(i_1,\dots,i_k)\in\mathbb N^k$, where $1\leq i_1<\dots<i_k\leq m$. The \emph{length} of $I$ is denoted by $|I|=k$. For $1\leq r\leq k$, we set $I_r=(i_1,\dots,i_{r-1},i_{r+1},\dots,i_k)$. Observe that $|I_r|=|I|-1$. We will denote $\partial_I=\partial_{i_1}\wedge\dots\wedge\partial_{i_k}$, $dx^I=dx^{i_1}\wedge\dots\wedge dx^{i_k}$, $d^m x=dx^1\wedge\dots\wedge dx^m$, $\boldsymbol\partial_m=\partial_1\wedge\dots\wedge\partial_m$ and $d_I^{m-k}x={i}_{\partial_I} d^m x$. As usual, the Einstein summation convention will be utilized for repeated indices and multi-indices.
\end{remark}

\begin{remark}\rm Of course the map $ \Phi_E ^{-1} $ does not depend on the choice of a metric, which is only a convenient way to write this map explicitly. In local coordinates, it may be written without using a metric; namely,
\begin{equation*}
\Phi_E^{-1}:\textstyle\bigwedge^{m-k}T^*M\otimes E^*\to\bigwedge^k TM\otimes\bigwedge^m T^*M\otimes E^*,\quad\alpha\otimes\eta\mapsto(-1)^{k(m-k)}(i_\alpha\boldsymbol\partial_m)\otimes d^m x\otimes\eta,
\end{equation*}
where we have used the notations introduced in the previous remark.
\end{remark}

An analogous result to the previous lemma holds on $\partial M$ by using the Riemannian metric induced on the boundary, $g_\partial=\iota_{\partial M}^*g$. More specifically, we have a canonical isomorphism given by
\begin{equation*}
\Phi_{E,\partial}:\textstyle\bigwedge^kT\partial M\otimes\bigwedge^{m-1}T^*\partial M\otimes E^*\to\bigwedge^{m-k-1}T^*\partial M\otimes E^*,\quad U_\partial\otimes\mu_\partial\otimes\eta\mapsto i_{U_\partial}\mu_\partial\otimes\eta,
\end{equation*}
whose inverse,
\begin{equation*}
\Phi_{E,\partial}^{-1}:\textstyle\bigwedge^{m-k-1}T^*\partial M\otimes E^*\to\bigwedge^k T\partial M\otimes\bigwedge^{m-1}T^*\partial M\otimes E^*,
\end{equation*}
may be written as
\begin{equation*}
\alpha_\partial\otimes\eta\mapsto(-1)^{k(m-k-1)}(\star_\partial\,\alpha_\partial)^{\sharp_\partial}\otimes\mu_g^\partial\otimes\eta,
\end{equation*}
where $\sharp_\partial:T^*\partial M\to T\partial M$ and $\flat_\partial:T\partial M\to T^*\partial M$, respectively, denote the sharp and flat isomorphisms defined by $g_\partial$, and $\star_\partial:\Omega^k(\partial M)\to\Omega^{m-k-1}(\partial M)$ is the Hodge star operator on the boundary.
Similarly, $\mu_g^\partial=\iota_{\partial M}^*(i_n\mu_g)\in\Omega^{m-1}( \partial M)$ is the Riemannian volume form on the boundary (cf. \cite[Corollary 15.34]{Le2012}), where $n\in\mathfrak X(M)|_{\partial M}$ is the outward pointing, unit, normal vector field on $\partial M$.

From Lemma \ref{lemma:invPhiE} and the previous considerations, we get the following result.

\begin{proposition}\label{isomorphic_duals}\rm
Both choices of the restricted dual are equivalent via the following isomorphisms:
\begin{equation}\label{isomorphism_metric}
\begin{split}
    &V^\star\to V^\dagger,\quad(\alpha,\alpha_\partial)\mapsto\left(\Phi_E(\alpha),\Phi_{E,\partial}(\alpha_\partial)\right),\\
    &V^\dagger \to V^\star, \quad (\beta, 
    \beta_\partial) \mapsto \left(\Phi^{-1}_E(\beta),\Phi^{-1}_{E,\partial}(\beta_\partial)\right).
\end{split}
\end{equation}
Moreover, in terms of \eqref{Psi_star_dagger}, we have $\Psi_\dagger \circ\left(\Phi_E,\Phi_{E,\partial}\right) = \Psi_\star$ and $\Psi_\star \circ\left(\Phi^{-1}_E,\Phi^{-1}_{E,\partial}\right) = \Psi_\dagger$.
\end{proposition}

Of course, both pairings are weakly non-degenerate and, thus, the restricted cotangent bundles and the restricted iterated bundles may be defined as follows. The \emph{restricted cotangent bundles} are given by
\begin{equation*}
\begin{array}{l}
T^\star V=V\times V^\star=\Omega^k(M,E)\times\Lambda_m^k(M,E^*)\times\Lambda_{m-1}^k(\partial M,E^*),\vspace{0.1cm}\\
T^\dagger V=V\times V^\dagger=\Omega^k(M,E)\times\Omega^{m-k}(M,E^*)\times\Omega^{m-k-1}(\partial M,E^*),
\end{array}
\end{equation*}
and both of them are subbundles of $T'V$ through the identification given by the corresponding pairing. 
Similarly, the \emph{restricted iterated bundles} are given by
\begin{equation*}
\begin{array}{ll}
T^\star(TV)=V\times V\times V^\star\times V^\star,\qquad & T^\dagger(TV)=V\times V\times V^\dagger\times V^\dagger,\vspace{0.1cm}\\
T(T^\star V)=V\times V^\star\times V\times V^\star, & T\left(T^\dagger V\right)=V\times V^\dagger\times V\times V^\dagger,\vspace{0.1cm}\\
T^\star(T^\star V)=V\times V^\star\times V^\star\times V, & T^\dagger\left(T^\dagger V\right)=V\times V^\dagger\times V^\dagger\times V.
\end{array}
\end{equation*}
To conclude, the \emph{restricted Pontryagin bundles} are given by
\begin{equation*}
\begin{array}{l}
T(T^\star V)\oplus T^\star(T^\star V)=V\times V^\star\times(V\times V^\star\times V^\star\times V)\;\left(\subset T(T^\star V)\oplus T'(T^\star V)\right),\vspace{0.1cm}\\
T\left(T^\dagger V\right)\oplus T^\dagger\left(T^\dagger V\right)=V\times V^\dagger\times(V\times V^\dagger\times V^\dagger\times V)\;\left(\subset T\left(T^\dagger V\right)\oplus T'\left(T^\dagger V\right)\right).
\end{array}
\end{equation*}

\subsection{Canonical forms, Tulczyjew triples and canonical Dirac structures}
\label{canonical}

There exist two manners of introducing the canonical symplectic form and the Tulczyjew triple corresponding to the possible choices of the restricted dual, from which we get two distinct canonical Dirac structures. Thanks to our choice of restricted dual spaces, the canonical symplectic forms and associated Dirac structures include boundary terms that will play a fundamental role in describing systems with energy flow through the boundary.

\begin{definition}\rm\label{DefCanForms_VectBundValued}
Associated with the choice of the restricted dual space, we define the canonical one- and two-forms as follows:
\begin{itemize}
\item[(i)] For $V^\star=\Lambda_m^k(M,E^*)\times\Lambda_{m-1}^k(\partial M,E^*)$, the \emph{canonical one-form}, $\Theta_{T^\star V}\in\Omega^1(T^\star\Omega^k(M,E))$, is defined as
\[
\Theta_{T^\star V}(z) \cdot \delta{z}=\langle z, T_{z}\pi^\star_{V}(\delta z)\rangle_\star, \qquad z \in T^\star V,~\delta z\in T_{z}(T^\star V),
\]
where $\pi^\star_{V}: T^\star V \to V$ is the natural projection. Furthermore, the \emph{canonical two-form} on $T^\star V$ is defined as $\Omega_{T^\star V}=-{\rm d}\Theta_{T^\star V}\in\Omega^2(T^\star\Omega^k(M,E))$.
\vskip 3pt
\item[(ii)] For $V^\dagger=\Omega ^{m-k}(M,E^*) \times \Omega ^{m-k-1}( \partial M,E^*)$, the \emph{canonical one-form}, $\Theta_{T^\dagger V}\in\Omega^1(T^\dagger\Omega^k(M,E))$, is defined as
\[
\Theta_{T^\dagger V}(z) \cdot \delta{z}=\langle z, T_{z}\pi^\dagger_{V}(\delta z)\rangle_\dagger,\qquad z \in T^\dagger V,~\delta z\in T_{z}\big(T^\dagger V\big),
\]
where $\pi^\dagger_{V}: T^\dagger V \to V$ is the natural projection. Furthermore, the \emph{canonical two-form} on $T^\dagger V$ is defined as $\Omega_{T^\dagger V}=-{\rm d}\Theta_{T^\dagger V}\in\Omega^2\big(T^\dagger\Omega^k(M,E)\big)$. 
\end{itemize}
\end{definition}

Note that $ \Theta _{T^\star V}$ is a smooth one-form on the Fr\'echet space $T^\star V$ and, thus, the exterior derivative can be computed in the usual sense. The same holds for $\Theta_{T^\dagger V}$. The explicit expressions of the canonical one-forms are
\begin{equation*}
\Theta_{T^\star V}(\varphi,\alpha,\alpha_\partial)\cdot(\delta\varphi,\delta\alpha,\delta\alpha_{\partial})=\blangle (\varphi, \alpha,\alpha_\partial),\delta\varphi\brangle_\star=\int_{ M}\alpha\bcdot\delta\varphi+\int_{\partial M}\alpha_\partial\bcdot\iota_{\partial M}^*\delta\varphi,
\end{equation*}
for each $(\varphi,\alpha,\alpha_\partial)\in T^\star V$ and $(\delta\varphi,\delta\alpha,\delta\alpha_\partial)\in T_{(\varphi,\alpha,\alpha_\partial)}(T^\star V)$, and
\begin{equation*}
\Theta_{T^\dagger V}(\varphi,\alpha,\alpha_\partial)\cdot(\delta\varphi,\delta\alpha,\delta\alpha_{\partial})=\blangle (\varphi, \alpha,\alpha_\partial),\delta\varphi\brangle_\dagger=\int_{ M}\delta\varphi\bwedge\alpha+\int_{\partial M}\iota_{\partial M}^*\delta\varphi\bwedge\alpha_\partial,
\end{equation*}
for each $(\varphi,\alpha,\alpha_\partial)\in T^\dagger V$ and $(\delta\varphi,\delta\alpha,\delta\alpha_\partial)\in T_{(\varphi,\alpha,\alpha_\partial)}\left(T^\dagger V\right)$.
In the same vein, the two-forms are explicitly given by
\begin{equation*}
\Omega_{T^\star V}(\varphi,\alpha,\alpha_\partial)((\dot\varphi,\dot\alpha,\dot\alpha_\partial),(\delta\varphi,\delta\alpha,\delta\alpha_\partial))=\blangle(\delta\alpha,\delta\alpha_\partial),\dot\varphi\brangle_\star-\blangle(\dot\alpha,\dot\alpha_\partial),\delta\varphi\brangle_\star,
\end{equation*}
for each $(\varphi,\alpha,\alpha_\partial)\in T^\star V$ and $(\delta\varphi,\delta\alpha,\delta\alpha_\partial), (\dot\varphi,\dot\alpha,\dot\alpha_\partial) \in T_{(\varphi,\alpha,\alpha_\partial)}(T^\star V)\simeq V \times V^{\star}$, and 
\begin{equation*}
\Omega_{T^\dagger V}(\varphi,\alpha,\alpha_\partial)((\dot\varphi,\dot\alpha,\dot\alpha_\partial),(\delta\varphi,\delta\alpha,\delta\alpha_\partial))=\blangle(\delta\alpha,\delta\alpha_\partial),\dot\varphi\brangle_\dagger-\blangle(\dot\alpha,\dot\alpha_\partial),\delta\varphi\brangle_\dagger,
\end{equation*}
for each $(\varphi,\alpha,\alpha_\partial)\in T^\dagger V$ and $(\delta\varphi,\delta\alpha,\delta\alpha_\partial),(\dot\varphi,\dot\alpha,\dot\alpha_\partial)\in T_{(\varphi,\alpha,\alpha_\partial)}\left(T^\dagger V\right)\simeq V\times V^\dagger$. From these expressions, the associated flat maps can be readily derived and are found to be isomorphisms when restricted dual spaces are used, as stated below.

\begin{proposition}\rm
The flat map of the canonical symplectic form defines a vector bundle morphism over the identity, namely:
\begin{itemize}
\item[(i)] Under the identification given by the pairing $\langle\cdot,\cdot\rangle_\star$, it reads
\begin{equation*}
\Omega_{T^\star V}^\flat: T(T^\star V)\to T^\star(T^\star V),\quad(\varphi,\alpha,\alpha_\partial,\dot\varphi,\dot\alpha,\dot\alpha_\partial)\mapsto(\varphi,\alpha,\alpha_\partial, -\dot\alpha,-\dot\alpha_\partial,\dot\varphi).
\end{equation*}
\item[(ii)] Under the identification given by the pairing $\langle\cdot,\cdot\rangle_\dagger$, it reads
\begin{equation*}
\Omega_{T^\dagger V}^\flat: T\big(T^\dagger V\big)\to T^\dagger\big(T^\dagger V\big),\quad(\varphi,\alpha,\alpha_\partial,\dot\varphi,\dot\alpha,\dot\alpha_\partial)\mapsto(\varphi,\alpha,\alpha_\partial, -\dot\alpha,-\dot\alpha_\partial,\dot\varphi).
\end{equation*}
\end{itemize}
\end{proposition}

\begin{definition}\label{def:OmegaflatVectBundleValued_kforms}\rm
Associated to each choice of the restricted dual, there exists a structure of three isomorphisms between the restricted iterated bundles, the so-called \emph{Tulczyjew triple} on the space of $E$-valued $k$-forms; namely:
\begin{itemize}
\item[(i)] For $V^\star=\Lambda_m^k(M,E^*)\times\Lambda_{m-1}^k(\partial M,E^*)$, the corresponding Tulczyjew triple reads
\begin{equation*}
	\begin{tikzpicture}
			\matrix (m) [matrix of math nodes,row sep=0.1em,column sep=6em,minimum width=2em]
			{	T^\star(T\Omega^k(M,E)) & T(T^\star\Omega^k(M,E)) & T^\star(T^\star\Omega^k(M,E))\\
			\left(\varphi,\dot\varphi, \dot\alpha,\dot\alpha_\partial,\alpha,\alpha_\partial\right) & \left(\varphi,\alpha,\alpha_\partial, \dot\varphi,\dot\alpha,\dot\alpha_\partial\right) & \left(\varphi,\alpha,\alpha_\partial, -\dot\alpha,-\dot\alpha_\partial,\dot\varphi\right).\\};
			\path[-stealth]
			(m-1-1) edge [bend left = 25] node [above] {$\gamma_{T^\star V}=\Omega_{T^{\star} V}^\flat\circ\kappa_{T^{\star} V}^{-1}$} (m-1-3)
			(m-1-2) edge [] node [above] {$\kappa_{T^\star V}$} (m-1-1)
			(m-1-2) edge [] node [above] {$\Omega_{T^\star V}^\flat$} (m-1-3)
			(m-1-3)
			(m-2-2) edge [|->] node [] {} (m-2-3)
			(m-2-2) edge [|->] node [] {} (m-2-1);
	\end{tikzpicture}
\end{equation*}
\item[(ii)] For $V^\dagger=\Omega ^{m-k}(M,E^*) \times \Omega ^{m-k-1}( \partial M,E^*)$, the situation is similar.
\end{itemize}
\end{definition}

Thanks to this setting, we can now introduce the canonical Dirac structures in this infinite dimensional context by following the same approach as in the finite dimensional case.

\begin{definition}\rm\label{def:diracstructurebundle}
Associated to the choice of the restricted dual space, we define the \emph{canonical Dirac structure} as follows:
\begin{itemize}
\item[(i)] The canonical Dirac structure on $T^\star V= \Omega^k(M,E)\times\Lambda_m^k(M,E^*)\times\Lambda_{m-1}^k(\partial M,E^*)$ is the subbundle $D_{T^\star V}=\operatorname{graph}\Omega_{T^\star V}^\flat$. For each $(\varphi,\alpha,\alpha_\partial)\in T^\star V$, it reads
\begin{equation*}
\begin{split}
D_{T^\star V}(\varphi,\alpha,\alpha_\partial)&=\big\{(\dot\varphi,\dot\alpha,\dot\alpha_\partial, \delta\alpha,\delta\alpha_\partial,\delta\varphi)\in T_{(\varphi,\alpha,\alpha_\partial)}(T^\star V)\times T_{(\varphi,\alpha,\alpha_\partial)}^\star(T^\star V)\mid\\
&\hspace{5cm}-\dot\alpha=\delta\alpha,~-\dot\alpha_\partial=\delta\alpha_\partial,~\dot\varphi=\delta\varphi\big\}.
\end{split}
\end{equation*}
\item[(ii)] The canonical Dirac structure on $T^\dagger V = \Omega^k(M,E)\times\Omega^{m-k}(M,E^*)\times\Omega^{m-k-1}(\partial M,E^*)$ is the subbundle $D_{T^\dagger V}=\operatorname{graph}\Omega_{T^\dagger\Omega^k(M,E)}^\flat$. For each $(\varphi,\alpha,\alpha_\partial)\in T^\dagger V$, it reads
\begin{equation*}
\begin{split}
D_{T^\dagger V}(\varphi,\alpha,\alpha_\partial)&=\big\{(\dot\varphi,\dot\alpha,\dot\alpha_\partial,\,\delta\alpha,\delta\alpha_\partial,\delta\varphi)\in T_{(\varphi,\alpha,\alpha_\partial)}\big(T^\dagger V\big)\times T_{(\varphi,\alpha,\alpha_\partial)}^\dagger\big(T^\dagger V\big)\mid\\
&\hspace{5cm}-\dot\alpha=\delta\alpha,~-\dot\alpha_\partial=\delta\alpha_\partial,~\dot\varphi=\delta\varphi\big\}.
\end{split}
\end{equation*}
\end{itemize}
\end{definition}

\subsection{Divergence theorem for skew-symmetric vector bundle-valued tensor fields}\label{sec:divergenceVectValued_kforms}

Here we introduce some technical results that will be useful in the forthcoming sections to compute the partial derivatives of the Lagrangian and to derive the equations of motion and boundary conditions in intrinsic form.
In particular, given a linear connection $\nabla$ on $E$ and its dual connection $\nabla^*$ on $E^*$, we shall state a covariant divergence theorem for bundle-valued forms. We shall also prove that taking the covariant divergence associated to $\nabla^*$ when the dual space $V^\star$ is considered, corresponds to taking the covariant exterior derivative ${\rm d}^{\nabla^*}$ when the dual space $V^\dagger$ is chosen.

In order to define the covariant divergence, we shall need the \emph{trace} operator, which is the vector bundle morphism over the identity, ${\rm id}_M$, defined as:
\begin{equation*}
\tr:\textstyle\bigwedge^{k+1}TM\otimes\bigwedge^m T^*M\otimes E^*\to\bigwedge^k TM\otimes\bigwedge^{m-1}T^*M\otimes E^*,\quad U\otimes\mu\otimes\eta\mapsto\displaystyle\sum_{i=1}^{k+1}\hat U^i\otimes i_{U_i}\mu\otimes\eta,
\end{equation*}
where $U=U_1\wedge\dots\wedge U_{k+1}\in\bigwedge^{k+1} TM$, $\mu\in T^*M$ and $\hat U^r=U_1\wedge\dots\wedge U_{r-1}\wedge U_{r+1}\wedge\dots\wedge U_{k+1}\in\bigwedge^k TM$. In local coordinates, for $\chi\otimes\eta\in\Lambda ^{k+1}_m(M,E^*)$ with $ \chi =\chi^J\partial_J\otimes d^m x$, it reads
$$
\tr(\chi\otimes\eta)=\sum_{r=1}^{k+1}\chi^J\,\partial_{J_r}\otimes d_{(j_r)}^{m-1} x\otimes\eta,
$$
where we use the multi-index notation (recall Remark \ref{remark:multindex}). For the particular case $k=0$, we have $\chi=\chi^i\partial_i\otimes d^m x$ and the previous expression reduces to (cf. also \cite{GB2022}):
$$
\tr(\chi\otimes\eta)=\chi^i\,d_i^{m-1}x\otimes\eta=i_{\partial_i}\chi(dx^i,\dots)\otimes\eta.
$$

Given a linear connection $\nabla$ on $\pi_{E,M}:E\to M$, let $\nabla^*$ be its dual connection. The covariant exterior derivatives of $\nabla$ and $\nabla^*$ are denoted by
\begin{equation*}
\nd:\Omega^s(M,E)\to\Omega^{s+1}(M,E),\qquad{\rm d}^{\nabla^*}:\Omega^s(M,E^*)\to\Omega^{s+1}(M,E^*),
\end{equation*}
respectively, $0\leq s\leq m-1$ (see, e.g., \cite{marsh2018mathematics}). 
Note that $\nd^*$ may be trivially extended to $\Lambda_s^k(M,E^*)\to\Lambda_{s+1}^k(M,E^*)$; namely,
\begin{equation*}
\nd^*(U\otimes\alpha)=U\otimes\nd^*\alpha,\qquad U\in\Gamma(\textstyle\bigwedge^k TM),~\alpha\in\Omega^s(M,E^*).
\end{equation*}

The \emph{divergence} for $E^*$-valued tensor fields is defined as
\begin{equation*}
\operatorname{div}^{\nabla^*}:\Lambda_m^{k+1}(M,E^*)\mapsto\Lambda_m^k(M,E^*),\quad\chi\mapsto\nd^*(\tr\chi). 
\end{equation*}
Explicitly, it is given by
\begin{equation}\label{def_div_E_star}
{\rm div}^{\nabla^*}(\xi)=\sum_{i=1}^{k+1}\hat U^i\otimes\nd^*(i_{U_i}\mu\otimes\eta),\qquad \xi=U\otimes\mu\otimes\eta\in\Lambda_m^{k+1}(M,E^*),
\end{equation}
where we write $U=U_1\wedge\hdots\wedge U_{k+1}$.

Note that, for standard tensor fields, the divergence is defined canonically, i.e., without using a linear connection:
\begin{equation*}
\operatorname{div}:\Lambda_m^{k+1}(M)\to\Lambda_m^k(M),\quad\chi\mapsto{\rm d}(\tr\chi).
\end{equation*}

\begin{remark}[Contractions in local coordinates]\rm\label{rem:contraction}
Let $(x^i)$, $1\leq i\leq m$, be local coordinates on $M$, as in Remark \ref{remark:multindex}, and write $\alpha=\alpha^I\,\partial_I\otimes d^m x\in\Lambda_m^k(M)$ and $\varphi=\varphi_I\,dx^I\in\Omega^k(M)$. Then the contraction locally reads
\begin{equation*}
\alpha\cdot\varphi=\alpha^I\,\varphi_I\,(\partial_I\cdot dx^I)\otimes d^m x=\alpha^I\,\varphi_I\,d^m x.
\end{equation*}
In addition, for $\chi=\chi^J\,\partial_J\otimes d^m x\in\Lambda_m^{k+1}(M)$, we obtain
\begin{equation*}
\chi\cdot\varphi=\chi^J\,\varphi_I\,(\partial_J\cdot dx^I)\otimes d^m x
=(-1)^k\sum_{r=1}^{k+1}\chi^J\varphi_{J_r}\,\partial_{j_r}\otimes d^m x,
\end{equation*}
where we have used that $\partial_J\cdot dx^I=(-1)^k\,\partial_{j_r}$ when $I=J_r$ for some $1\leq r\leq k+1$ and vanishes otherwise.
\end{remark}

\begin{proposition}[Covariant divergence theorem]\label{prop:covariantdivergence}
Consider  $\delta\varphi\in\Omega^k(M,E)$ and $\chi\in\Lambda_m^{k+1}(M,E^*)$. Then
\begin{equation*}
\operatorname{div}(\chi\bcdot\delta\varphi)=\chi\bcdot\nd\delta\varphi+(-1)^k\left(\operatorname{div}^{ \nabla ^*}\chi\right)\bcdot\delta\varphi.
\end{equation*}
In particular, by Stokes' theorem we have
\begin{equation*}
\int_M\chi\bcdot\nd\delta\varphi=-(-1)^k\int_M\left(\operatorname{div}^{ \nabla ^*}\chi\right)\bcdot\delta\varphi+\int_{\partial M}\iota_{\partial M}^*(\tr\chi)\bcdot\iota_{\partial M}^*\delta\varphi.
\end{equation*}
\end{proposition}

\begin{proof}
Let $(x^i)$, $1\leq i\leq m$, be local coordinates on $M$, as in Remark \ref{remark:multindex}, $\{B_a\mid1\leq a\leq n\}$ be a basis of local sections of $\pi_{E,M}:E\to M$, and $\{B^a\mid1\leq a\leq n\}$ be the dual basis. Hence, we may write
\begin{equation*}
\delta\varphi=(\delta\varphi)_I^a\,dx^I\otimes B_a,\qquad\chi=\chi_a^J\,\partial_J\otimes d^m x\otimes B^a,
\end{equation*}
for some local functions $(\delta\varphi)_I^a,\chi_a^J \in C^\infty(M)$, where $|I|=k$, $|J|=k+1$ and $1\leq a\leq n$. In addition, the covariant exterior derivative is locally given by
\begin{equation*}
\nd\delta\varphi=\left(\partial_i(\delta\varphi)_I^a+\Gamma_{i,b}^a\,(\delta\varphi)_I^b\right)dx^i\wedge dx^I\otimes B_a,
\end{equation*}
where $\Gamma_{i,b}^a\in C^\infty(M)$, $1\leq i\leq m$, $1\leq a,b\leq n$, are the Christoffel symbols of $\nabla$. Analogously, we have
\begin{equation*}
\operatorname{div}^{ \nabla ^*}\chi={\rm d}^{\nabla^*}\left(\sum_{r=1}^{k+1}\chi_a^J \,\partial_{J_r}\otimes d_{(j_r)}^{m-1} x\otimes B^a\right)=\sum_{r=1}^{k+1}\left(\partial_{j_r}\chi_a^J - \Gamma_{j_r,a}^b\,\chi_b^J\right)\partial_{J_r}\otimes d^m x\otimes B^a,
\end{equation*}
where we have used that the Christoffel symbols of $\nabla^*$ are $-\Gamma_{i,a}^b$, $1\leq i\leq m$, $1\leq a,b\leq m$. From this and Remark \ref{rem:contraction}, we obtain:
\begin{align*}
\chi\bcdot\nd\delta\varphi & =(-1)^k\sum_{r=1}^{k+1}\chi_a^J\left(\partial_{j_r}(\delta\varphi)_{J_r}^a+\Gamma_{j_r,b}^a\,(\delta\varphi)_{J_r}^b\right)d^m x,\\
\left(\operatorname{div}^{ \nabla ^*}\chi\right)\bcdot\delta\varphi & =\sum_{r=1}^{k+1}\left(\partial_{j_r}\chi_a^J-\Gamma_{j_r,a}^b\,\chi_b^J\right)(\delta\varphi)_{J_r}^a\,d^m x.
\end{align*}
Moreover, we have
\begin{equation*}
\chi\bcdot\delta\varphi=(-1)^k\sum_{r=1}^{k+1}\chi_a^J (\delta\varphi)_{J_r}^a\,\partial_{j_r}\otimes d^m x,\qquad\tr(\chi\bcdot\delta\varphi)=\sum_{r=1}^{k+1}\chi_a^J (\delta\varphi)_{J_r}^a\,d_{(j_r)}^{m-1} x.
\end{equation*}
The result is now a straightforward computation by using the previous expressions:
\begin{align*}
(-1)^k\chi\bcdot\nd\delta\varphi+\left(\operatorname{div}^{ \nabla ^*}\chi\right)\bcdot\delta\varphi & = \sum_{r=1}^{k+1}\left(\chi_a^J\partial_{j_r}(\delta\varphi)_{J_r}^a+\left(\partial_{j_r}\chi_a^J\right)(\delta\varphi)_{J_r}^a\right)d^m x\\
& = \sum_{r=1}^{k+1}\partial_{j_r}\left(\chi_a^J (\delta\varphi)_{J_r}^a\right)d^m x\\
& =(-1)^k\operatorname{div}\left(\chi\bcdot\delta\varphi\right).
\end{align*}
The second part is a straightforward consequence of  Stokes' theorem by noting that
\begin{equation*}
\tr (\chi\bcdot\delta\varphi)=\sum_{r=1}^{k+1}\chi_a^J(\delta\varphi)_{J_r}^a\,d_{(j_r)}^{m-1} x=(\tr \chi)\bcdot\delta\varphi,
\end{equation*}
and $\iota_{\partial M}^*((\tr \chi)\bcdot\delta\varphi)=\iota_{\partial M}^*(\tr \chi)\bcdot\iota_{\partial M}^*\delta\varphi$.
\end{proof}

\medskip

Thanks to the map $\Phi_E$ in \eqref{eq:PhiE} and the following lemma, we will be able to relate the partial derivatives of the Lagrangian when the two choices of the restricted dual are considered (see section \ref{sec:partialderivativesvectorkforms} below).

\begin{lemma}\label{lemma:PhiEcontractions}\rm
Consider $\delta\varphi\in\Omega^k(M,E)$, $\zeta\in\Lambda_m^k(M,E^*)$ and $\chi\in\Lambda_m^{k+1}(M,E^*)$. Then
\begin{align*}
& \zeta\bcdot\delta\varphi=\delta\varphi\bwedge\Phi_E(\zeta),\qquad\qquad\left(\operatorname{div}^{ \nabla ^*}\chi\right)\bcdot\delta\varphi= \delta\varphi\bwedge{\rm d}^{\nabla^*}\Phi_E(\chi),\\
& (\tr\chi)\bcdot\delta\varphi=\delta\varphi\bwedge\Phi_E(\chi).
\end{align*}
\end{lemma}

\begin{proof}
By working locally, as in Remark \ref{remark:multindex}, and using a basis $\{B_a\mid 1\leq a\leq n\}$ of local sections of $\pi_{E,M}:E\to M$, we may write:
\begin{equation*}
\delta\varphi=(\delta\varphi)_I^a\,dx^I\otimes B_a,\qquad\zeta=\zeta_a^I\,\partial_I\otimes d^m x\otimes B^a,\qquad\chi=\chi_a^J\,\partial_J\otimes d^m x\otimes B^a,
\end{equation*}
for some local functions $(\delta\varphi)_I^a,\zeta_a^I,\chi_a^J\in C^\infty(M)$, where $|I|=k$, $|J|=k+1$, and $\{B^a\mid 1\leq a\leq n\}$ is the corresponding dual basis. Hence,
\begin{eqnarray*}
\operatorname{div}^{\nabla^*}\chi=\sum_{r=1}^{k+1}\left(\partial_{j_r}\chi_a^J-\Gamma_{j_r,a}^b\,\chi_b^J\right)\partial_{J_r}\otimes d^m x\otimes B^a,\qquad & \Phi_E(\chi)=\chi_a^J\,d_J^{m-k-1}x\otimes B^a,\\
{\rm d}^{\nabla^*}\Phi_E(\chi)%=\left(\partial_j\chi_a^J-\Gamma_{j,a}^b\,\chi_b^J\right)\,dx^j\wedge d_J^{m-k-1}x\otimes B^a
=\sum_{r=1}^{k+1}\left(\partial_{j_r}\chi_a^J-\Gamma_{j_r,a}^b\,\chi_b^J\right)\,d_{J_r}^{m-k}x\otimes B^a,\qquad & \displaystyle\tr\chi=\sum_{r=1}^{k+1}\chi_a^J\,\partial_{J_r}\otimes d_{(j_r)}^{m-1}x\otimes B^a,
\end{eqnarray*}
where $\Gamma_{i,b}^a\in C^\infty(M)$, $1\leq i\leq m$, $1\leq a,b\leq n$, are the Christoffel symbols of $\nabla$. The result is now a straightforward computation with the aid of Remark \ref{rem:contraction}:
\begin{align*}
\delta\varphi\bwedge\Phi_E(\zeta) & =(\delta\varphi)_{I'}^a\,\zeta_a^I\,dx^{I'}\wedge d_I^{m-k}x=(\delta\varphi)_I^a\,\zeta_a^I\,d^m x=\zeta\bcdot\delta\varphi,\\
\delta\varphi\bwedge{\rm d}^{\nabla^*}\Phi_E(\chi) & =\sum_{j=1}^{k+1}\left(\partial_{j_r}\chi_a^J-\Gamma_{j_r,a}^b\,\chi_b^J\right)(\delta\varphi)_I^a\,dx^I\wedge d_{J_r}^{m-k}x\\
& =\sum_{r=1}^{k+1}\left(\partial_{j_r}\chi_a^J-\Gamma_{j_r,a}^b\,\chi_b^J\right)(\delta\varphi)_{J_r}^a\,d^mx=\left(\operatorname{div}^{\nabla^*}\chi\right)\bcdot\delta\varphi,\\
\delta\varphi\bwedge\Phi_E(\chi) & =\chi_a^J\,(\delta\varphi)_I^a\,dx^I\wedge d_J^{m-k-1}x\\
& =\sum_{r=1}^{k+1}\chi_a^J\,(\delta\varphi)_{J_r}^a\,d_{(j_r)}^{m-1}x=(\tr \chi)\bcdot\delta\varphi.
\end{align*}
\end{proof}

\subsection{Lagrangians and partial derivatives}\label{sec:partialderivativesvectorkforms}

Given a linear connection $ \nabla $ on $ \pi _{E,M}: E \rightarrow M$, we shall consider the class of Lagrangian functions defined in terms of a Lagrangian density $\mathscr{L}$ as
\begin{equation}\label{eq:lagrangianvectorbundlekforms}
L_\nabla:TV\to\mathbb R,\quad L _ \nabla (\varphi,\nu )= \int_M\mathscr L\left(\varphi, \nu ,\nd\varphi\right),
\end{equation}
with $\mathscr{L}$ being a vector bundle map
\begin{equation}\label{Lagrangian_density_bundle}
\textstyle\mathscr \!\!\!\mathscr{L}:W^k(M,E):=\left(\bigwedge^k T^*M\otimes E\right)\times_M\left(\bigwedge^k T^*M\otimes E\right)\times_M\left(\bigwedge^{k+1}T^*M\otimes E\right)\to\bigwedge^m T^*M.
\end{equation}
Note that we used the shorthand $W^k(M,E)\rightarrow M$ to denote the starting vector bundle. Note also that $L_\nabla$ depends parametrically on the chosen linear connection since the Lagrangian density is evaluated on the covariant exterior derivative $ {\rm d} ^ \nabla \varphi $.

\medskip

Below, we first define the fiber derivatives of the Lagrangian density $\mathscr{L}$. We then show how the partial functional derivative of $L_\nabla$ can be expressed in terms of these derivatives, and we find that it belongs to the restricted dual. This ensures that a Lagrange--Dirac dynamical system can be constructed for such Lagrangians.

\medskip

The \textit{fiber derivatives} of $\mathscr L$,
\begin{align*}
& \frac{\partial\mathscr L}{\partial\varphi}:W^k(M,E)\to\textstyle\bigwedge^k TM\otimes\bigwedge^m T^*M\otimes E^*,\\
& \frac{\partial\mathscr L}{\partial\nu }:W^k(M,E)\rightarrow\textstyle\bigwedge^k TM\otimes\bigwedge^m T^* M\otimes E^*,\\
& \frac{\partial\mathscr L}{\partial\zeta}:W^k(M,E)\to\textstyle\bigwedge^{k+1}TM\otimes\bigwedge^m T^* M\otimes E^*,
\end{align*}
are defined in the standard way; namely, for each $x\in M$ and $(\varphi_x,\nu _x,\zeta_x)\in W^k(M,E)_x$, they are given by:
\begin{align*}
\frac{\partial\mathscr L}{\partial\varphi}(\varphi_x, \nu _x,\zeta_x)\bcdot\delta\varphi_x & =\left.\frac{d}{d\epsilon}\right|_{ \epsilon=0}\mathscr L(\varphi_x+\epsilon \delta\varphi_x, \nu _x,\zeta_x),\qquad\delta\varphi_x\in\textstyle\bigwedge^k T_x^*M\otimes E_x,\\
\frac{\partial\mathscr L}{\partial\nu}(\varphi_x, \nu _x,\zeta_x)\bcdot\delta\nu_x & =\left.\frac{d}{d\epsilon}\right|_{ \epsilon=0}\mathscr L(\varphi_x, \nu _x+\epsilon \delta\nu_x,\zeta_x),\qquad\delta\nu_x\in\textstyle\bigwedge^k T_x^*M\otimes E_x,\\
\frac{\partial\mathscr L}{\partial\zeta}(\varphi_x, \nu _x,\zeta_x)\bcdot\delta\zeta_x & =\left.\frac{d}{d\epsilon}\right|_{ \epsilon=0}\mathscr L(\varphi_x, \nu _x,\zeta_x+\epsilon \delta\zeta_x),\qquad\delta\zeta_x\in\textstyle\bigwedge^{k+1} T_x^*M\otimes E_x.
\end{align*}

When working with $V^\dagger$ as the restricted dual, we will need a slightly different realisation of the fiber derivatives of $\mathscr L$, which we denote by
\begin{align*}
& \frac{\partialnew\mathscr L}{\partialnew\varphi}:W^k(M,E)\to\textstyle\bigwedge^{m-k} T^* M\otimes E^*,\vspace{0.1cm}\\
& \frac{\partialnew\mathscr L}{\partialnew\nu }:W^k(M,E)\to\textstyle\bigwedge^{m-k} T^* M\otimes E^*,\vspace{0.1cm}\\
&\frac{\partialnew\mathscr L}{\partialnew\zeta}:W^k(M,E)\to\textstyle\bigwedge^{m-k-1} T^* M\otimes E^*.
\end{align*}
More specifically, they are defined as:
\begin{align*}
\delta\varphi_x\bwedge\frac{\partialnew\mathscr L}{\partialnew\varphi}(\varphi_x,\nu _x,\zeta_x)=\left.\frac{d}{d\epsilon}\right|_{{\epsilon}=0}\mathscr L(\varphi_x+{\epsilon}\,\delta\varphi_x,\nu _x,\zeta_x),\qquad & \delta\varphi_x\in\textstyle\bigwedge^k T_x^*M\otimes E_x,\\
\delta{\nu }_x\bwedge\frac{\partialnew\mathscr L}{\partialnew\nu }(\varphi_x,\nu _x,\zeta_x)=\left.\frac{d}{d\epsilon}\right|_{\epsilon=0}\mathscr L(\varphi_x,\nu _x+{\epsilon}\,\delta\nu _x,\zeta_x),\qquad & \delta\nu _x\in\textstyle\bigwedge^k T_x^*M\otimes E_x,\\
\delta{\zeta}_x\bwedge\frac{\partialnew\mathscr L}{\partialnew\zeta}(\varphi_x,\nu _x,\zeta_x)=\left.\frac{d}{d\epsilon}\right|_{{\epsilon}=0}\mathscr L(\varphi_x,\nu _x,\zeta_x+{\epsilon}\,\delta\zeta_x),\qquad & \delta\zeta_x\in\textstyle\bigwedge^{k+1} T_x^*M\otimes E_x.
\end{align*}
Lemma \ref{lemma:PhiEcontractions} gives the following relations between these fiber derivatives through the map $\Phi_E$ defined in \eqref{eq:PhiE}:
\begin{align*}
\frac{\partialnew\mathscr L}{\partialnew\varphi}=\Phi_E\circ\frac{\partial\mathscr L}{\partial\varphi},\quad\frac{\partialnew\mathscr L}{\partialnew\nu }=\Phi_E\circ\frac{\partial\mathscr L}{\partial\nu },\quad\frac{\partialnew\mathscr L}{\partialnew\zeta}=\Phi_E\circ\frac{\partial\mathscr L}{\partial\zeta}.
\end{align*}

\medskip

The \textit{partial functional derivatives} of a general function $L:TV \rightarrow\mathbb{R}$ are the maps
\begin{equation*}
\frac{\delta L}{\delta\varphi},~\frac{\delta L}{\delta\nu }:TV=T\Omega^k(M,E)\to V'
\end{equation*}
defined as
\begin{equation*}
\frac{\delta L}{\delta\varphi}(\varphi,\nu)(\delta\varphi)=\left.\frac{d}{d \epsilon}\right|_{\epsilon=0}L(\varphi+\epsilon\,\delta\varphi,\nu),\qquad\frac{\delta L}{\delta\nu }(\varphi,\nu)(\delta\nu)=\left.\frac{d}{d\epsilon}\right|_{\epsilon=0}L(\varphi,\nu+\epsilon\,\delta\nu),
\end{equation*}
for each $(\varphi,\nu)\in TV$ and $\delta\varphi, \delta\nu\in V$.
Based on this, the \textit{differential} of $L:TV \rightarrow \mathbb{R}$ is the map
\begin{equation*}
{\rm d}L :TV\to T'(TV),\quad(\varphi, \nu )\mapsto\left(\varphi, \nu ,\frac{\delta L }{\delta\varphi}(\varphi,\nu ),\frac{\delta L }{\delta\nu}(\varphi, \nu )\right).
\end{equation*}

The next result ensures that  the partial functional derivatives of a Lagrangian defined through a density as in  \eqref{eq:lagrangianvectorbundlekforms} and \eqref{Lagrangian_density_bundle} may be regarded as elements of either $V^\star\subset V'$ or $V^\dagger\subset V'$ by means of the corresponding identifications. As a consequence, its differential is taking values in either $T^\star(TV)$ or in $T^\dagger(TV)$.

\begin{lemma}\rm\label{lemma:partialderivativevectorbundlekforms}
Let $\nabla$ be a linear connection on $\pi_{E,M}:E\to M$ and $L_\nabla:TV\to\mathbb R$ be a Lagrangian defined in terms of a density $\mathscr{L}$, as in \eqref{eq:lagrangianvectorbundlekforms}, where we set $\zeta={\rm d}^{\nabla}\varphi$. Then, for each $(\varphi,\nu )\in TV$, the partial functional derivatives of $L_\nabla$ are given as follows:
\begin{enumerate}
    \item For $V^\star$ under the identification given by $\blangle\cdot,\cdot\brangle_\star$, they read
    \begin{align*}
    \frac{\delta L_\nabla}{\delta\varphi}(\varphi, \nu ) & =\left(\frac{\partial\mathscr L}{\partial\varphi}(\varphi,\nu ,\nd\varphi)-(-1)^{k}\operatorname{div}^{ \nabla ^*} \left( \frac{\partial\mathscr L}{\partial\zeta}(\varphi,\nu ,\nd\varphi) \right) ,~\iota_{\partial M}^*\left(\tr\frac{\partial\mathscr L}{\partial\zeta}(\varphi, \nu ,\nd\varphi)\right)\right),\\
    \frac{\delta L_\nabla}{\delta\nu }(\varphi, \nu ) & =\left(\frac{\partial\mathscr L}{\partial \nu }(\varphi,\nu ,\nd\varphi),~0\right).
    \end{align*}
    \item For $V^\dagger$ under the identification given by $\blangle\cdot,\cdot\brangle_\dagger$, they read
    \begin{align*}
    \frac{\delta L_\nabla}{\delta\varphi}(\varphi, \nu ) & =\left(\frac{\partialnew\mathscr L}{\partialnew\varphi}(\varphi,\nu ,\nd\varphi)-(-1)^k{\rm d}^{ \nabla^*} \left( \frac{\partialnew\mathscr L}{\partialnew\zeta}(\varphi, \nu ,\nd\varphi) \right) ,~\iota_{\partial M}^*\left(\frac{\partialnew\mathscr L}{\partialnew\zeta}(\varphi,\nu ,\nd\varphi)\right)\right),\\
    \frac{\delta L_\nabla}{\delta\nu }(\varphi, \nu )& =\left(\frac{\partialnew\mathscr L}{\partialnew \nu }(\varphi,\nu ,\nd\varphi),~0\right).
    \end{align*}
\end{enumerate} 
\end{lemma}

\begin{proof}
Let $\delta\varphi\in V$. By using Proposition \ref{prop:covariantdivergence}, we obtain the first part:
\begin{align*}
\frac{\delta L_\nabla}{\delta\varphi}(\varphi,\nu )(\delta\varphi) & =\left.\frac{d}{d \epsilon}\right|_{\epsilon=0}\int_{ M}\mathscr L(\varphi+ \epsilon\,\delta\varphi,\nu ,\nd(\varphi+\epsilon\,\delta\varphi))\\
& =\int_{ M}\left(\frac{\partial\mathscr L}{\partial\varphi}(\varphi,\nu ,\nd\varphi)\bcdot\delta\varphi+\frac{\partial\mathscr L}{\partial\zeta}(\varphi,\nu ,\nd\varphi)\bcdot\nd\delta\varphi\right)\\
& =\int_{ M}\left(\frac{\partial\mathscr L}{\partial\varphi}(\varphi,\nu ,\nd\varphi)- (-1)^{k}\operatorname{div}^{\nabla^*}\left(\frac{\partial\mathscr L}{\partial\zeta}(\varphi,\nu ,\nd\varphi)\right)\right)\bcdot\delta\varphi\\
& \qquad+\int_{\partial M}\iota_{\partial M}^*\left(\tr \frac{\partial\mathscr L}{\partial\zeta}(\varphi,\nu ,\nd\varphi)\right)\bcdot\iota_{\partial M}^*\delta\varphi,\vspace{2mm}\\
\frac{\delta L_\nabla}{\delta\nu }(\varphi,\nu )(\delta \nu ) & =\left.\frac{d}{d \epsilon}\right|_{ \epsilon=0}\int_{ M}\mathscr L(\varphi,\nu + \epsilon\,\delta \nu ,\nd\varphi)=\int_{ M}\frac{\partial\mathscr L}{\partial\nu }(\varphi,\nu ,\nd\varphi)\bcdot\delta \nu.
\end{align*}
The second part is obtained from Lemma \ref{lemma:PhiEcontractions} in the previous expressions; namely,
\begin{align*}
\frac{\delta L_\nabla}{\delta\varphi}(\varphi,\nu )(\delta\varphi) & =\int_{ M}\delta\varphi\bwedge\left(\frac{\partialnew\mathscr L}{\partialnew\varphi}(\varphi,\nu ,\nd\varphi)- (-1)^k{\rm d}^{\nabla^*}\left(\frac{\partialnew\mathscr L}{\partialnew\zeta}(\varphi,\nu ,\nd\varphi)\right)\right)\\
& \qquad+\int_{\partial M}\iota_{\partial M}^*\delta\varphi\bwedge\iota_{\partial M}^*\left(\frac{\partialnew\mathscr L}{\partialnew\zeta}(\varphi,\nu ,\nd\varphi)\right),\vspace{2mm}\\
\frac{\delta L_\nabla}{\delta\nu }(\varphi,\nu )(\delta \nu ) & =\int_{ M}\delta\nu\bwedge\frac{\partialnew\mathscr L}{\partialnew\nu }(\varphi,\nu ,\nd\varphi).
\end{align*}
\end{proof}

\medskip

From the previous Lemma, the differential
\begin{equation*}
{\rm d}L_{\nabla}:TV\to T'(TV),\quad(\varphi, \nu )\mapsto\left(\varphi, \nu ,\frac{\delta L_{\nabla}}{\delta\varphi}(\varphi,\nu ),\frac{\delta L_{\nabla}}{\delta\nu}(\varphi, \nu )\right)
\end{equation*}
of $L_\nabla$ takes values in the restricted cotangent bundle, i.e., either $T^\star(TV)$ or  $T^\dagger(TV)$ depending on the chosen identification. We can therefore apply to it the canonical isomorphisms $\gamma_{T^\star V}$ and $\gamma_{T^\dagger V}$, which is a crucial step in the definition of Lagrange--Dirac systems.

\begin{definition}\rm\label{DiracDiff_VectBundleValuedkForms}
There are two choices to define the \emph{Dirac differential} of the Lagrangian \eqref{eq:lagrangianvectorbundlekforms}, which we distinguish through the corresponding superscript:
\begin{enumerate}
    \item[(i)] By using (i) of Lemma  \ref{lemma:partialderivativevectorbundlekforms}, we define
    \begin{equation*}
    {\rm d}_{ D}^\star L_{\nabla}=\gamma_{T^\star V}\circ{\rm d} L_{\nabla}:TV\to T^\star(T^\star V),\quad (\varphi,\nu )\mapsto\left(\varphi,\frac{\delta L_{\nabla}}{\delta\nu }(\varphi,\nu ),-\frac{\delta L_{\nabla}}{\delta\varphi}(\varphi, \nu ), \nu \right).
    \end{equation*}
    \item[(ii)] By using (ii) of Lemma \ref{lemma:partialderivativevectorbundlekforms}, we define
    \begin{equation*}
    {\rm d}_{ D}^\dagger L_{\nabla}=\gamma_{T^\dagger V}\circ{\rm d} L_{\nabla}:TV\to T^\dagger\big(T^\dagger V\big),\quad(\varphi, \nu )\mapsto\left(\varphi,\frac{\delta L_{\nabla}}{\delta \nu }(\varphi, \nu  ),-\frac{\delta L_{\nabla}}{\delta\varphi}(\varphi, \nu ),\nu \right).
    \end{equation*}
\end{enumerate}
\end{definition}

The concrete expressions of the Dirac differential in terms of the density $\mathscr{L}$ in (i) and (ii) are found by using the results of Lemma \ref{lemma:partialderivativevectorbundlekforms}.
Analogously, for each choice of the restricted dual space we may express the \emph{Legendre transform} of $L_\nabla$,
\begin{equation*}
\mathbb FL_\nabla:TV\to T'V,\quad(\varphi,\nu)\mapsto\mathbb FL_\nabla(\varphi,\nu)=\left(\varphi,\frac{\delta L_\nabla}{\delta\nu}(\varphi,\nu)\right),
\end{equation*}
as follows:
\begin{itemize}
    \item[(i)] For $V^\star=\Lambda_m^k(M,E^*)\times\Lambda_{m-1}^k(\partial M,E^*)$, it reads
    \begin{equation*}
    \mathbb{F}^{\star}L_{\nabla}: TV \to T^{\star}V,\quad(\varphi,\nu)\mapsto\left(\varphi,\frac{\partial \mathscr L}{\partial\nu}(\varphi,\nu,{\rm d}^{\nabla}\varphi),0\right).
    \end{equation*}
    \item[(ii)] For $V^\dagger=\Omega^{m-k}(M,E^*)\times\Omega^{m-k-1}(\partial M,E^*)$, it reads
    \begin{equation*}
    \mathbb{F}^{\dagger}L_{\nabla}: TV \to T^{\dagger}V,\quad(\varphi,\nu)\mapsto\left(\varphi,\frac{\partialnew\mathscr L}{\partialnew\nu}(\varphi,\nu,{\rm d}^{\nabla}\varphi),0\right).
    \end{equation*}
\end{itemize}

\subsection{Lagrange--Dirac dynamical systems}\label{sec:LDVectBundleValued_kforms}

Before introducing Lagrange--Dirac dynamical systems, we define the body and boundary forces by considering the infinite-dimensional version of the external force map $F:TQ\rightarrow T^*Q$ and associated Lagrangian force field $\widetilde{F}:TQ \rightarrow T^*(T^*Q)$ as discussed in \eqref{F} and \eqref{Ftilde}. With our definition of the restricted dual space, this infinite-dimensional version automatically incorporates both body and boundary contributions.

As above, we choose a linear connection $\nabla$ on the vector bundle $\pi_{E,M}:E\to M$ and we consider a Lagrangian $L_\nabla:TV\to\mathbb R$ given by \eqref{eq:lagrangianvectorbundlekforms}.

\begin{definition}\rm\label{ExtForce_VectBundleValuedkForms}
Associated to each choice of the restricted dual, we define \emph{external forces} as follows:
\begin{enumerate}[(i)]
    \item For $V^\star=\Lambda_m^k(M,E^*)\times\Lambda_{m-1}^k(\partial M,E^*)$, an \textit{external force} is a map 
    \begin{equation*}
    F^\star: TV \to T^{\star}V,\quad(\varphi,\nu)\mapsto F^\star(\varphi,\nu)=(\varphi, \mathcal F^{\star}(\varphi,\nu),\mathcal F^\star_\partial(\varphi,\nu)),
    \end{equation*}
    for some $\mathcal F^\star:TV\to\Lambda_m^k(M,E^*)$ and $\mathcal F^\star_\partial:TV\to\Lambda_{m-1}^k(\partial M,E^*)$. The associated \emph{Lagrangian force field} is the map $\widetilde{F}^\star: TV \to T^\star\left(T^\star V\right)$ defined as
    \[
    \left\langle \widetilde{F}^\star(\varphi,\nu), W\right\rangle_\star = \left\langle F^\star(\varphi,\nu), T_{\mathbb F^\star L_\nabla(\varphi,\nu)}\pi_{V}^\star(W) \right\rangle_\star,\qquad (\varphi,\nu) \in TV,~W\in T_{\mathbb F^\star L(\varphi,\nu)}(T^\star V),
    \]
    where $\pi_V^\star:T^{\star}V \to V$ is the natural projection and $T\pi_V^\star:T\left(T^\star V\right) \to TV$ denotes its tangent map. More explicitly, it reads
    \begin{equation*}
    \widetilde{F}^\star(\varphi,\nu)= \left(\varphi,\frac{\partial\mathscr L}{\partial\nu}(\varphi,\nu,{\rm d}^{\nabla}\varphi),0,\mathcal F^\star(\varphi,\nu),\mathcal F^\star_\partial(\varphi,\nu),0 \right).
    \end{equation*}
    \item For $V^\dagger=\Omega^{m-k}(M,E^*)\times\Omega^{m-k-1}(\partial M,E^*)$, an \textit{external force} is a map 
    \begin{equation*}
    F^{\dagger}:TV\to T^{\dagger} V,\quad(\varphi,\nu)\mapsto F^{\dagger}(\varphi,\nu)=(\varphi,\mathcal F^{\dagger}(\varphi,\nu),\mathcal F^{\dagger}_\partial(\varphi,\nu)),
    \end{equation*}
    for some $\mathcal F^{\dagger}:TV\to\Omega^{m-k}(M,E^*)$ and $\mathcal F^{\dagger}_\partial:TV\to\Omega^{m-k-1}(\partial M,E^*)$. The associated \emph{Lagrangian force field} is the map $\widetilde{F}^{\dagger}: TV \to T^{\dagger}\left(T^{\dagger} V\right)$ defined as
    \[
    \left\langle \widetilde{F}^\dagger(\varphi,\nu), W\right\rangle_\dagger = \left\langle F^\dagger(\varphi,\nu), T_{\mathbb F^\dagger L_\nabla(\varphi,\nu)}\pi_{V}^\dagger(W) \right\rangle_\dagger,\qquad (\varphi,\nu) \in TV,~W\in T_{\mathbb F^\dagger L(\varphi,\nu)}(T^\dagger V),
    \]
    where $\pi_V^\dagger:T^{\dagger}V \to V$ is the natural projection and $T\pi_V^\dagger:T\left(T^\dagger V\right) \to TV$ denotes its tangent map. More explicitly, it reads
    \begin{equation*}
    \widetilde{F}^{\dagger}(\varphi,\nu)= \left(\varphi,\frac{\partialnew\mathscr L}{\partialnew\nu}(\varphi,\nu,{\rm d}^{\nabla}\varphi),0,\mathcal F^{\dagger}(\varphi,\nu),\mathcal F^{\dagger}_\partial(\varphi,\nu),0\right).
    \end{equation*}
\end{enumerate}
\end{definition}

We are now ready to give the definition of the Lagrange--Dirac dynamical systems on $V=\Omega^k(M,E)$ with body and boundary forces.

\begin{definition}\rm\label{def:externalforces_VectBundleValuedkforms}
Associated to each choice of the restricted dual, we define the Lagrange--Dirac dynamical systems as follows:
\begin{itemize}
\item[(i)]
For the case $V^{\star}=\Lambda_m^k(M,E^*)\times\Lambda_{m-1}^k(\partial M,E^*)$, a {\it forced Lagrange--Dirac dynamical system} is a quadruple $(V=\Omega^k(M,E), D_{T^{\star}V}, L_{\nabla}, F^{\star})$. A solution for the forced Lagrange--Dirac dynamical system is any curve $(\varphi,\nu, \alpha,\alpha_\partial): [t_0,t_1] \to TV \oplus T^{\star}V$ that satisfies
\begin{equation}\label{ForcedLagDirac_VectBundleValuedkforms_Star}
\big((\varphi,\alpha,\alpha_\partial,\dot\varphi,\dot\alpha,\dot\alpha_\partial), {\rm d}_{ D}^\star L_{ \nabla }(\varphi,\nu)-\widetilde{F}^\star(\varphi,\nu)\big)\in D_{T^\star V}(\varphi,\alpha,\alpha_\partial).
\end{equation}
\vskip 3pt

\item[(ii)] For the case $V^{\dagger}=\Omega^{m-k}(M,E^*)\times\Omega^{m-k-1}(\partial M,E^*)$, a forced Lagrange--Dirac dynamical systems is a quadruple $(V=\Omega^k(M,E), D_{T^{\dagger}V}, L_{\nabla}, F^{\dagger})$. A solution of the forced Lagrange--Dirac dynamical system is any curve $(\varphi,\nu, \alpha,\alpha_\partial): [t_0,t_1] \to TV \oplus T^{\dagger}V$ that satisfies
\begin{equation}\label{ForcedLagDirac_VectBundleValuedkforms_Dagger}
\big((\varphi,\alpha,\alpha_\partial,\dot\varphi,\dot\alpha,\dot\alpha_\partial),{\rm d}_{ D}^{\dagger} L_{\nabla}(\varphi,\nu)-\widetilde{F}^{\dagger}(\varphi,\nu)\big)\in D_{T^{\dagger}V}(\varphi,\alpha,\alpha_\partial).
\end{equation}
\end{itemize}
\end{definition}
%
%\begin{framed}
\begin{proposition}\label{proposition:LDVectBundValued_kforms}\rm
Associated to each choice of the restricted dual space, the following statements hold:
\medskip

\begin{itemize}
    \item[(i)] For the case $V^{\star}=\Lambda_m^k(M,E^*)\times\Lambda_{m-1}^k(\partial M,E^*)$, a curve $(\varphi,\nu, \alpha,\alpha_\partial): [t_0,t_1] \to TV \oplus T^{\star}V$ is a solution of \eqref{ForcedLagDirac_VectBundleValuedkforms_Star} if and only if it satisfies the following system equations:
     \begin{equation}\label{EqMotionForcedLagDirac_VectBundleValuedkforms_Star}
    \left\{\begin{array}{ll}
    \dot\varphi=\nu,&\vspace{4mm}\\
    \displaystyle\alpha=\frac{\partial\mathscr L}{\partial\nu}(\varphi,\nu,{\rm d}^{\nabla}\varphi),\; & \displaystyle\dot\alpha=\frac{\partial\mathscr L}{\partial\varphi}(\varphi,\nu,{\rm d}^{\nabla}\varphi)-(-1)^{k}\operatorname{div}^{ \nabla ^*} \left( \frac{\partial\mathscr L}{\partial\zeta}(\varphi,\nu,{\rm d}^{\nabla}\varphi) \right) +\mathcal F^{\star}(\varphi,\nu),\vspace{2mm}\\
    \alpha_\partial=0, & \displaystyle\dot\alpha_\partial=\iota_{\partial M}^*\left(\tr \frac{\partial\mathscr L}{\partial\zeta}(\varphi,\nu,{\rm d}^{\nabla}\varphi)\right)+\mathcal F^{\star}_{\partial}(\varphi,\nu).
    \end{array}\right.
    \end{equation}
    \vskip 3pt
   
    \item[(ii)] 
    For the case $V^{\dagger}=\Omega^{m-k}(M,E^*)\times\Omega^{m-k-1}(\partial M,E^*)$, a curve $(\varphi,\nu, \alpha,\alpha_\partial): [t_0,t_1] \to TV \oplus T^{\dagger}V$ is a solution of \eqref{ForcedLagDirac_VectBundleValuedkforms_Dagger} if and only if it satisfies the following system equations:
    \begin{equation}\label{EqMotionForcedLagDirac_VectBundleValuedkforms_Dagger}
    \left\{\begin{array}{ll}
    \dot\varphi=\nu,&\vspace{0.4cm}\\
    \displaystyle\alpha=\frac{\partialnew\mathscr L}{\partialnew\nu}(\varphi,\nu,{\rm d}^{\nabla}\varphi),\; & \displaystyle\dot\alpha=\frac{\partialnew\mathscr L}{\partialnew\varphi}(\varphi,\nu,{\rm d}^{\nabla}\varphi)-(-1)^k\nd^* \left( \frac{\partialnew\mathscr L}{\partialnew\zeta}(\varphi,\nu,{\rm d}^{\nabla}\varphi) \right) +\mathcal F^\dagger(\varphi,\nu),\vspace{0.2cm}\\
    \alpha_\partial=0, & \displaystyle\dot\alpha_\partial=\iota_{\partial M}^*\left(\frac{\partialnew\mathscr L}{\partialnew\zeta}(\varphi,\nu,{\rm d}^{\nabla}\varphi)\right)+\mathcal F^\dagger_{\partial}(\varphi,\nu).
    \end{array}\right.
    \end{equation}
\end{itemize}
\end{proposition}

\begin{proof}
A direct computation using Definitions \ref{def:diracstructurebundle}, \ref{DiracDiff_VectBundleValuedkForms}, \ref{ExtForce_VectBundleValuedkForms} and \ref{def:externalforces_VectBundleValuedkforms} leads to the desired results.
\end{proof}

\medskip

These equations are the infinite dimensional version of the Lagrange--Dirac system given in \eqref{LDS}. Thanks to the choice of the restricted dual and duality pairing, they incorporate boundary conditions and boundary forces, see the last equation in \eqref{EqMotionForcedLagDirac_VectBundleValuedkforms_Star} and \eqref{EqMotionForcedLagDirac_VectBundleValuedkforms_Dagger}, while keeping the same canonical form as in the finite dimensional case. By eliminating the variables $\nu$, $\alpha$, and $\alpha_\partial$ via the first, second, and third equations, one gets the Euler--Lagrange equations with distributed and boundary forces, see below.

\begin{remark}[Local expression]\rm\label{remark:localexpression}
Let $(x^i)$, $1\leq i\leq m$, be local coordinates on $M$, as in Remark \ref{remark:multindex}, and $\{B_a\mid 1\leq a\leq n\}$ be a basis of local sections of $\pi_{E,M}:E\to M$, as in the proof of Proposition \ref{prop:covariantdivergence}. If the intersection of the coordinate domain with $\partial M$ is non-empty, the coordinates are assumed to be adapted to the boundary, i.e.,
\begin{equation*}
\partial M=\{(x^1,\hdots,x^m)\in\mathbb R^m\mid x^m=0\}.
\end{equation*}
The Christoffel symbols of $\nabla$ are denoted by $\Gamma_{i,b}^a\in C^\infty(M)$, $1\leq i\leq m$, $1\leq a,b\leq m$. Given a multi-index $I$ of length $k$, we define the following sets of indices:
\begin{align*}
\mathcal I & =\{(J,r)\in\mathbb N^{k+1}\times\{1,\dots,k+1\}\mid 1\leq j_1<\hdots<j_{k+1}\leq m,~J_r=I\}.
\end{align*}
Analogously, we introduce the following set:
\begin{equation*}
\mathcal I_m =\{I\in\mathbb N^k\mid 1\leq i_1<\hdots<i_k<m\},
\end{equation*}
and define $I^m=(i_1,\dots,i_k,m)\in\mathbb N^{k+1}$ for each $I\in\mathcal I_m$.

On the other hand, the local expression of $\alpha$ (and similar for $\alpha_\partial$, and the body and boundary forces) is the same regardless the choice of the restricted dual; namely,
\begin{equation*}
\alpha=\alpha_a^I\,\partial_I\otimes d^m m\otimes B^a\in\Lambda_m^k(M,E^*),\qquad\alpha=\alpha_a^I\,d_I^{m-k}x\,\otimes B^a\in\Omega^{m-k}(M,E^*).
\end{equation*}
For that reason, we will drop the superscripts $\star$ and $\dagger$ when working in coordinates.

Therefore, both systems of equations given in \eqref{EqMotionForcedLagDirac_VectBundleValuedkforms_Star} and \eqref{EqMotionForcedLagDirac_VectBundleValuedkforms_Dagger} have the same local expression, which reads:
\begin{equation}\label{EqMotionForcedLagDirac_VectBundleValuedkforms_local}
\left\{\begin{array}{ll}
\dot\varphi^a_I=\nu^a_I,&\vspace{4mm}\\
\displaystyle\alpha_a^I=\frac{\partial\mathscr L}{\partial\nu^a_I},\qquad & \displaystyle\dot\alpha^I_a=\frac{\partial\mathscr L}{\partial\varphi^a_I}-(-1)^{k}\bigg(\sum_{(J,r)\in\mathcal I}\partial_{j_r}\frac{\partial\mathscr L}{\partial\zeta_J^a}-\Gamma_{j_r,a}^b\frac{\partial\mathscr L}{\partial\zeta_J^b}\bigg) +\mathcal F^I_a,\vspace{2mm}\\
(\alpha_\partial)_a^I=0, & \displaystyle(\dot\alpha_\partial)_a^I=\left.\frac{\partial\mathscr L}{\partial\zeta_{I^m}^a}\right|_{x^m=0}+(\mathcal F_\partial)_a^I,\qquad I\in\mathcal I_m.
\end{array}\right.
\end{equation}
For $k=0$, the previous expression reduces to \eqref{EqMotionForcedLagDirac_VectBundleValuedkforms_k1}. For $k=1$, it reduces to:
\begin{equation*}
\left\{\begin{array}{ll}
\dot\varphi^a_i=\nu^a_i,&\vspace{4mm}\\
\displaystyle\alpha_a^i=\frac{\partial\mathscr L}{\partial\nu^a_i},\qquad & \displaystyle\dot\alpha^i_a=\frac{\partial\mathscr L}{\partial\varphi^a_i}+\sum_{j=1}^m\epsilon(i,j)\,\bigg(\partial_{j}\frac{\partial\mathscr L}{\partial\zeta_{(i,j)}^a}-\Gamma_{j,a}^b\frac{\partial\mathscr L}{\partial\zeta_{(i,j)}^b}\bigg)+\mathcal F^i_a,\vspace{2mm}\\
(\alpha_\partial)_a^i=0, & \displaystyle(\dot\alpha_\partial)_a^i=\left.\frac{\partial\mathscr L}{\partial\zeta_{(i,m)}^a}\right|_{x^m=0}+(\mathcal F_\partial)_a^i,\qquad 1\leq i<m,
\end{array}\right.
\end{equation*}
where we have used that $\zeta\in\Omega^{k+1}(M,E)$ and, thus $\zeta_{(i,j)}=-\zeta_{(j,i)}$ for each $1\leq i<j\leq m$, and we have denoted:
\begin{equation*}
\epsilon(i,j)=\begin{cases}
-1,\qquad & 1\leq j\leq i-1,\\
0, & i=j,\\
+1, & i+1\leq j\leq m.
\end{cases}
\end{equation*}
\end{remark}

\subsection{Energy considerations}

The \emph{energy density} associated with $\mathscr{L}$ is $\mathscr{E}: W^k(M,E) \rightarrow \bigwedge^mT^*M$ given by
\begin{align*}
\mathscr{E}(\varphi_x,\nu_x,\zeta_x) & =\frac{\partial\mathscr{L}}{\partial\nu}(\varphi_x,\nu_x,\zeta_x)\bcdot\nu_x-\mathscr{L}(\varphi_x,\nu_x,\zeta_x)\\
& =\nu_x\bwedge\frac{\partialnew\mathscr{L}}{\partialnew\nu}(\varphi_x,\nu_x,\zeta_x)-\mathscr{L}(\varphi_x,\nu_x,\zeta_x),
\end{align*}
for each $(\varphi_x,\nu_x,\zeta_x)\in W^k(M,E)_x$. Note that both choices of the restricted dual space can be used, yielding the same energy density.

\begin{proposition}[Energy balance]\rm\label{prop:energybalancevectorbundle}
Associated to each choice of the restricted dual, we have the following local and global form of the energy balance along the solutions of the forced Lagrange--Dirac equations:
\begin{itemize}
    \item[(i)] For $V^\star=\Lambda_m^k(M,E^*)\times\Lambda_{m-1}^k(\partial M,E^*)$, let $(\varphi,\nu,\alpha,\alpha_\partial):[t_0,t_1]\to TV\oplus T^\star V$ be a solution of \eqref{ForcedLagDirac_VectBundleValuedkforms_Star}. Then:
    \begin{equation}\label{EB_i_local}
    \frac{\partial }{\partial t}\mathscr{E}(\varphi,\nu,\nd\varphi)=-\operatorname{div}\left(\frac{\partial\mathscr L}{\partial\zeta}(\varphi,\nu,\nd\varphi)\bcdot\nu\right)+\mathcal F^\star(\varphi,\nu)\bcdot\nu,
    \end{equation}
    and
    \begin{equation}\label{EB_i}
    \frac{d}{dt} \int_M \mathscr{E}(\varphi,\nu,\nd\varphi)=\underbrace{\int_M\mathcal F^\star(\varphi,\nu)\bcdot\nu}_{\text{spatially distributed contribution}}+\underbrace{\int_{\partial M}\mathcal F^\star_\partial(\varphi,\nu)\bcdot \iota_{\partial M}^*\nu}_{\text{boundary contribution}}.
    \end{equation}

    \item[(ii)] For $V^{\dagger}=\Omega^{m-k}(M,E^*)\times\Omega^{m-k-1}(\partial M,E^*)$, let $(\varphi,\nu,\alpha,\alpha_\partial):[t_0,t_1]\to TV\oplus T^\dagger V$ be a solution of \eqref{ForcedLagDirac_VectBundleValuedkforms_Dagger}. Then:
    \begin{equation}\label{EB_ii_local}
    \frac{\partial }{\partial t}\mathscr{E}(\varphi,\nu,\nd\varphi)=-{\rm d}\left(\nu\bwedge\frac{\partialnew\mathscr L}{\partialnew\zeta}(\varphi,\nu,\nd\varphi)\right)+\nu\bwedge\mathcal F^\dagger(\varphi,\nu),
    \end{equation}
    and
    \begin{equation}\label{EB_ii}
    \frac{d}{dt} \int_M \mathscr{E}(\varphi,\nu,\nd\varphi)=\underbrace{\int_M\nu\bwedge\mathcal F^\dagger(\varphi,\nu)}_{\text{spatially distributed contribution}}+\underbrace{\int_{\partial M}\iota_{\partial M}^*\nu\bwedge\mathcal F^\dagger_\partial(\varphi,\nu)}_{\text{boundary contribution}}.
    \end{equation}
\end{itemize}
\end{proposition}

\begin{proof}
For the case $V^{\star}=\Lambda_m^k(M, E^*)\times\Lambda_{m-1}^k(\partial M, E^*)$, 
by taking the time derivative of $\mathscr{E}(\varphi,\nu,\nabla\varphi)$ along the solution curve that satisfies the system equations \eqref{EqMotionForcedLagDirac_VectBundleValuedkforms_Star},  it follows
\begin{equation*}
\begin{split}
\frac{\partial }{\partial t}\mathscr{E}(\varphi,\nu,\nd\varphi)
&=\left(\frac{\partial }{\partial t}\frac{\partial\mathscr{L}}{\partial\nu}\right)\bcdot \nu+\frac{\partial\mathscr{L}}{\partial\nu} \bcdot \dot\nu-\left(\frac{\partial\mathscr{L}}{\partial\varphi}\bcdot  \dot\varphi+\frac{\partial\mathscr{L}}{\partial\nu}\bcdot  \dot\nu+\frac{\partial\mathscr{L}}{\partial\zeta}\bcdot \nd\nu\right)\\
&=\left(-(-1)^k\operatorname{div}^{ \nabla ^*}\frac{\partial\mathscr{L}}{\partial\zeta}+\mathcal F^{\star}(\varphi,\nu)\right)\bcdot \nu-\frac{\partial\mathscr{L}}{\partial\zeta}\bcdot \nd\nu\\
&=-\operatorname{div}\left(\frac{\partial\mathscr L}{\partial\zeta}\bcdot\nu\right)+\mathcal F^\star(\varphi,\nu)\bcdot\nu,
\end{split}
\end{equation*}
where we used Proposition \ref{prop:covariantdivergence}. 
By integrating on $M$ and using integration by parts, we obtain
\begin{equation*}
\begin{split}
\frac{d}{dt} \int_M \mathscr{E}(\varphi,\nu, \nd\varphi)&=\int_{M}\left[-\operatorname{div}\left(\frac{\partial\mathscr{L}}{\partial\zeta}\bcdot \nu\right)+\mathcal F^{\star}(\varphi,\nu)\bcdot\nu\right]\\
&=\int_{M}\mathcal F^{\star}(\varphi,\nu)\bcdot \nu-\int_{\partial M}\iota_{\partial M}^*\left(\tr\frac{\partial\mathscr L}{\partial\zeta}\right)\bcdot \iota_{\partial M}^*\nu\\
&=\int_{M}\mathcal F^{\star}(\varphi,\nu)\bcdot \nu + \int_{\partial M}\mathcal F^\star_{\partial}(\varphi,\nu)\bcdot \iota_{\partial M}^*\nu.
\end{split}
\end{equation*}
Analogously, for the case $V^{\dagger}=\Omega^{m-k}(M, E^*)\times\Omega^{m-k-1}(\partial M, E^*)$, by taking the time derivative of $\mathscr{E}(\varphi,\nu,\nd\varphi)$ along the solution curve that satisfies the system equations \eqref{EqMotionForcedLagDirac_VectBundleValuedkforms_Dagger},  it follows
\begin{equation*}
\begin{split}
\frac{\partial }{\partial t}\mathscr{E}(\varphi,\nu,\nd\varphi)
&=\nu\bwedge\frac{\partial }{\partial t}\frac{\partialnew\mathscr{L}}{\partialnew\nu}+\dot\nu\bwedge\frac{\partialnew\mathscr{L}}{\partialnew\nu}-\left(\dot\varphi\bwedge\frac{\partialnew\mathscr{L}}{\partialnew\varphi}+\dot\nu\bwedge\frac{\partialnew\mathscr{L}}{\partialnew\nu}+\nd\nu\bwedge\frac{\partialnew\mathscr{L}}{\partialnew\zeta}\right)\\
&=\nu\bwedge\left(- (-1)^k\nd^*\frac{\partialnew\mathscr L}{\partialnew\zeta}+\mathcal F^\dagger(\varphi,\nu)\right)-\nd\nu\bwedge\frac{\partialnew\mathscr{L}}{\partialnew\zeta}\\
&=-{\rm d}\left(\nu\bwedge\frac{\partialnew\mathscr L}{\partialnew\zeta}\right)+\nu\bwedge\mathcal F^\dagger(\varphi,\nu),
\end{split}
\end{equation*}
where we have used the Leibniz rule for the covariant exterior derivative:
\begin{equation*}
{\rm d}\left(\nu\bwedge\frac{\partialnew\mathscr L}{\partialnew\zeta}\right)={\rm d}^\nabla\nu\bwedge\frac{\partialnew\mathscr L}{\partialnew\zeta}+(-1)^k\,\nu\bwedge\left({\rm d} ^{ \nabla ^*}\frac{\partialnew\mathscr L}{\partialnew\zeta}\right).
\end{equation*}    
By integrating on $M$ and using integration by parts, we conclude:
\begin{align*}
\frac{d}{dt} \int_M \mathscr{E}(\varphi,\nu,\nd\varphi) & =\int_{M}\left[-{\rm d}\left(\nu\bwedge\frac{\partialnew\mathscr L}{\partialnew\zeta}\right)+\nu\bwedge\mathcal F^\dagger(\varphi,\nu)\right]\\
& =\int_{M}\nu\bwedge\mathcal F^\dagger(\varphi,\nu)-\int_{\partial M}\iota_{\partial M}^*\nu\bwedge\iota_{\partial M}^*\left(\frac{\partialnew\mathscr L}{\partialnew\zeta}\right)\\
& =\int_{M}\nu\bwedge\mathcal F^\dagger(\varphi,\nu)+\int_{\partial M}\iota_{\partial M}^*\nu\bwedge\mathcal F^\dagger_\partial(\varphi,\nu).    
\end{align*}
\end{proof}

\medskip

\begin{remark}[Expressions of the energy flux density]\label{energy_flux_density}{\rm The local energy balance equations allow for the identification of the general expression of the energy flux density in terms of $\mathscr{L}$, see \eqref{EB_i_local} and \eqref{EB_ii_local}. It is given by the vector field density  $S=\frac{\partial\mathscr L}{\partial\zeta}\bcdot\nu$ or the $(m-1)$-form $\mathsf{S}=\nu\bwedge\frac{\partialnew\mathscr L}{\partialnew\zeta}$, depending on the chosen restricted dual.}
\end{remark}

\begin{remark}[Interior and boundary contributions]\label{energy_balance_rmk}\rm
The global energy balance equations \eqref{EB_i} and \eqref{EB_ii} show the explicit form of each contribution to the energy change, both within the entire domain and through its boundary.
These energy balance equations  can be regarded as Lagrangian analogs of the energy balance equation for distributed (infinite dimensional) port-Hamiltonians systems with external variables; see equation (48) in \cite{VdSMa2002}, extended here to systems on bundle-valued differential forms.

Note that $\mathcal F_\partial^\star\in \Lambda^k_{m-1}(\partial M, E^*)$ and $\iota_{\partial M}^*\nu\in \Omega^k(\partial M, E)$ are respectively understood as the external effort and flow variables on the boundary $\partial M$, while $\mathcal F^\star\in \Lambda^k_m(M, E^*)$ and $\nu\in \Omega^k(M,E)$ are the external effort and flow variables on $M$. Their associated interior and boundary power densities are obtained via the contraction operation, such as in $\mathcal{F}^\star\bcdot\nu\in \Omega^m(M)$ and $\mathcal{F}^\star_\partial \bcdot \iota_{\partial M}^*\nu\in \Omega^{m-1}(\partial M)$.

Similarly, when the restricted dual $V^\dagger$ is used, the external effort variables on the boundary $\partial M$, resp., on $M$ are given as $\mathcal F_\partial^\dagger\in \Omega^{m-k-1}(\partial M, E^*)$, resp., $\mathcal F^\dagger\in \Omega^{m-k}( M, E^*)$, while the flow variables are unchanged. The associated interior and boundary power densities are now obtained via the wedge operation, such as in $\nu \bwedge\mathcal{F}^\dagger\in \Omega^m(M)$ and $\iota_{\partial M}^*\nu\bwedge\mathcal{F}^\dagger_\partial\in \Omega^{m-1}(\partial M)$.

The choice of the most appropriate form of power density, hence of restricted dual, depends on the problem considered.
\end{remark}

\subsection{Variational principles for Lagrange--Dirac dynamical systems on the space of bundle valued forms}\label{VP}

As commented earlier, the Lagrange--Dirac approach that we develop keeps intact the canonical geometric structure of finite dimensional systems. As a consequence, it also has naturally associated variational formulations.

A first variational formulation is the \textit{Lagrange--d'Alembert principle}, see \eqref{LDA} for the finite dimensional setting, which is a critical action principle for curves $q(t)\in Q$ in the configuration space, here represented by the bundle-valued form $\varphi(t)\in V$. It underlies the Lagrange--d'Alembert equations for $q(t)$ which are issued from the Lagrange--Dirac dynamical system by expressing the momentum and velocity variables $p(t)$ and $v(t)$, here given by $\alpha(t)$, $\alpha_\partial(t)$, and $\nu(t)$, in terms of the configuration curve only. This principle is stated in \S\ref{paragr_LDA}.

It is also possible to develop a critical action principle directly for curves in the Pontryagin bundle, the \textit{Lagrange--d'Alembert--Pontryagin principle}, see \eqref{LDAP} for the finite dimensional setting, which characterizes the solution curves $(q(t),v(t),p(t))\in TQ\oplus T^*Q$ of the Lagrange--Dirac dynamical system on the  Pontryagin bundle, here represented by $(\varphi(t), \nu(t), \alpha(t),\alpha_\partial (t))\in T V\oplus T^\star V$ (or $T V\oplus T^\dagger V$). This principle is stated in \S\ref{paragr_LDAP}.

\subsubsection{The Lagrange--d'Alembert principle}\label{paragr_LDA}

As earlier, we fix a linear connection $\nabla$ on $\pi_{E,M}:E\to M$ and consider a Lagrangian $L_{\nabla}:TV\to\mathbb R$  defined through a density, as in \eqref{eq:lagrangianvectorbundlekforms}. We aim to find here a variational formulation underlying the
\textit{Lagrange--d'Alembert equations}:
\begin{equation}\label{LdA_BundleValuedkforms_star}
\left\{\begin{array}{l}
\displaystyle\frac{\partial }{\partial t}\frac{\partial\mathscr L}{\partial\nu}(\varphi,\dot\varphi,\nd\varphi)=\frac{\partial\mathscr L}{\partial\varphi}(\varphi,\dot\varphi,\nd\varphi)-(-1)^{k}\operatorname{div}^{ \nabla ^*}\frac{\partial\mathscr L}{\partial\zeta}(\varphi,\dot\varphi,\nd\varphi)+\mathcal F^\star(\varphi,\dot{\varphi}),\vspace{2mm}\\ 
\displaystyle\mathcal F_\partial^\star(\varphi, \dot\varphi)=-\iota_{\partial M}^*\left(\tr\frac{\partial\mathscr L}{\partial\zeta}(\varphi,\dot\varphi,\nd\varphi)\right),
\end{array}\right.
\end{equation}
and
\begin{equation}\label{LdA_BundleValuedkforms_dagger}
\left\{\begin{array}{l}
\displaystyle\frac{\partial }{\partial t}\frac{\partialnew\mathscr L}{\partialnew\nu}(\varphi,\dot\varphi,\nd\varphi)=\frac{\partialnew\mathscr L}{\partialnew\varphi}(\varphi,\dot\varphi,\nd\varphi)-(-1)^k\nd^* \left( \frac{\partialnew\mathscr L}{\partialnew\zeta}(\varphi,\dot\varphi,\nd\varphi) \right) +\mathcal F^{\dagger}(\varphi,\dot{\varphi}),\vspace{2mm}\\
\displaystyle\mathcal F^{\dagger}_\partial(\varphi, \dot\varphi)=-\iota_{\partial M}^*\left(\frac{\partialnew\mathscr L}{\partialnew\zeta}(\varphi,\dot\varphi,\nd\varphi)\right),
\end{array}\right.
\end{equation}
which are found from \eqref{EqMotionForcedLagDirac_VectBundleValuedkforms_Star} and \eqref{EqMotionForcedLagDirac_VectBundleValuedkforms_Dagger} by eliminating the variables $\alpha(t)$, $\alpha_\partial(t)$, and $\nu(t)$.

\begin{proposition}\label{HamPrinLag_BundleValuedkForms}\rm
Associated to each choice of the restricted dual space, we have the following \emph{Lagrange--d'Alembert variational principles}:
\begin{itemize}
    \item[\bf (i)] For $V^\star=\Lambda_m^k(M, E^*)\times\Lambda_{m-1}^k(\partial M, E^*)$, a curve $\varphi: [t_{0},t_{1}] \to\Omega^{k}(M, E)$ is critical for the Lagrange--d'Alembert action functional, i.e.,
    \begin{equation*}
    \begin{split}
    &\delta\int_{t_{0}}^{t_{1}} L_{\nabla}(\varphi, \dot{\varphi})\,dt+\int_{t_{0}}^{t_{1}} \blangle F^{\star}(\varphi, \dot{\varphi}), \delta \varphi \brangle_{\star}\,dt\\
    &=\delta\int_{t_{0}}^{t_{1}}\left( \int_{M}\mathscr L(\varphi,\dot\varphi,\nd\varphi)\right)dt+\int_{t_{0}}^{t_{1}}\left(\int_{M}\mathcal F^{\star}(\varphi,\dot{\varphi})\bcdot \delta{\varphi}+\int_{\partial M}\mathcal F^{\star}_\partial(\varphi,\dot{\varphi})\bcdot \iota^*_{\partial M}\delta{\varphi}\right) dt=0,
    \end{split}
    \end{equation*}
    for free variations $\delta\varphi$ vanishing at $t=t_0, t_1$, if and only if it satisfies the \emph{Lagrange--d'Alembert equations} \eqref{LdA_BundleValuedkforms_star}.
    \item[\bf (ii)] For $V^{\dagger}=\Omega^{m-k}(M, E^*)\times\Omega^{m-k-1}(\partial M, E^*)$, a curve $\varphi: [t_{0},t_{1}] \to\Omega^{k}(M, E)$ is critical for the Lagrange--d'Alembert action functional, i.e.,
    \begin{equation*}
    \begin{split}
    &\delta\int_{t_{0}}^{t_{1}} L_{\nabla}(\varphi, \dot{\varphi})\,dt+\int_{t_{0}}^{t_{1}} \blangle F^{\dagger}(\varphi, \dot{\varphi}), \delta \varphi\brangle_{\dagger} dt\\
    &=\delta\int_{t_{0}}^{t_{1}}\left( \int_{M}\mathscr L(\varphi,\dot\varphi,\nd\varphi)\right)\,dt+\int_{t_{0}}^{t_{1}}\left(\int_{M}\delta{\varphi}\bwedge\mathcal F^{\dagger}(\varphi,\dot{\varphi})+\int_{\partial M}\delta{\varphi}\bwedge\mathcal F^{\dagger}_\partial(\varphi,\dot{\varphi})\right) \,dt=0,
    \end{split}
    \end{equation*}
    for free variations $\delta\varphi$ vanishing at $t=t_0, t_1$, if and only if it satisfies the \emph{Lagrange--d'Alembert equations} \eqref{LdA_BundleValuedkforms_dagger}.
\end{itemize}
\end{proposition}

\begin{proof}
To prove {\bf (i)} we compute the critical condition as follows:
\begin{align*}
0 & =\int_{t_0}^{t_1}\int_M\left(\frac{\partial\mathscr L}{\partial\varphi}\bcdot\delta\varphi+\frac{\partial\mathscr L}{\partial\nu}\bcdot\delta\dot\varphi+\frac{\partial\mathscr L}{\partial\zeta}\bcdot\delta\nd\varphi\right)dt+\int_{t_0}^{t_1}\langle F^\star(\varphi,\dot\varphi),\delta\varphi\rangle_\star\\
& =\int_{t_0}^{t_1}\bigg[\int_M\left(\frac{\partial\mathscr L}{\partial\varphi}-\frac{\partial}{\partial t}\frac{\partial\mathscr L}{\partial\nu}-(-1)^k{\rm div}^{\nabla^*}\frac{\partial\mathscr L}{\partial\zeta}+\mathcal F^\star(\varphi,\dot\varphi)\right)\bcdot\delta\varphi\\
& \qquad +\int_{\partial M}\left(\iota_{\partial M}^*\left(\tr\frac{\partial\mathscr L}{\partial\zeta}\right)+\mathcal F_\partial^\star(\varphi,\dot\varphi)\right)\bcdot \iota^*_{\partial M}\delta\varphi\bigg]dt+\left[\int_M\frac{\partial\mathscr L}{\partial\nu}\bcdot \delta\varphi\right]_{t=t_0}^{t=t_1},
\end{align*}
where we have used Proposition \ref{prop:covariantdivergence}. Since the variations $\delta\varphi$ are free and vanish at $t=t_0,t_1$, we get the desired Lagrange--d'Alembert equations. The proof of {\bf (ii)} is similar by using the Leibniz rule for the covariant exterior derivative.
\end{proof}

\subsubsection{The Lagrange--d'Alembert--Pontryagin principle}\label{paragr_LDAP}

As in Definition \ref{def:externalforces_VectBundleValuedkforms}, there exist two formulations of the Lagrange--d'Alembert--Pontryagin principle, each one corresponding to a choice of the restricted dual space.

\begin{proposition}\rm\label{HPPrinciple_OmegakME}
Associated to each choice of the restricted dual space of $V=\Omega^k(M,E)$, we have a corresponding Lagrange--d'Alembert--Pontryagin variational principle:
\begin{itemize}
    \item[(i)] For $V^{\star}=\Lambda_m^k(M,E^*)\times\Lambda_{m-1}^k(\partial M,E^*)$, a curve $(\varphi,\nu,\alpha,\alpha_\partial):[t_0,t_1]\to TV\oplus T^\star V$ is critical for the \emph{Lagrange--d'Alembert--Pontryagin action functional}, i.e.,  
    \begin{equation*}
    \begin{split}
    &\delta\int_{t_{0}}^{t_{1}}\bigl( L_{\nabla}(\varphi,\nu)+ \blangle (\alpha, \alpha_{\partial}),\dot\varphi-\nu\brangle_{\star}\bigr)\,dt+\int_{t_{0}}^{t_{1}} \blangle F^\star(\varphi,\dot{\varphi}), \delta{\varphi}\brangle_{\star} dt\\[2mm]
    &=\delta\int_{t_{0}}^{t_{1}}\left(\int_{M}\mathscr L(\varphi,\nu,\nd\varphi)
    +\int_{M} \alpha\bcdot(\dot\varphi-\nu)  + \int_{\partial M} \alpha_{\partial}\bcdot(\dot\varphi-\nu)\right) dt\\[2mm]
    &\hspace{4cm}
    +\int_{t_{0}}^{t_{1}}\left(\int_{M}\mathcal F^\star(\varphi,\dot{\varphi})\bcdot\delta{\varphi}+ \int_{\partial M}\mathcal F^\star_\partial(\varphi,\dot{\varphi})\bcdot \iota^*_{\partial M}\delta{\varphi}\right) dt=0
    \end{split}
    \end{equation*}
    for free variations $\delta\varphi,\delta \nu,\delta\alpha,\delta\alpha_ \partial$ with $\delta\varphi$ vanishing at $t=t_0,t_1$, if and only if it satisfies the \emph{Lagrange--d'Alembert--Pontryagin equations}, which are exactly the equations of motion given in \eqref{EqMotionForcedLagDirac_VectBundleValuedkforms_Star}.
    
    \item[(ii)] For $V^{\dagger}=\Omega^{m-k}(M,E^*)\times\Omega^{m-k-1}(\partial M,E^*)$, a curve $(\varphi,\nu,\alpha,\alpha_\partial):[t_0,t_1]\to TV\oplus T^\dagger V$ is critical for the \emph{Lagrange--d'Alembert--Pontryagin action functional}, i.e., 
    \begin{equation*}
    \begin{split}
    &\delta\int_{t_{0}}^{t_{1}}\bigl( L_{\nabla}(\varphi,\nu)+ \blangle (\alpha, \alpha_{\partial}),\dot\varphi-\nu\brangle_{\dagger}\bigr)\,dt+\int_{t_{0}}^{t_{1}} \blangle F^\dagger(\varphi,\dot{\varphi}), \delta{\varphi}\brangle_{\dagger} dt\\[2mm]
    &=\delta\int_{t_{0}}^{t_{1}}\left(\int_{M}\mathscr L(\varphi,\nu,\nd\varphi)
    +\int_{M} (\dot\varphi-\nu)\bwedge\alpha  + \int_{\partial M} (\dot\varphi-\nu)\bwedge\alpha_{\partial}\right) dt\\[2mm]
    &\hspace{4cm}
    +\int_{t_{0}}^{t_{1}}\left(\int_{M}\delta{\varphi}\bwedge\mathcal F^\dagger(\varphi,\dot{\varphi})+ \int_{\partial M}\iota^*_{\partial M}\delta{\varphi}\bwedge\mathcal F^\dagger_\partial(\varphi,\dot{\varphi})\right) dt=0
    \end{split}
    \end{equation*}
    for free variations $\delta\varphi,\delta \nu,\delta\alpha,\delta\alpha_\partial$ with $\delta\varphi$ vanishing at $t=t_0,t_1$, if and only if it satisfies the \emph{the Lagrange--d'Alembert--Pontryagin equations}, which are exactly the equations of motion given in \eqref{EqMotionForcedLagDirac_VectBundleValuedkforms_Dagger}.
\end{itemize}
\end{proposition}
\begin{proof}
The results are obtained by direct computations for each choice of the restricted dual space, as in the proof of Proposition \ref{HamPrinLag_BundleValuedkForms}.
\end{proof}

\begin{remark}[Lagrange--Dirac dynamical systems on the family of standard forms]\rm
The Lagrange--Dirac theory for systems whose configuration space is the family of (standard) $k$-forms, $V=\Omega^k(M)$, may be recovered as a particular case of the results presented in this section. More specifically, the two choices for the restricted dual boil down to
\begin{equation*}
V^\star=\Lambda_m^k(M)\times\Lambda_{m-1}^k(\partial M)\qquad\text{and}\qquad V^\dagger=\Omega^{m-k}(M)\times\Omega^{m-k-1}(\partial M).
\end{equation*}
The corresponding dual pairings are given by
\begin{equation*}
\left\langle(\alpha,\alpha_\partial),\varphi\right\rangle_\star=\int_{ M}\alpha\cdot\varphi+\int_{\partial M}\alpha_\partial\cdot\iota_{\partial M}^*\varphi,\qquad(\alpha,\alpha_\partial)\in V^\star,~\varphi\in V,
\end{equation*}
and
\begin{equation*}
\left\langle(\alpha,\alpha_\partial),\varphi\right\rangle_\dagger=\int_{ M}\varphi\wedge\alpha+\int_{\partial M}\iota_{\partial M}^*\varphi\wedge\alpha_\partial,\qquad(\alpha,\alpha_\partial)\in V^\dagger,~\varphi\in V,
\end{equation*}
respectively. Note that the second choice corresponds with the dual pairing introduced in Part I of this paper, \cite{GBRAYo2025I}. Given that standard forms can be differentiated without the need of a linear connection, the exterior derivative and the divergence are canonical. Hence, the divergence theorem given in Proposition \ref{prop:covariantdivergence} reads
\begin{equation*}
{\rm div}(\chi\cdot\delta\varphi)=\chi\cdot{\rm d}\delta\varphi+(-1)^k\,{\rm div}(\chi)\cdot\delta\varphi,
\end{equation*}
for each $\varphi\in\Omega^k(M)$ and $\chi\in\Lambda_m^{k+1}(M)$. Lastly, the Lagrangian \eqref{eq:lagrangianvectorbundlekforms} does not depend on the choice of the linear connection, hence given as
\begin{equation*}
L:T\Omega^k(M)\to\mathbb R,\quad(\varphi,\nu)\mapsto L(\varphi,\nu)=\int_M\mathscr L(\varphi,\nu,{\rm d}\varphi).
\end{equation*}
\end{remark}
%%%%%%%

\subsection{Example: matter fields}\label{sec:particlefields}

Let us apply the theory introduced above to matter fields in the space+time decomposition. More specifically, we shall regard such fields as sections of a vector bundle over a spacetime $\overline{M}$ written as $\overline{M}= [t_0,t_1]\times M$.
In Sections \ref{sec:particlefieldsinteraction} and \ref{sec:YMHeqs} below, we will consider specific examples of these fields, such as the Higgs field or the Klein--Gordon field, and we will couple them with gauge fields, such as the Yang--Mills field.

\paragraph{Geometric setting} Let $(M,g)$ be a compact Riemannian manifold and $\pi_{E,M}:E\to M$ be a vector bundle endowed with a (not necessarily positive definite) bundle metric $\kappa:E\times_M E\to\mathbb R$. A fibered inner product on the space of $E$-valued $k$-forms on $M$,
\begin{equation*}
\mathbf g:\big(\textstyle\bigwedge^k T^*M\otimes E\big)\times_M\big(\textstyle\bigwedge^k T^*M\otimes E\big)\to\mathbb R,
\end{equation*}
may be defined as
\begin{equation*}
\mathbf g_x(\alpha_1\otimes\varphi_1,\alpha_2\otimes\varphi_2)=g_x(\alpha_1,\alpha_2)\,\kappa_x(\varphi_1,\varphi_2),
\end{equation*}
for each $x\in M$, $\alpha_i\in\bigwedge^k T_x^*M$ and $\varphi_i\in E_x$, $i=1,2$. 

The musical isomorphisms induced by $g$ are denoted by $\sharp:T^*M\to TM$ and $\flat:TM\to T^*M$. Similarly, the musical isomorphisms induced by $\kappa$ are denoted by $\sharp_\kappa:E^*\to E$ and $\flat_\kappa:E\to E^*$, where $\pi_{E^*,M}:E^*\to M$ is the dual bundle of $\pi_{E,M}$. In turn, the musical isomorphisms induced by $\mathbf g$ are given by
\begin{equation}\label{bold_sharp_flat}
\begin{aligned}
\boldsymbol{\sharp}:\textstyle\bigwedge^k T^*M\otimes E\to\bigwedge^k TM\otimes E^*,\qquad & \alpha\otimes\varphi\mapsto\alpha^\sharp\otimes\varphi^{\flat_\kappa},\\
\boldsymbol{\flat}:\textstyle\bigwedge^k TM\otimes E^*\to\bigwedge^k T^*M\otimes E,\qquad & U\otimes\eta\mapsto U^\flat\otimes\eta^{\sharp_\kappa}.
\end{aligned}
\end{equation}
Lastly, the Hodge star operator on $M$ is trivially extended to $E$-valued and $E^*$-valued forms. Namely, for each $\alpha\in\bigwedge^k T^*M$ and $\varphi\in E$, it is defined as $\boldsymbol\star(\alpha\otimes\varphi)=(\star\alpha)\otimes\varphi$, where $\star:\bigwedge^k T^*M\to\bigwedge^{m-k} T^*M$ denotes the standard Hodge star operator. By definition, it satisfies $\mathbf g(\sigma_1,\sigma_2)\mu_g=\sigma_1^{\flat_\kappa}\bwedge\boldsymbol\star\sigma_2$ for each $x\in M$ and $\sigma_1,\sigma_2\in\bigwedge^k T_x^*M\otimes E_x$, where $\mu_g\in\Omega^m(M)$ is the Riemannian volume form and we note $\sigma_1^{\flat_\kappa} \in \bigwedge^k T_x^*M\otimes E^*_x$. 

\begin{remark}\rm
The Riemannian metric $g$ induces a Riemannian metric on the boundary, $g_\partial=\iota_{\partial M}^*g$. The corresponding musical isomorphisms are denoted by $\sharp_\partial:T^*\partial M\to T\partial M$ and $\flat_\partial:T\partial M\to T^*\partial M$. Similarly, the Hodge star operator and
the Riemannian volume form on the boundary (cf. \cite[Corollary 15.34]{Le2012}) are denoted by $\star_\partial:\Omega^k(\partial M)\to\Omega^{m-k-1}(\partial M)$ and $\mu_g^\partial=\iota_{\partial M}^*(i_n\mu_g)\in\Omega^{m-1}( \partial M)$, respectively, where $n\in\mathfrak X(M)|_{\partial M}$ is the outward pointing, unit, normal vector field on $\partial M$, and $i_U:\Omega^k(M)\to\Omega^{k-1}(M)$ is the left interior multiplication by $U\in\mathfrak X(M)$. Hence, the previous constructions may be extended to forms on the boundary.
\end{remark}

Now consider the extended bundle $\overline\pi_{E,M}:\overline E\to\overline M$, where $\overline M=[t_0,t_1]\times M$ and $\overline E=[t_0,t_1]\times E$. By considering the Minkowski metric on $\overline M$, i.e., $\eta=-dt\otimes dt+g$ (where the necessary pullbacks are omitted for brevity), a fibered inner product on the space of $\overline E$-valued $k$-forms on $\overline M$,
\begin{equation*}
\boldsymbol\eta:\left(\textstyle\bigwedge^k T^*\overline M\otimes\overline E\right)\times_{\overline M}\left(\textstyle\bigwedge^k T^*\overline M\otimes\overline E\right)\to\mathbb R,
\end{equation*}
may be defined as above; namely,
\begin{equation*}
\boldsymbol\eta_{(t,x)}(\overline\alpha_1\otimes\overline\varphi_1,\overline\alpha_2\otimes\overline\varphi_2)=\eta_{(t,x)}(\overline\alpha_1,\overline\alpha_2)\,\kappa_x(\varphi_1,\varphi_2),
\end{equation*}
for each $(t,x)\in\overline M$, $\overline\alpha_i\in\bigwedge^k T_{(t,x)}^*\overline M$ and $\overline\varphi_i=(t,\varphi_i)\in\overline E_{(t,x)}=\{t\}\times E_x$, $i=1,2$.
\medskip

In particular, in the case $k=1$ we have the following two situations:
\begin{enumerate}
    \item If $\overline\alpha_1=dt$, then
    \begin{equation*}
    \boldsymbol\eta_{(t,x)}\left(\overline\alpha_1\otimes\overline\varphi_1,\overline\alpha_2\otimes\overline\varphi_2\right)=\underset{-1}{\underbrace{\eta_{(t,x)}(dt,dt)}}\,\kappa_x(\varphi_1,\varphi_2)=-\mathbf g_x\left(\varphi_1,\varphi_2\right).
    \end{equation*}
    \item If $\overline\alpha_i=\alpha_i\in T_x^*M$, $i=1,2$, then
    \begin{equation*}
    \boldsymbol\eta_{(t,x)}\left(\overline\alpha_1\otimes\overline\varphi_1,\overline\alpha_2\otimes\overline\varphi_2\right)=\underset{g_x(\alpha_1,\alpha_2)}{\underbrace{\eta_{(t,x)}(\alpha_1,\alpha_2)}}\,\kappa_x(\varphi_1,\varphi_2)=\mathbf g_x\left(\alpha_1\otimes\varphi_1,\alpha_2\otimes\varphi_2\right).
    \end{equation*}
\end{enumerate}

Given a linear connection  $\nabla$ on $\pi_{E,M}:E\to M$, define the linear connection $\overline\nabla$ induced on $\overline\pi_{E,M}:\overline E\to\overline M$ as
\begin{equation}\label{varphi_bar}
\overline\nabla\overline\varphi(t,x)=\left(t,\nabla\varphi(t,x)+\dot\varphi(t,x)\,dt\right),\quad\overline\varphi(t,x)=(t,\varphi(t,x))\in\Gamma\left(\overline E\to\overline M\right)=\Omega^0\left(\overline M,\overline E\right).
\end{equation}

\paragraph{Action functional and Lagrangian densities} The action functional of a \emph{matter field} takes the general form
\begin{equation}\label{eq:lagrangiansigma}
\mathcal A:\Omega^0\left(\overline M,\overline E\right)\to\mathbb R,\quad\overline\varphi\mapsto\mathcal A(\overline\varphi)=\int_{\overline M}
\mathfrak{L}\big(\overline\varphi,{\rm d}^{\overline\nabla}\overline\varphi\big),
\end{equation}  
where the spacetime Lagrangian density $\mathfrak{L}$ is a bundle map
\[
\mathfrak{L}: \textstyle\left(\bigwedge^0 T^*\overline{M}\otimes \overline{E}\right)\times_{\overline{M}}\left(\bigwedge^1T^*\overline{M}\otimes \overline{E}\right)\to\bigwedge^{m+1} T^*\overline{M}.
\]
If $\mathfrak{L}$ does not depend explicitly on time $t$, then it gives rise to a Lagrangian density of the form considered in \eqref{Lagrangian_density_bundle}, namely,
\begin{equation}\label{Lagrangian_particle}
\textstyle\mathscr \!\!\!\mathscr{L}:\left(\bigwedge^0 T^*M\otimes E\right)\times_M\left(\bigwedge^0 T^*M\otimes E\right)\times_M\left(\bigwedge^{1}T^*M\otimes E\right)\to\bigwedge^m T^*M,
\end{equation}
via the relation
\[
\mathfrak{L\big(\overline\varphi,{\rm d}^{\overline\nabla}\overline\varphi\big)}= \mathscr{L}(\varphi, \dot\varphi, {\rm d}^{\nabla}\varphi)\,dt.
\]
This is shown by defining $\mathscr{L}$ in terms of $\mathfrak{L}$ as
\[
\mathscr{L}(\varphi, \dot\varphi, {\rm d}^{\nabla}\varphi)=i_{\partial_t}\left(\mathfrak{L\big(\overline\varphi,{\rm d}^{\overline\nabla}\overline\varphi\big)}\right)
\]
and using \eqref{varphi_bar}.
Thanks to this, the action functional \eqref{eq:lagrangiansigma} can be written in the form considered earlier in terms of $L_\nabla$ and $\mathscr{L}$ as
\begin{equation}\label{3_1}
\mathcal A(\overline\varphi)=\int_{t_0}^{t_1}L_{\nabla}(\varphi, \dot \varphi) dt=\int_{t_0}^{t_1}\int _M \mathscr{L}(\varphi, \dot\varphi, {\rm d}^\nabla\varphi) dt.
\end{equation} 

For concreteness, let us consider the spacetime Lagrangian density
\begin{equation}\label{concreteness}
\mathfrak{L}(\overline{\varphi}, {\rm d}^{\overline{\nabla}}\overline{\varphi})=-\Big(\frac{1}{2}\boldsymbol\eta\left({\rm d}^{\overline\nabla}\overline\varphi,{\rm d}^{\overline\nabla}\overline\varphi\big)+\mathbf V(\overline\varphi)\right)dt\wedge\mu_g,
\end{equation} 
with $\mathbf V:\overline E\to\mathbb R$  a potential function assumed to be time independent, i.e., $\mathbf V(t,\varphi_x)=\mathbf V(\varphi_x)$ for each $(t,x)\in\overline M$ and $\varphi_x\in E_x$. From the definitions of $\boldsymbol\eta$ and $\overline\nabla$ given above, it is easy to check that
\begin{equation*}
\boldsymbol\eta\big({\rm d}^{\overline\nabla}\overline\varphi,{\rm d}^{\overline\nabla}\overline\varphi\big)=-\mathbf g(\dot\varphi, \dot  \varphi )+\mathbf g({\rm d}^\nabla\varphi,{\rm d}^\nabla\varphi).
\end{equation*}
Note that since $\overline{\varphi}$ and $\varphi$ are $0$-forms, we have ${\rm d}^{\overline{\nabla}}\overline{\varphi}=\overline\nabla\overline{\varphi}$ and ${\rm d}^\nabla\varphi=\nabla\varphi$, as given in \eqref{varphi_bar}. We shall however keep the exterior differential notation below.

Hence the Lagrangian density \eqref{Lagrangian_particle} reads
\begin{align}\label{eq:lagrangiandensitysigma}
\mathscr L(\varphi_x,\nu_x,\zeta_x) & =\left(\frac{1}{2}\mathbf g(\nu_x,\nu_x)-\frac{1}{2}\mathbf g(\zeta_x,\zeta_x)-\mathbf V(\varphi_x)\right)\mu_g\\\nonumber
& =\frac{1}{2}\left(\nu_x^{\flat_\kappa}\bwedge\boldsymbol\star\nu_x-\zeta_x^{\flat_\kappa}\bwedge\boldsymbol\star\zeta_x\right)-\star\mathbf V(\varphi_x),
\end{align}
for each $(\varphi_x,\nu_x,\zeta_x)\in E_x\times E_x\times T_x^*M\otimes E_x$ and $x\in M$.

\paragraph{Equations and boundary conditions} The equations and boundary conditions in intrinsic form are given in \eqref{EqMotionForcedLagDirac_VectBundleValuedkforms_Star} and \eqref{EqMotionForcedLagDirac_VectBundleValuedkforms_Dagger}, with $k=0$. Their local expression, as stated in Remark \ref{remark:localexpression}, simplify considerably in the case $k=0$:
\begin{equation}\label{EqMotionForcedLagDirac_VectBundleValuedkforms_k1}
\left\{
\begin{array}{ll}
\dot\varphi^a=\nu^a,&\vspace{4mm}\\
\displaystyle\alpha_a=\frac{\partial\mathscr L}{\partial\nu^a},\; & \displaystyle\dot\alpha_a=\frac{\partial\mathscr L}{\partial\varphi^a}-\Big(\partial_j \frac{\partial\mathscr L}{\partial\zeta^a_j} - \Gamma^b_{j,a} \frac{\partial\mathscr L}{\partial\zeta^b_j} \Big) +\mathcal F_a,\vspace{2mm}\\
(\alpha_\partial)_a=0, & \displaystyle(\dot\alpha_\partial)_a=\left.\frac{\partial\mathscr L}{\partial\zeta_m^a}\right|_{x^m=0}+(\mathcal F_\partial)_a.
\end{array}\right.
\end{equation}

Let us focus on the Lagrangian density given in \eqref{eq:lagrangiandensitysigma}. The partial derivative of the potential is defined similarly to the functional derivatives of the Lagrangian density. More specifically, it is the bundle morphism $\partial\mathbf V/\partial\varphi:E\to E^*$ given by
\begin{equation*}
\frac{\partial\mathbf V}{\partial\varphi}(\varphi_x)\cdot\delta\varphi_x=\left.\frac{d}{d\epsilon}\right|_{\epsilon=0}\mathbf V(\varphi_x+\epsilon\delta\varphi_x),\qquad\varphi_x,\delta\varphi_x\in E_x,~x\in M.
\end{equation*}
Moreover, we denote $\operatorname{grad}_\kappa\mathbf V=\sharp_\kappa\circ\partial\mathbf V/\partial\varphi:E\to E$. The next result follows from a straightforward computation.

\begin{lemma}\rm
Let $V=\Omega^0(M,E)$ and $(\varphi,\dot\varphi)\in TV$. By denoting $\zeta={\rm d}^\nabla\varphi\in\Omega^1(M,E)$, the partial derivatives of the Lagrangian density \eqref{eq:lagrangiandensitysigma} evaluated at $(\varphi,\dot\varphi,\zeta)$ are given by
\begin{align*}
& \frac{\partial\mathscr L}{\partial\varphi}=-\frac{\partial\mathbf V}{\partial\varphi}\otimes\mu_g, && \frac{\partial\mathscr L}{\partial\nu}=\dot\varphi^{\boldsymbol\sharp}\otimes\mu_g, && \frac{\partial\mathscr L}{\partial\zeta}=-\left({\rm d}^\nabla\varphi\right)^{\boldsymbol\sharp}\otimes\mu_g,\\
& \frac{\partialnew\mathscr L}{\partialnew\varphi}=-\boldsymbol\star\frac{\partial\mathbf V}{\partial\varphi}, && \frac{\partialnew\mathscr L}{\partialnew\nu}=\boldsymbol\star\dot\varphi^{\flat_\kappa}, && \frac{\partialnew\mathscr L}{\partialnew\zeta}=-\boldsymbol\star\left({\rm d}^\nabla\varphi\right)^{\flat_\kappa}.
\end{align*}
\end{lemma}

Let us consider external body and boundary forces $\mathcal F^\star:TV\to\Lambda_m^0(M,E^*)=\Omega^m(M,E^*)$ and $\mathcal F_\partial^\star:TV\to\Lambda_{m-1}^0(\partial M,E^*)=\Omega^{m-1}(\partial M,E^*)$. The forced Lagrange--Dirac equations on $TV\oplus T^\star V$ \eqref{EqMotionForcedLagDirac_VectBundleValuedkforms_Star} for the matter field Lagrangian density \eqref{eq:lagrangiandensitysigma} read
\begin{equation*}
\left\{\begin{array}{ll}
\dot\varphi=\nu,&\vspace{4mm}\\
\displaystyle\alpha=\dot\varphi^{\boldsymbol\sharp}\otimes\mu_g,\; & \displaystyle\dot\alpha=-\frac{\partial\mathbf V}{\partial\varphi}\otimes\mu_g+\operatorname{div}^{ \nabla ^*} \left(\left({\rm d}^\nabla\varphi\right)^{\boldsymbol\sharp}\otimes\mu_g\right)+\mathcal F^{\star},\vspace{2mm}\\
\alpha_\partial=0, & \displaystyle\dot\alpha_\partial=-\iota_{\partial M}^*\left(\tr\left({\rm d}^\nabla\varphi\right)^{\boldsymbol\sharp}\otimes\mu_g\right)+\mathcal F^{\star}_{\partial}.
\end{array}\right.
\end{equation*}
By eliminating the variables $\nu$, $\alpha$ and $\alpha_\partial$, we obtain the following system:
\begin{equation*}
\left\{
\begin{array}{l}
\displaystyle\vspace{0.2cm}\ddot\varphi^{\boldsymbol\sharp} \otimes \mu _g - \operatorname{div}^{\nabla^*}\left(\left({\rm d}^\nabla\varphi\right)^{\boldsymbol\sharp} \otimes \mu _g\right)+ \frac{\partial\mathbf V}{\partial\varphi}\otimes\mu_g=\mathcal{F} ^\star ,\\
\displaystyle \iota_{ \partial M} ^*\left(\operatorname{tr}\left({\rm d}^\nabla\varphi\right)^{\boldsymbol\sharp}\otimes\mu _g\right)=\mathcal{F} _ \partial ^\star.
\end{array}
\right.
\end{equation*}
Analogously, let $\mathcal F^\dagger:TV\to\Omega^m(M,E^*)$ and $\mathcal F_\partial^\dagger:TV\to\Omega^{m-1}(\partial M,E^*)$ be external body and boundary forces. The forced Lagrange--Dirac equations on $TV\oplus T^\dagger V$ \eqref{EqMotionForcedLagDirac_VectBundleValuedkforms_Dagger} for the matter field Lagrangian density \eqref{eq:lagrangiandensitysigma} read
\begin{equation*}
\left\{\begin{array}{ll}
\dot\varphi=\nu,&\vspace{0.4cm}\\
\displaystyle\alpha=\boldsymbol\star\dot\varphi^{\flat_\kappa},\; & \displaystyle\dot\alpha=-\boldsymbol\star\frac{\partial\mathbf V}{\partial\varphi}+\nd^* \left(\boldsymbol\star\left({\rm d}^\nabla\varphi\right)^{\flat_\kappa}\right)+\mathcal F^\dagger,\vspace{0.2cm}\\
\alpha_\partial=0, & \displaystyle\dot\alpha_\partial=-\iota_{\partial M}^*\left(\boldsymbol\star\left({\rm d}^\nabla\varphi\right)^{\flat_\kappa}\right)+\mathcal F^\dagger_{\partial}.
\end{array}\right.
\end{equation*}
By eliminating the variables $\nu$, $\alpha$ and $\alpha_\partial$, we obtain the following system:
\begin{equation}\label{eq:harmonic}
\left\{
\begin{array}{l}
\displaystyle\vspace{0.2cm}\boldsymbol\star\ddot\varphi^{\flat_\kappa}-\nd^*\left(\boldsymbol\star\left({\rm d}^\nabla\varphi\right)^{\flat_\kappa}\right) + \boldsymbol\star\frac{\partial\mathbf V}{\partial\varphi}=\mathcal{F} ^\dagger,\\
\displaystyle \iota_{ \partial M} ^*\left(\boldsymbol\star\left({\rm d}^\nabla\varphi\right)^{\flat_\kappa}\right)=\mathcal{F} _ \partial ^\dagger.
\end{array}
\right.
\end{equation}

In the case that $\nabla$ is a metric connection with respect to $\kappa$, i.e., ${\rm d}\kappa(\varphi_1,\varphi_2)=\kappa(\nabla\varphi_1,\varphi_2)+\kappa(\varphi_1,\nabla\varphi_2)$ for each $\varphi_1,\varphi_1\in V$, we have $\flat_\kappa\circ{\rm d}^\nabla=\nd^*\circ\flat_\kappa$. In addition, we denote the external currents by $\mathcal F^\dagger=\boldsymbol\star\beth^{\flat_\kappa}$ and $\mathcal F_\partial^\dagger=\boldsymbol\star_\partial\gimel^{\flat_\kappa}$ for some $\beth:TV\to\Omega^0(M,E)$ and $\gimel:TV\to\Omega^0(\partial M,E)$. By using this, the system \eqref{eq:harmonic} reads
\begin{equation}\label{eq:particlefield}
\left\{
\begin{array}{l}
\displaystyle\vspace{0.2cm}\ddot\varphi+\delta^\nabla\left({\rm d}^\nabla\varphi\right)+\operatorname{grad}_\kappa\mathbf V=\beth,\\
\displaystyle \boldsymbol\star_\partial\left(\iota_{ \partial M} ^*\left(\boldsymbol\star{\rm d}^\nabla\varphi\right)\right)=\gimel,
\end{array}
\right.
\end{equation}
where $\delta^\nabla=-\boldsymbol{\star}\circ\nd\circ\boldsymbol\star:\Omega^1(M,E)\to\Omega^0(M,E)$ is the codifferential of $\nd$.

\begin{remark}\rm In local coordinates, the Riemannian metric on $M$ and the bundle metric on $\pi_{E,M}:E\to M$ are given by $g=g_{ij}\,dx^i\otimes dx^j$ and $\kappa=\kappa_{ab}\,B^a\otimes B^b$, respectively, for some $g_{ij},\kappa_{ab}\in C^\infty(M)$, $1\leq i,j\leq m$, $1\leq a,b\leq n$. Let us denote by $(g^{ij})_{1\leq i,j\leq m}$ and $(\kappa^{ab})_{1\leq a,b\leq n}$ the inverse matrices of $(g_{ij})_{1\leq i,j\leq m}$ and $(\kappa_{ab})_{1\leq a,b\leq n}$, respectively. For simplicity, suppose that $\nabla$ is locally the trivial connection, i.e., $\Gamma_{i,b}^a=0$ for each $1\leq i\leq m$ and $1\leq a,b\leq n$. Hence, \eqref{eq:particlefield} is locally given by:
\begin{equation*}
\left\{\begin{array}{l}
\displaystyle\ddot\varphi^a-\partial_j\left(g^{ij}\,\partial_i\varphi^a\right)+\kappa^{ab}\frac{\partial\mathbf V}{\partial\varphi^b}=\beth^a,\\
\displaystyle \left.g^{im}\,\partial_i\varphi^a\right|_{x^m=0}=\gimel^a,
\end{array}\right.
\end{equation*}
where we have denoted $\beth=\beth^a\,B_a$ and $\gimel=\gimel^a\,B_a$, for some $\beth^a\in C^\infty(M)$ and $\gimel^a\in C^\infty(\partial M)$, $1\leq a\leq n$.
\end{remark}

\paragraph{Energy considerations} The energy density is found as
\begin{align*}
\mathscr E& =\left(\frac{1}{2}\mathbf g(\dot\varphi,\dot\varphi)+\frac{1}{2}\mathbf g({\rm d}^\nabla\varphi,{\rm d}^\nabla\varphi)+\mathbf V(\varphi)\right)\mu_g\\
& =\frac{1}{2}\left(\dot\varphi^{\flat_\kappa}\bwedge\boldsymbol\star\dot\varphi+({\rm d}^\nabla\varphi)^{\flat_\kappa}\bwedge\boldsymbol\star{\rm d}^\nabla\varphi\right)+\mathbf V(\varphi)\mu_g.
\end{align*}

From Proposition \ref{prop:energybalancevectorbundle}, we obtain the local and global energy balances. For $V^\star$, they read
\begin{align*}
\frac{\partial\mathscr E}{\partial t}  =-{\rm div}\,S+\mathcal F^\star\bcdot\dot\varphi,
\end{align*}
and
\begin{align*}
\frac{d}{dt} \int_M \mathscr{E}  =\underbrace{\int_M\,\mathcal F^\star\bcdot\dot\varphi}_{\text{spatially distributed contribution}} +\underbrace{\int_{\partial M}\mathcal F_\partial^\star\bcdot\iota_{\partial M}^*\dot\varphi}_{\text{boundary contribution}},
\end{align*}
where $S=-\dot\varphi\bcdot\left({\rm d}^\nabla\varphi\right)^{\boldsymbol\sharp}\otimes\mu_g\in\Lambda_m^1(M)$. Analogously, for $V^\dagger$, the balances read
\begin{align*}
\frac{\partial\mathscr E}{\partial t}  = -{\rm d}\mathsf S+\dot\varphi\bwedge\boldsymbol\star\beth^{\flat_K},
\end{align*}
and
\begin{align*}
\frac{d}{dt} \int_M \mathscr{E} =\underbrace{\int_M\,\dot\varphi\bwedge\boldsymbol\star\beth^{\flat_\kappa}}_{\text{spatially distributed contribution}} +\underbrace{\int_{\partial M}\big(\iota_{\partial M}^*\dot\varphi\big)\bwedge\boldsymbol\star_\partial\gimel^{\flat_\kappa}}_{\text{boundary contribution}},
\end{align*}
where $\mathsf S=-\dot\varphi\bwedge\boldsymbol\star\left({\rm d}^\nabla\varphi\right)^{\flat_\kappa}\in\Omega^{m-1}(M)$.

%%%%%%%
\section{Non-Abelian gauge theories: Yang--Mills equations}\label{sec:yangmills}

The Lagrange--Dirac theory introduced in the previous section for systems on the family of vector bundle-valued forms may be slightly modified to deal with non-Abelian gauge theories, such as the Yang--Mills theory. In addition, a further extension is presented to account for interaction between gauge fields and matter fields.

\subsection{Geometric setting}\label{sec:geometricsetting_gauge}

Let $G$ be a Lie group, $M$ be a compact manifold, $\pi_{P,M}:P\to M$ be a principal $G$-bundle, and $\tilde{\mathfrak g}=(P\times\mathfrak g)/G\to M$ be the adjoint bundle, where the action of $G$ on its Lie algebra $\mathfrak g$ is given by the adjoint representation. The corresponding bundle of connections is denoted by $\operatorname{Conn}(P)=\left(J^1P\right)/G\to  M$, where $J^1P$ is the first jet bundle of $\pi_{P,M}:P\to M$.

The configuration space for non-Abelian gauge theories is the family of principal connections on $\pi_{P,M}:P\to M$,
\begin{equation*}
\mathcal C(P)=\Gamma(\operatorname{Conn}(P)\to M).    
\end{equation*}
Recall that it is an affine bundle modeled on $T^*M\otimes\tilde{\mathfrak g}\to M$ and there is a bijective correspondence between principal connections on $\pi_{P,M}:P\to M$ and sections of the bundle of connections, which we denote by
\begin{equation*}
\mathcal C(P)\ni A\overset{1:1}{\longleftrightarrow}\mathbf A\in\Omega^1(P,\mathfrak g)\text{ principal connection}.
\end{equation*}

Moreover, since $\mathcal C(P)$ is an affine space modeled on $\Omega^1(M,\tilde{\mathfrak g})$, there is a canonical isomorphism
\begin{equation}\label{eq:trivialization}
T\mathcal C(P)\simeq\mathcal C(P)\times\Omega^1(M,\tilde{\mathfrak g}),\qquad v_A\mapsto\big(A,\dot A\big),
\end{equation}
where $\dot A=\dot{\overline\gamma}(0)$, for $\gamma:(-\epsilon,\epsilon)\to\mathcal C(P)$ a curve such that $\gamma(0)=A$ and $\dot\gamma(0)=v_A$, and $\overline\gamma:(-\epsilon,\epsilon)\to\Omega^1(M,\tilde{\mathfrak g})$ defined as $\overline\gamma(t)=\gamma(t)-A$ for each $t\in(-\epsilon,\epsilon)$. The restricted cotangent bundles are thus defined as
\begin{enumerate}
    \item $T^\star\mathcal C(P)=\mathcal C(P)\times\Lambda_m^1(M,\tilde{\mathfrak g}^*)\times\Lambda_{m-1}^1(\partial M,\tilde{\mathfrak g}^*)$ with the pairing given by
    \begin{equation*}
    \blangle(\varsigma,\varsigma_\partial),\dot A\brangle_\star=\int_M\varsigma\bcdot\dot A+\int_{\partial M}\varsigma_\partial\bcdot\iota_{\partial M}^*\dot A,
    \end{equation*}
    for each $(\varsigma,\varsigma_\partial)\in \Lambda_m^1(M,\tilde{\mathfrak g}^*)\times\Lambda_{m-1}^1(\partial M,\tilde{\mathfrak g}^*)$ and $\dot A\in\Omega^1(M,\tilde{\mathfrak g})$.
    \item $T^\dagger\mathcal C(P)=\mathcal C(P)\times\Omega^{m-1}(M,\tilde{\mathfrak g}^*)\times\Omega^{m-2}(\partial M,\tilde{\mathfrak g}^*)$ with the duality pairing given by
    \begin{equation*}
    \blangle(\varsigma,\varsigma_\partial),\dot A\brangle_\dagger=\int_M\dot A\bwedge\varsigma+\int_{\partial M}\iota_{\partial M}^*\dot A\bwedge\varsigma_\partial,
    \end{equation*}
    for each $(\varsigma,\varsigma_\partial)\in\Omega^{m-1}(M,\tilde{\mathfrak g}^*)\times\Omega^{m-2}(\partial M,\tilde{\mathfrak g}^*)$ and $\dot A\in\Omega^1(M,\tilde{\mathfrak g})$.
\end{enumerate}

Therefore, the restricted iterated bundles, as well as the canonical forms, the Tulczyjew triples and the canonical Dirac structures, may be defined as for the case of vector bundle-valued forms.

The horizontal exterior derivative induced by a principal connection $\mathbf A\in\Omega^1(P,\mathfrak g)$ is denoted by ${\rm d}^{\mathbf A}:\Omega^k(P,\mathfrak g)\to\Omega^{k+1}(P,\mathfrak g)$. In particular, for basic 1-forms of the adjoint type, it reads (cf. \cite[\S 2.3]{GBRa2008})
\begin{equation}\label{eq:horizontalderivativebasic1forms}
{\rm d}^{\mathbf A}\eta={\rm d}\eta+[\mathbf A,\eta]+[\eta,\mathbf A],\qquad\eta\in\overline \Omega^1(P,\mathfrak g).
\end{equation}
Analogously, when regarded as a section $A\in\mathcal C(P)$, the principal connection induces a linear connection on the adjoint bundle as well as an exterior derivative on the space of $\tilde{\mathfrak g}$-valued forms on $M$, which we denote by $\nabla^A$ and ${\rm d}^A:\Omega^k(M,\tilde{\mathfrak g})\to\Omega^{k+1}(M,\tilde{\mathfrak g})$, respectively. In the same vein, the dual connection of $\nabla^A$ and the corresponding covariant exterior derivative are denoted by $\nabla^{A*}$ and ${\rm d}^{A*}:\Omega^k(M,\tilde{\mathfrak g}^*)\to\Omega^{k+1}(M,\tilde{\mathfrak g}^*)$, respectively. Moreover, the divergence of $\nabla^{A*}$ is denoted by $\operatorname{div}^{A*}:\Lambda_m^{k+1}(M,\tilde{\mathfrak g}^*)\to\Lambda_m^k(M,\tilde{\mathfrak g}^*)$.

The curvature of the principal connection $\mathbf A\in\Omega^1(P,\mathfrak g)$ is given by $B_{\mathbf A}={\rm d}^{\mathbf A}\mathbf A\in\overline\Omega^2(P,\mathfrak g)$, where the bar denotes that it is a basic form of the adjoint type. Hence, it may be dropped to a $\tilde{\mathfrak g}$-valued form on the base space, $B_A\in\Omega^2(M,\tilde{\mathfrak g})$ which is known as the \emph{reduced curvature} of $A$. 

\subsection{Non-Abelian gauge fields}\label{sec:nonbabelianfields}

As we shall see later, in the space+time decomposition and after choosing the temporal gauge, the action functional for a non-Abelian gauge theory can be expressed in terms of a Lagrangian density given by a bundle morphism
\begin{equation}\label{eq:lagrangiandensitynonabelian}
\mathscr L_{\rm gau}:\operatorname{Conn}(P)\times_M(T^*M\otimes\tilde{\mathfrak g})\times_M\big(\textstyle\bigwedge^2 T^*M\otimes\tilde{\mathfrak g}\big)\to\textstyle\bigwedge^m T^*M
\end{equation}
covering the identity $\operatorname{id}_M$. The fiber derivative with respect to the first argument,
\begin{equation*}
\frac{\partial\mathscr L_{\rm gau}}{\partial A}:\operatorname{Conn}(P)\times_M(T^*M\otimes\tilde{\mathfrak g})\times_M\big(\textstyle\bigwedge^2 T^*M\otimes\tilde{\mathfrak g}\big)\to TM\otimes\textstyle\bigwedge^m T^*M\otimes\tilde{\mathfrak g}^*,
\end{equation*}
is defined as
\begin{equation*}
\frac{\partial\mathscr L_{\rm gau}}{\partial A}(A_x,\varepsilon_x,\beta_x)\bcdot\delta A_x=\left.\frac{d}{d\epsilon}\right|_{ \epsilon=0}\mathscr L_{\rm gau}(A_x+\epsilon\delta A_x,\varepsilon_x,\beta_x),\qquad\delta A_x\in T_x^*M\otimes\tilde{\mathfrak g}_x,
\end{equation*}
for each $x\in M$ and $(A_x,\varepsilon_x,\beta_x)\in\operatorname{Conn}(P)_x\times\big(T^*M\otimes\tilde{\mathfrak g}\big)_x\times\big(\bigwedge^{2}T^*M\otimes\tilde{\mathfrak g}\big)_x$. The fiber derivatives $\partial\mathscr L_{\rm gau}/\partial\varepsilon $ and $\partial\mathscr L_{\rm gau}/\partial\beta$ are defined analogously. As for the case of vector bundle-valued forms, when working with $T^\dagger\mathcal C(P)$ as the restricted cotangent bundle, we utilize
\begin{equation*}
\frac{\partialnew\mathscr L}{\partialnew A}:{\rm Conn}(P)\times_M(T^*M\otimes\tilde{\mathfrak g})\times_M\big(\textstyle\bigwedge^2 T^*M\otimes\tilde{\mathfrak g}\big)\to\textstyle\bigwedge^{m-1}T^*M\otimes\tilde{\mathfrak g}^*,
\end{equation*}
which is defined as
\begin{equation*}
\delta A_x\bwedge\frac{\partialnew\mathscr L_{\rm gau}}{\partialnew A}(A_x,\varepsilon_x,\beta_x)=\left.\frac{d}{d\epsilon}\right|_{ \epsilon=0}\mathscr L_{\rm gau}(A_x+\epsilon\delta A_x,\varepsilon_x,\beta_x),\qquad\delta A_x\in T_x^*M\otimes\tilde{\mathfrak g}_x.
\end{equation*}
The fiber derivatives $\partialnew\mathscr L_{\rm gau}/\partialnew\varepsilon $ and $\partialnew\mathscr L_{\rm gau}/\partialnew\beta$ are defined analogously. The relation between the two realisations is given by the map $\Phi_E$ introduced in \eqref{eq:PhiE}.

The corresponding Lagrangian is the map
\begin{equation}\label{eq:lagrangiannonabelian}
L_{\rm gau}:T\mathcal C(P)\simeq\mathcal C(P)\times\Omega^1(M,\tilde{\mathfrak g})\to\mathbb R,\quad(A,\varepsilon)\mapsto\int_M\mathscr L_{\rm gau}\left(A,\varepsilon,B_A\right).
\end{equation}

The partial functional derivatives 
\begin{equation*}
\frac{\delta L_{\rm gau}}{\delta A},~\frac{\delta L_{\rm gau}}{\delta\varepsilon}:T\mathcal C(P)\to\Omega^1(M,\tilde{\mathfrak g})',
\end{equation*}
of $L_{\rm gau}$ are defined as
\begin{equation*}
\frac{\delta L_{\rm gau}}{\delta A}(A,\varepsilon)(\delta A)=\left.\frac{d}{d\epsilon}\right|_{\epsilon=0}L_{\rm gau}(A+\epsilon\,\delta A,\varepsilon),\quad\frac{\delta L_{\rm gau}}{\delta\varepsilon}(A,\varepsilon)(\delta\varepsilon)=\left.\frac{d}{d\epsilon}\right|_{\epsilon=0}L_{\rm gau}(A,\varepsilon+\epsilon\,\delta\varepsilon),
\end{equation*}
for each $(A,\varepsilon)\in T\mathcal C(P)$ and $\delta A,\delta\varepsilon\in\Omega^1(M,\tilde{\mathfrak g})$.

\begin{lemma}\label{lemma:partialderivativegauge}\rm
For each $(A,\varepsilon)\in T\mathcal C(P)\simeq\mathcal C(P)\times\Omega^1(M,\tilde{\mathfrak g})$, the fiber derivatives of the Lagrangian \eqref{eq:lagrangiannonabelian} are given as follows:
\begin{enumerate}[(i)]
    \item When working with $T^\star\mathcal C(P)$, they read
    \begin{align*}
    \frac{\delta L_{\rm gau}}{\delta A}(A,\varepsilon) & =\left(\frac{\partial\mathscr L_{\rm gau}}{\partial A}(A,\varepsilon,B_A)+\operatorname{div}^{A*}\left(\frac{\partial\mathscr L_{\rm gau}}{\partial\beta}(A,\varepsilon,B_A)\right),~\iota_{\partial M}^*\left(\tr \frac{\partial\mathscr L_{\rm gau}}{\partial\beta}(A,\varepsilon,B_A)\right)\right),\\
    \frac{\delta L_{\rm gau}}{\delta\varepsilon}(A,\varepsilon) & =\left(\frac{\partial\mathscr L_{\rm gau}}{\partial\varepsilon}(A,\varepsilon,B_A),~0\right).
    \end{align*}
    \item When working with $T^\dagger\mathcal C(P)$, they read
    \begin{align*}
    \frac{\delta L_{\rm gau}}{\delta A}(A,\varepsilon) & =\left(\frac{\partialnew\mathscr L_{\rm gau}}{\partialnew A}(A,\varepsilon,B_A)+{\rm d}^{A*}\left(\frac{\partialnew\mathscr L_{\rm gau}}{\partialnew\beta}(A,\varepsilon,B_A) \right) ,~\iota_{\partial M}^*\left(\frac{\partialnew\mathscr L_{\rm gau}}{\partialnew\beta}(A,\varepsilon,B_A)\right)\right),\\
    \frac{\delta L_{\rm gau}}{\delta\varepsilon}(A,\varepsilon)& =\left(\frac{\partialnew\mathscr L_{\rm gau}}{\partialnew\varepsilon}(A,\varepsilon,B_A),~0\right).
    \end{align*}
\end{enumerate}
\end{lemma}

\begin{proof}
Given $\delta A\in\Omega^1(M,\tilde{\mathfrak g})$, we denote by $\delta\mathbf A\in\overline\Omega^1(P,\mathfrak g)$ the corresponding basic form of the adjoint type. By using the Cartan formula, the curvature of $\mathbf A$ reads $B_{\mathbf A}={\rm d}\mathbf A+[\mathbf A,\mathbf A]$, where $[\mathbf A,\mathbf A](u,v)=[\mathbf A(u),\mathbf A(v)]$ for each $u,v\in\mathfrak X(P)$. This, together with \eqref{eq:horizontalderivativebasic1forms}, yields
\begin{align*}
\left.\frac{d}{d\epsilon}\right|_{\epsilon=0}B_{\mathbf A+\epsilon\delta\mathbf A} & =\left.\frac{d}{d\epsilon}\right|_{\epsilon=0}{\rm d}(\mathbf A+\epsilon\delta\mathbf A)+[\mathbf A+\epsilon\delta\mathbf A,\mathbf A+\epsilon\delta\mathbf A]\\
& ={\rm d}(\delta\mathbf A)+[\mathbf A,\delta\mathbf A]+[\delta\mathbf A,\mathbf A]\\
& ={\rm d}^{\mathbf A}(\delta\mathbf A).
\end{align*}
Subsequently, this drops to $d/d\epsilon|_{\epsilon=0}B_{A+\epsilon\delta A}={\rm d}^A(\delta A)$. By using this and Proposition \ref{prop:covariantdivergence}, we obtain
\begin{align*}
\frac{\delta L_{\rm gau}}{\delta A}(A,\varepsilon)(\delta A) & =\left.\frac{d}{d\epsilon}\right|_{\epsilon=0}\int_M\mathscr L_{\rm gau}(A+\epsilon\delta A,\varepsilon,B_{A+\epsilon\delta A})\\
& =\int_M\left(\frac{\partial\mathscr L_{\rm gau}}{\partial A}\bcdot\delta A+\frac{\partial\mathscr L_{\rm gau}}{\partial\beta}\bcdot{\rm d}^A(\delta A)\right)\\
& =\int_M\left(\frac{\partial\mathscr L_{\rm gau}}{\partial A}+\operatorname{div}^{A*}\left(\frac{\partial\mathscr L_{\rm gau}}{\partial\beta}\right)\right)\bcdot\delta A+\int_{\partial M}\iota_{\partial M}^*\left(\operatorname{tr}\frac{\partial\mathscr L_{\rm gau}}{\partial\beta}\right)\bcdot\iota_{\partial M}^*\delta A.
\end{align*}
The computation of $\delta L_{\rm gau}/\delta\varepsilon$ is straightforward. The second part is a straightforward application of Lemma \ref{lemma:PhiEcontractions}.
\end{proof}

\medskip

From the previous Lemma, the differential
\begin{equation*}
{\rm d}L_{\rm gau}:T\mathcal{C}(P)\to T'(T\mathcal{C}(P)),\quad(A, \varepsilon )\mapsto\left(A, \varepsilon  ,\frac{\delta L_{\rm gau}}{\delta A}(A,\varepsilon ),\frac{\delta L_{\rm gau}}{\delta\varepsilon}(A, \nu )\right)
\end{equation*}
of $L_{\rm gau}$ takes values in the restricted cotangent bundle, i.e., either $T^\star(T\mathcal{C}(P))$ or  $T^\dagger(T\mathcal{C}(P))$. As in the previous case on $V=\Omega^k(M,E)$, see \S\ref{sec:partialderivativesvectorkforms} and \S\ref{sec:LDVectBundleValued_kforms}, we can formulate the Lagrange--Dirac systems with boundary energy flow for this class of Lagrangian.

\medskip
In the present geometric context, an external force with values on $T^\star\mathcal C(P)$ takes the form
\begin{equation*}
\hat F^\star:T\mathcal C(P)\to T^\star\mathcal C(P),\quad(A,\varepsilon)\mapsto\left(A,\hat{\mathcal F}^\star(A,\varepsilon),\hat{\mathcal F}_\partial^\star(A,\varepsilon)\right).
\end{equation*}
In the same vein, an external force with values on $T^\dagger\mathcal C(P)$ reads
\begin{equation*}
\hat F^\dagger:T\mathcal C(P)\to T^\dagger\mathcal C(P),\quad(A,\varepsilon)\mapsto\left(A,\hat{\mathcal F}^\dagger(A,\varepsilon),\hat{\mathcal F}_\partial^\dagger(A,\varepsilon)\right).
\end{equation*}

A slight modification of Proposition \ref{proposition:LDVectBundValued_kforms} yields the dynamical equations for non-Abelian gauge theories in the space+time decomposition.

\begin{proposition}\label{proposition:LDnonabelian}\rm
Let $L_{\rm gau}:T\mathcal C(P)\simeq\mathcal C(P)\times\Omega^1(M,\tilde{\mathfrak g})\to\mathbb R$ be a Lagrangian as in \eqref{eq:lagrangiannonabelian}. Associated to each choice of the restricted dual space, the following statements hold:
\begin{enumerate}[(i)]
    \item The forced Lagrange--Dirac equations for a curve $(A,\varepsilon,\varsigma,\varsigma_\partial):[t_0,t_1] \to T\mathcal C(P)\oplus T^{\star}\mathcal C(P)$ read 
    \begin{equation}\label{LDequations_nonabeliangauge_star}
    \left\{\begin{array}{ll}
    \dot A=\varepsilon,&\vspace{4mm}\\
    \displaystyle\varsigma=\frac{\partial\mathscr L_{\rm gau}}{\partial\varepsilon}( A,\varepsilon,B_A),\; & \displaystyle\dot\varsigma=\frac{\partial\mathscr L_{\rm gau}}{\partial A}( A,\varepsilon,B_A)+\operatorname{div}^{A*} \left( \frac{\partial\mathscr L_{\rm gau}}{\partial\beta}( A,\varepsilon,B_A) \right) +\hat{\mathcal F}^{\star}(A,\varepsilon),\vspace{2mm}\\
    \varsigma_\partial=0, & \displaystyle\dot\varsigma_\partial=\iota_{\partial M}^*\left(\tr \frac{\partial\mathscr L_{\rm gau}}{\partial\beta}( A,\varepsilon,B_A)\right)+\hat{\mathcal F}^{\star}_{\partial}(A,\varepsilon).
    \end{array}\right.
    \end{equation}
    \medskip
    \item 
    The forced Lagrange--Dirac equations for a curve $(A,\varepsilon,\varsigma,\varsigma_\partial): [t_0,t_1]\to T\mathcal C(P)\oplus T^{\dagger}\mathcal C(P)$ read
    \begin{equation}\label{LDequations_nonabeliangauge_dagger}
    \left\{\begin{array}{ll}
    \dot A=\varepsilon,&\vspace{0.4cm}\\
    \displaystyle\varsigma=\frac{\partialnew\mathscr L_{\rm gau}}{\partialnew\varepsilon}( A,\varepsilon,B_A),\; & \displaystyle\dot\varsigma=\frac{\partialnew\mathscr L_{\rm gau}}{\partialnew A}( A,\varepsilon,B_A)+{\rm d}^{A*}\left( \frac{\partialnew\mathscr L_{\rm gau}}{\partialnew\beta}( A,\varepsilon,B_A) \right) +\hat{\mathcal F}^\dagger(A,\varepsilon),\vspace{0.2cm}\\
    \varsigma_\partial=0, & \displaystyle\dot\varsigma_\partial=\iota_{\partial M}^*\left(\frac{\partialnew\mathscr L_{\rm gau}}{\partialnew\beta}( A,\varepsilon,B_A)\right)+\hat{\mathcal F}_\partial^\dagger(A,\varepsilon).
    \end{array}\right.
    \end{equation}
\end{enumerate}
\end{proposition}

\begin{remark}\rm
Given $A:[t_0,t_1]\to\mathcal C(P)$, we utilize \eqref{eq:trivialization} to identify $dA/dt\simeq(A,\dot A):[t_0,t_1]\to T\mathcal C(P)\simeq\mathcal C(P)\times\Omega^1(M,\tilde{\mathfrak g})$.
\end{remark}

Analogous to vector bundle-valued forms, the \emph{energy density} for non-Abelian gauge theories is given by
\begin{equation}\label{energy_gau}
\begin{aligned}
\mathscr{E}_{\rm gau}(A_x,\varepsilon_x,\beta_x) & =\frac{\partial\mathscr{L}_{\rm gau}}{\partial\varepsilon}(A_x,\varepsilon_x,\beta_x)\bcdot\varepsilon_x-\mathscr{L}_{\rm gau}(A_x,\varepsilon_x,\beta_x)\\
& =\varepsilon_x\bwedge\frac{\partialnew\mathscr{L}_{\rm gau}}{\partialnew\varepsilon}(A_x,\varepsilon_x,\beta_x)-\mathscr{L}_{\rm gau}(A_x,\varepsilon_x,\beta_x),
\end{aligned}
\end{equation}
for each $(A_x,\varepsilon_x,\beta_x)\in\operatorname{Conn}(P)_x\times(T^*M\otimes\tilde{\mathfrak g})_x\times\big(\bigwedge^2 T^*M\otimes\tilde{\mathfrak g}\big)_x$. The proof of the following result is similar to the one of Proposition \ref{prop:energybalancevectorbundle}.

\begin{proposition}[Energy balance]\rm\label{prop:energybalancegauge}
Associated to each choice of the restricted dual, we have the following energy balance along the solutions of the forced Lagrange--Dirac equations:
\begin{itemize}
    \item[(i)] For $T^\star\mathcal C(P)=\mathcal C(P)\times\Lambda_m^1(M,\tilde{\mathfrak g}^*)\times\Lambda_{m-1}^1(\partial M,\tilde{\mathfrak g}^*)$, let $(A,\varepsilon,\varsigma,\varsigma_\partial):[t_0,t_1] \to T\mathcal C(P)\oplus T^{\star}\mathcal C(P)$ be a solution of \eqref{LDequations_nonabeliangauge_star}. Then:
    \begin{align*}
    \frac{\partial }{\partial t}\mathscr{E}_{\rm gau}(A,\varepsilon,B_A)=-\operatorname{div}\left(\frac{\partial\mathscr L_{\rm gau}}{\partial\beta}(A,\varepsilon,B_A)\bcdot\varepsilon\right)+\hat{\mathcal F}^\star(A,\varepsilon)\bcdot\varepsilon,
    \end{align*}
    and
    \begin{align*}
    \frac{d}{dt} \int_M \mathscr{E}_{\rm gau}(A,\varepsilon,B_A)=\underbrace{\int_M\hat{\mathcal F}^\star(A,\varepsilon)\bcdot\varepsilon}_{\text{spatially distributed contribution}}+\underbrace{\int_{\partial M}\hat{\mathcal F}^\star_\partial(A,\varepsilon)\bcdot \iota_{\partial M}^*\varepsilon}_{\text{boundary contribution}}.
    \end{align*}

    \item[(ii)] For $T^\dagger\mathcal C(P)=\mathcal C(P)\times\Omega^{m-1}(M,\tilde{\mathfrak g}^*)\times\Omega^{m-2}(\partial M,\tilde{\mathfrak g}^*)$, let $(A,\varepsilon,\varsigma,\varsigma_\partial):[t_0,t_1] \to T\mathcal C(P)\oplus T^{\dagger}\mathcal C(P)$ be a solution of \eqref{LDequations_nonabeliangauge_dagger}. Then:
    \begin{align*}
    \frac{\partial }{\partial t}\mathscr{E}_{\rm gau}(A,\varepsilon,B_A)=-{\rm d}\left(\varepsilon\bwedge\frac{\partialnew\mathscr L_{\rm gau}}{\partialnew\beta}(A,\varepsilon,B_A)\right)+\varepsilon\bwedge\hat{\mathcal F}^\dagger(A,\varepsilon),
    \end{align*}
    and
    \begin{align*}
    \frac{d}{dt} \int_M \mathscr{E}_{\rm gau}(A,\varepsilon,B_A)=\underbrace{\int_M\varepsilon\bwedge\hat{\mathcal F}^\dagger(A,\varepsilon)}_{\text{spatially distributed contribution}}+\underbrace{\int_{\partial M}\iota_{\partial M}^*\varepsilon\bwedge\hat{\mathcal F}^\dagger_\partial(A,\varepsilon)}_{\text{boundary contribution}}.
    \end{align*}
\end{itemize}
\end{proposition}

\begin{remark}[Expressions of the energy flux density]\label{energy_flux_density_YM}{\rm The local energy balance equations allow for the identification of the general expression of the energy flux density in terms of $\mathscr{L}_{\rm gau}$. It is given by the vector field density  $S=\frac{\partial\mathscr L_{\rm gau}}{\partial\beta}\bcdot\varepsilon= \frac{\partial\mathscr L_{\rm gau}}{\partial B_A}\bcdot\dot A$ or the $(m-1)$-form $\mathsf{S}=\varepsilon\bwedge\frac{\partialnew\mathscr L_{\rm gau}}{\partialnew\beta}= \dot A\bwedge\frac{\partialnew\mathscr L_{\rm gau}}{\partialnew B_A}$, depending on the chosen restricted dual.}
\end{remark}

\begin{remark}[Interior and boundary contributions]\rm
To clarify the structure of the interior and boundary contribution, we recall that $\hat{\mathcal F}_\partial^\star\in \Lambda^1_{m-1}(\partial M, \tilde{\mathfrak{g}}^*)$ and $\iota_{\partial M}^*\varepsilon=\iota_{\partial M}^*\dot A\in \Omega^1(\partial M, \tilde{\mathfrak{g}})$  are respectively understood as the external effort and flow variables on the boundary $\partial M$, while $\hat{\mathcal F}^\star\in \Lambda^1_m(M, E^*)$ and $\varepsilon=\dot A\in \Omega^1(M,E)$ are the external effort and flow variables on $M$. Their associated interior and boundary power densities are obtained via the contraction operations. The same comment applies to $\hat{\mathcal F}_\partial^\dagger\in \Omega^{m-2}(\partial M, \tilde{\mathfrak{g}}^*)$, resp., $\hat{\mathcal F}^\dagger\in \Omega^{m-2}(\partial M, \tilde{\mathfrak{g}}^*)$, where now the wedge operation is used to get the associated power densities.
\end{remark}

\begin{remark}[Lagrange--d'Alembert form]\rm In a similar way to \eqref{LdA_BundleValuedkforms_star} and \eqref{LdA_BundleValuedkforms_dagger}, mentioned earlier, by eliminating the variables $\varsigma(t)$, $\varsigma_\partial(t)$ and $\varepsilon(t)$ in 
\eqref{LDequations_nonabeliangauge_star} and \eqref{LDequations_nonabeliangauge_dagger}, we get the equations in the \textit{Lagrange--d'Alembert form} as follows:
\begin{equation*}%\label{LdA_connection_star}
\left\{\begin{array}{l}
\displaystyle\frac{\partial }{\partial t}\frac{\partial\mathscr L_{\rm gau}}{\partial\varepsilon}(A,\dot A,B_A)=\frac{\partial\mathscr L_{\rm gau}}{\partial A}(A,\dot A,B_A)+\operatorname{div}^{ A_*}\left(\frac{\partial\mathscr L_{\rm gau}}{\partial\beta}(A,\dot A,B_A)\right)+\hat{\mathcal{F}}^\star(A,\dot{A}),\vspace{2mm}\\ 
\displaystyle\hat{\mathcal{F}}_\partial^\star(A, \dot A)=-\iota_{\partial M}^*\left(\tr\frac{\partial\mathscr L_{\rm gau}}{\partial\beta}(A,\dot A,B_A)\right),
\end{array}\right.
\end{equation*}
and
\begin{equation*}%\label{LdA_connection_dagger}
\left\{\begin{array}{l}
\displaystyle\frac{\partial }{\partial t}\frac{\partialnew\mathscr L_{\rm gau}}{\partialnew \varepsilon}(A,\dot A,B_A)=\frac{\partialnew\mathscr L_{\rm gau}}{\partialnew A}(A,\dot A,B_A)+{\rm d}^{A_*} \left( \frac{\partialnew\mathscr L_{\rm gau}}{\partialnew\beta}(A,\dot A,B_A) \right) +\hat{\mathcal{F}}^{\dagger}(A,\dot{A}),\vspace{2mm}\\
\displaystyle\hat{\mathcal{F}}^{\dagger}_\partial( A, \dot A)=-\iota_{\partial M}^*\left(\frac{\partialnew\mathscr L_{\rm gau}}{\partialnew\beta}(A,\dot A,B_A)\right).
\end{array}\right.
\end{equation*}
\end{remark}

\subsection{Yang--Mills Lagrangian density}\label{sec:YMlagrangiandensity}

Let us particularize the previous equations to the Yang--Mills Lagrangian by following the geometric construction given in \cite[\S 3.3]{GBRa2008}. In the following, $G$ is assumed to be compact, connected and semisimple. Hence, the Killing form on $\mathfrak g$, which is denoted by $K:\mathfrak g\times\mathfrak g\to\mathbb R$, is non-degenerate. In general, $K$ is indefinite, but in our case the negative definiteness of $K$ is ensured by compacity. In addition, let $g$ be a Riemannian metric on $M$. The musical isomorphisms induced by $g$ are denoted by $\sharp:T^*M\to TM$ and $\flat:TM\to T^*M$.

The adjoint invariance of $K$ ensures that it induces a well defined fibered product on the adjoint bundle, which we denote by the same symbol for simplicity, $K:\tilde{\mathfrak g}\times_M\tilde{\mathfrak g}\to\mathbb R$.  The musical isomorphisms induced by $K$ are denoted by $\sharp_K:\tilde{\mathfrak g}^*\to\tilde{\mathfrak g}$ and $\flat_K:\tilde{\mathfrak g}\to\tilde{\mathfrak g}^*$. In addition, a fibered inner product on the $\tilde{\mathfrak g}$-valued forms on $M$,
\begin{equation*}
\mathbf g:\big(\textstyle\bigwedge^k T^*M\otimes\tilde{\mathfrak g}\big)\times_M\big(\textstyle\bigwedge^k T^*M\otimes\tilde{\mathfrak g}\big)\to\mathbb R,
\end{equation*}
may be defined as
\begin{equation}\label{eq:mathbfg}
\mathbf g_x\left(\alpha_1\otimes\sigma_1,\alpha_2\otimes\sigma_2\right)=-g_x(\alpha_1,\alpha_2)\,K(\sigma_1,\sigma_2),\quad x\in M,~\alpha_1,\alpha_2\in\textstyle\bigwedge^k T_x^*M,~\sigma_1,\sigma_2\in\tilde{\mathfrak g}_x,
\end{equation}
for pure elements, and extended to the tensor product by linearity. Note that the negative sign has been added to ensure the positive definiteness of $\mathbf g$. The musical isomorphisms induced by $\mathbf g$ are given by
\begin{align*}
\boldsymbol{\sharp}:\textstyle\bigwedge^k T^*M\otimes\tilde{\mathfrak g}\to\bigwedge^k TM\otimes\tilde{\mathfrak g}^*,\qquad & \alpha\otimes\sigma\mapsto\alpha^\sharp\otimes\sigma^{\flat_K},\\
\boldsymbol{\flat}:\textstyle\bigwedge^k TM\otimes\tilde{\mathfrak g}^*\to\bigwedge^k T^*M\otimes\tilde{\mathfrak g},\qquad & U\otimes\eta\mapsto U^\flat\otimes\eta^{\sharp_K}.
\end{align*}
At last, the Hodge star operator on $M$ is trivially extended to $E$-valued $k$-forms, being $E=\tilde{\mathfrak g}$ or $E=\tilde{\mathfrak g}^*$. Namely, for each $\alpha\in\bigwedge^k T^*M$ and $\sigma\in E$, it is defined as $\boldsymbol\star(\alpha\otimes\sigma)=(\star\alpha)\otimes\sigma$, where $\star:\bigwedge^k T^*M\to\bigwedge^{m-k} T^*M$ denotes the standard Hodge star operator.

On the other hand, we define the extended principal $G$-bundle as $\overline \pi_{P,M}:\overline P\to\overline M$, where $\overline P=[t_0,t_1]\times P$ and $\overline M=[t_0,t_1]\times M$. The base space $\overline M$ is interpreted as the spacetime and is endowed with the \emph{Minkowski metric}, that is, $\eta=-dt\otimes dt+g$, where $t\in[t_0,t_1]$ is the global coordinate (we omit the necessary pullbacks for brevity). As above, we can define a fibered inner product
\begin{equation*}
\boldsymbol\eta:\left(\textstyle\bigwedge^k T^*\overline M\otimes([t_0,t_1]\times\tilde{\mathfrak g})\right)\times_{\overline M}\left(\textstyle\bigwedge^k T^*\overline M\otimes([t_0,t_1]\times\tilde{\mathfrak g})\right)\to\mathbb R,
\end{equation*}
where we have used that the adjoint bundle of $\overline\pi_{P,M}:\overline P\to\overline M$ is given by $\left(\overline P\times\mathfrak g\right)/G=[t_0,t_1]\times\tilde{\mathfrak g}$. Namely, we set
\begin{equation*}
\boldsymbol\eta_{(t,x)}\left(\overline\alpha_1\otimes\overline\sigma_1,\overline\alpha_2\otimes\overline\sigma_2\right)=\eta_{(t,x)}\left(\overline\alpha_1,\overline\alpha_2\right)K(\sigma_1,\sigma_2)
\end{equation*}
for each $(t,x)\in\overline M$, $\overline\alpha_i\in\bigwedge^k T_{(t,x)}^*\overline M$ and $\overline\sigma_i=(t,\sigma_i)\in\left([t_0,t_1]\times\tilde{\mathfrak g}\right)_{(t,x)}$, $i=1,2$. In particular, for $k=2$ we distinguish two situations:
\begin{enumerate}
\item If $\overline\alpha_i=dt\wedge\alpha_i$ for some $\alpha_i\in T_x^*M$, $i=1,2$, then
\begin{equation*}\eta_{(t,x)}\left(\overline\alpha_1,\overline\alpha_2\right)=
\det\begin{pmatrix}
\eta_{(t,x)}(dt,dt) & \eta_{(t,x)}(dt,\alpha_2)\\
\eta_{(t,x)}(\alpha_1,dt) & \eta_{(t,x)}(\alpha_1,\alpha_2)
\end{pmatrix}=
\det\begin{pmatrix}
-1 & 0\\
0 & g_x(\alpha_1,\alpha_2)
\end{pmatrix}=
-g_x(\alpha_1,\alpha_2).
\end{equation*}
Thus, $\boldsymbol\eta_{(t,x)}\left(\overline\alpha_1\otimes\overline\sigma_1,\overline\alpha_2\otimes\overline\sigma_2\right)=\mathbf g_x\left(\alpha_1\otimes\sigma_1,\alpha_2\otimes\sigma_2\right)$.
\item If $\overline\alpha_i=\alpha_i^1\wedge\alpha_i^2\in\bigwedge^2 T_x^*M$, $i=1,2$, then
\begin{align*}
\eta_{(t,x)}\left(\overline\alpha_1,\overline\alpha_2\right) & =
\det\begin{pmatrix}
\eta_{(t,x)}(\alpha_1^1,\alpha_2^1) & \eta_{(t,x)}(\alpha_1^1,\alpha_2^2)\\
\eta_{(t,x)}(\alpha_1^2,\alpha_2^1) & \eta_{(t,x)}(\alpha_1^2,\alpha_2^2)
\end{pmatrix}\\
& =\det\begin{pmatrix}
g_x(\alpha_1^1,\alpha_2^1) & g_x(\alpha_1^1,\alpha_2^2)\\
g_x(\alpha_1^2,\alpha_2^1) & g_x(\alpha_1^2,\alpha_2^2)
\end{pmatrix}\\
& =g_x(\alpha_1,\alpha_2).
\end{align*}
Thus, $\boldsymbol\eta_{(t,x)}\left(\overline\alpha_1\otimes\overline\sigma_1,\overline\alpha_2\otimes\overline\sigma_2\right)=-\mathbf g_x\left(\alpha_1\otimes\sigma_1,\alpha_2\otimes\sigma_2\right)$.
\end{enumerate}

The bundle of connections of $\overline\pi_{P,M}:\overline P\to\overline M$ is denoted by $\operatorname{Conn}(\overline P)=\left(J^1\overline P\right)/G\to\overline M$. The correspondence between principal connections on $\overline\pi_{P,M}:\overline P\to\overline M$ and sections of the bundle of connections is denoted by
\begin{equation*}
\mathcal C(\overline P)\ni\mathsf A\overset{1:1}{\longleftrightarrow}\mathcal A\in\Omega^1\left(\overline P,\mathfrak g\right)\text{ principal connection},
\end{equation*}
where $\mathcal C(\overline P)=\Gamma(\operatorname{Conn}(\overline P)\to \overline M)$.

\begin{remark}\rm\label{remark:codifferentialoverlineM}
As for principal connections on the original bundles, a principal connection $\mathsf A$ on the extended bundle $\overline\pi_{P,M}:\overline P\to\overline M$ induces a linear connection $\nabla^{\mathsf A}$ on $\Omega^1\left(\overline M,[t_0,t_1]\times\tilde{\mathfrak g}\right)$. The corresponding exterior derivative is denoted by ${\rm d}^{\mathsf A}:\Omega^k\left(\overline M,[t_0,t_1]\times\tilde{\mathfrak g}\right)\to\Omega^{k+1}\left(\overline M,[t_0,t_1]\times\tilde{\mathfrak g}\right)$. Similarly, $\delta^{\mathsf A}:\Omega^k\left(\overline M,[t_0,t_1]\times\tilde{\mathfrak g}\right)\to\Omega^{k-1}\left(\overline M,[t_0,t_1]\times\tilde{\mathfrak g}\right)$ denotes the codifferential of ${\rm d}^{\mathsf A}$, where the Hodge star operator $\overline\star:\bigwedge^k T^*\overline M\to\bigwedge^{m+1-k}T^*\overline M$ is induced by the Minkowski metric $\eta=-dt\otimes dt+g$ on $\overline M$. Similarly, when regarded as $\mathcal A\in\Omega^1\left(\overline P,\mathfrak g\right)$, the covariant exterior derivative on the space of $\mathfrak g$-valued forms on $\overline P$ is denoted by ${\rm d}^{\mathcal A}:\Omega^k\left(\overline P,\mathfrak g\right)\to\Omega^{k+1}\left(\overline P,\mathfrak g\right)$.
\end{remark}

\begin{definition}\label{def_YM}\rm
The \emph{Yang--Mills Lagrangian density} is the bundle morphism
\begin{equation*}
\mathscr L_{\rm YM}:\operatorname{Conn}(\overline P)\times_{\overline M}\left(\textstyle\bigwedge^2 T^*\overline M\otimes([t_0,t_1]\times\tilde{\mathfrak g})\right)\to\textstyle\bigwedge^{m+1} T^*\overline M,
\end{equation*}
covering the identity $\operatorname{id}_{\overline M}$, defined as
\begin{equation}
\label{YM_L}
\mathscr L_{\rm YM}\left(\mathsf A_{(t,x)},\mathsf B_{(t,x)}\right)=\displaystyle\frac{1}{2}\boldsymbol\eta\left(\mathsf B_{(t,x)},\mathsf B_{(t,x)}\right)dt \wedge \mu_g,
\end{equation}
where $\mu_g\in\Omega^m(M)$ is the volume form given by $g$. 
\end{definition}

The pair $\left(\mathsf A_{(t,x)},\mathsf B_{(t,x)}\right)$, with $(t,x)\in\overline M$, defines (pointwisely) a principal connection on the extended principal bundle and its reduced curvature.

Given a principal connection $\mathcal A\in\Omega^1\left(\overline P,\mathfrak g\right)$ on $\overline\pi_{P,M}$, we may write $\mathcal A(t,x)=\mathbf A(t,x)+\mathbf A_0(t,x)\,dt$ for each $(t,x)\in\overline M$, where $\mathbf A(t,\cdot)\in\Omega^1(P,\mathfrak g)$ is a principal connection on $\pi_{P,M}$ and $\mathbf A_0(t,\cdot)\in C_G^\infty(P,\mathfrak g)$ (the space of $\Ad$-invariant functions on $P$ with values on $\mathfrak g$). The curvature of $\mathcal A$ is thus given by
\begin{equation}\label{eq:curvatureA}
\mathsf B_{\mathcal A}={\rm d}^{\mathcal A}\mathcal A={\rm d}^{\mathbf A}\mathbf A+\left(-\dot{\mathbf A}+{\rm d}^{\mathbf A}\mathbf A_0\right)\wedge dt\in\overline\Omega^2\left(\overline P,\mathfrak g\right).
\end{equation}
It is always possible to choose the temporal gauge $\mathbf A_0=0$ as a ``gauge fixing'' (cf. e.g., \cite[\S 3.3.]{GBRa2008} for details). By using this and the expressions of $\boldsymbol\eta$ for $k=2$ computed above, the action functional $\mathcal A_{\rm YM}:\mathcal C(\overline P)\to\mathbb R$ induced by the Yang--Mills density \eqref{YM_L} reads
\begin{equation*}
\mathcal A_{\rm YM}(\mathsf A)=\frac{1}{2}\int_{\overline M}\boldsymbol\eta\left(\mathsf B_{\mathsf A},\mathsf B_{\mathsf A}\right)dt\wedge\mu_g=\frac{1}{2}\int_{t_0}^{t_1}\int_M\left(\mathbf g(\dot A,\dot A)-\mathbf g(B_A,B_A)\right)\mu_g dt,
\end{equation*}
where $\mathsf A\in\mathcal C(\overline P)$ and $\mathsf B_{\mathsf A}\in\Omega^2\left(\overline M,[t_0,t_1]\times\tilde{\mathfrak g}\right)$ denotes its reduced curvature. This enables us to express the Yang--Mills Lagrangian density in the form treated in Section \ref{sec:nonbabelianfields}, see \eqref{eq:lagrangiandensitynonabelian}, namely,
\begin{equation}
\label{split_Lagrangian_density}
\mathscr L_{\rm YM}\left(A_x,\varepsilon_x, \beta_x\right)=\frac{1}{2}\left(\mathbf g\left(\varepsilon_x,\varepsilon_x\right)-\mathbf g\left(\beta_x,\beta_x\right)\right)\mu_g.
\end{equation}
As above, the triple $\left(A_x,\varepsilon_x,\beta_x\right)$, $x\in M$, defines (pointwisely) a principal connection, (negative) its time derivative and its reduced curvature. 

Physically, $A(t,\cdot)\in\mathcal C(P)$ describes the \emph{potential} (which is not physical in the sense that it cannot be measured), whereas $E(t)=-\dot A(t,\cdot)\in\Omega^1(M,\tilde{\mathfrak g})$ and $B_{A(t,\cdot)}\in\Omega^2(M,\tilde{\mathfrak g})$ describe the \emph{electric} and \emph{magnetic} parts of the field strength, respectively.

\subsection{Yang--Mills equations}

Lastly, the forced Lagrange--Dirac equations for the Yang--Mills Lagrangian density are computed. To that end, we need the following two lemmas.

\begin{lemma}\label{lemma:metricconnection}\rm
Let $A\in\mathcal C(P)$. If $G$ is compact and semisimple, then $\nabla^A$ is a metric connection with respect to $K$. In particular, $\flat_K\circ{\rm d}^A={\rm d}^{A*}\circ\flat_K$.
\end{lemma}
\begin{proof}
We consider local coordinates on $M$ and a basis of local sections of $\tilde{\mathfrak g}$, as in the proof of Proposition \ref{prop:covariantdivergence}. The principal connection and the Killing form are locally given by $A=A_i^a\,dx^i\otimes \hat\xi_a$ and $K=K_{ab}\,\hat\xi^a\otimes \hat\xi^b$, respectively, for some local functions $A_i^a,K_{ab}\in C^\infty(M)$ such that $K_{ab}=K_{ba}$, $1\leq a,b\leq n$, $1\leq i\leq m$. Since $K$ is nondegenerate, it is always possible to pick the basis so that $K_{ab}=\epsilon_a\delta_{ab}$, $1\leq a,b\leq n$, where $\epsilon_a\in\{-1,1\}$. Moreover, since $G$ is compact, $K$ is negative-definite and, thus, $\epsilon_a=-1$ for each $1\leq a\leq n$. The adjoint invariance of $K$ gives
\begin{equation}\label{p1}
-f_{bc}^a=K(\hat\xi_a,f_{bc}^d \hat\xi_d)=K(\hat\xi_a,[\hat\xi_b,\hat\xi_c])=K([\hat\xi_a,\hat\xi_b],\hat\xi_c)=K(f_{ab}^d \hat\xi_d,\hat\xi_c)=-f_{ab}^c,
\end{equation}
where $f_{bc}^a\in\mathbb R$, $1\leq a,b,c\leq n$, are the structure constants of $\mathfrak g$. In short, we obtain $f_{bc}^a=-f_{ba}^c$ for each $1\leq a,b,c\leq n$. On the other hand, given $\xi_1=\xi_1^a\,\hat\xi_a,\xi_2=\xi_2^a\,\hat\xi_a\in\Gamma(\tilde{\mathfrak g}\to M)$, we have
\begin{equation*}
\nabla^A\xi_1%={\rm d}\xi_1+[A,\xi_1]
=\left(\partial_i\xi_1^a+f_{bc}^a\,A_i^b\,\xi_1^c\right)dx^i\otimes \hat\xi_a,
\end{equation*}
and analogous for $\xi_2$. By using this and \eqref{p1}, we obtain the first part:
\begin{align*}
{\rm d} \left( K(\xi_1,\xi_2)\right)  & =-{\rm d}\left(\xi_1^a\,\xi_2^a\right) =-\partial_i\left(\xi_1^a\xi_2^a\right)dx^i\\
& =-\left(\left(\partial_i\xi_1^a\right)\xi_2^a+\xi_1^a\left(\partial_i\xi_2^a\right)\right)dx^i\\
& =-\left(\left(\partial_i\xi_1^a-f_{bc}^a\,A_i^b\,\xi_1^a\right)\xi_2^c+\xi_1^a\left(\partial_i\xi_2^a+f_{bc}^a\,A_i^b\,\xi_2^c\right)\right)dx^i\\
& =-\left(\left(\partial_i\xi_1^a+f_{bc}^a\,A_i^b\,\xi_1^c\right)\xi_2^a+\xi_1^a\left(\partial_i\xi_2^a+f_{bc}^a\,A_i^b\,\xi_2^c\right)\right)dx^i\\
& =K(\nabla^A\xi_1,\xi_2)+K(\xi_1,\nabla^A\xi_2).
\end{align*}
The last part is an easy computation using local coordinates.
\end{proof}

\begin{lemma}\rm\label{lemma:partialderivativesYM}
For each $(A,\varepsilon)\in T\mathcal C(P)\simeq\mathcal C(P)\times\Omega^1(M,\tilde{\mathfrak g})$, by denoting $E=-\varepsilon$, the fiber derivatives of the Lagrangian density \eqref{split_Lagrangian_density} are given by
\begin{equation*}
\left\{\begin{array}{lll}
\displaystyle\frac{\partial\mathscr L_{\rm YM}}{\partial A}=0,\quad & \displaystyle\frac{\partial\mathscr L_{\rm YM}}{\partial\varepsilon}=-E^{\boldsymbol{\sharp}}\otimes\mu_g,\quad & \displaystyle\frac{\partial\mathscr L_{\rm YM}}{\partial\beta}=-B_A^{\boldsymbol\sharp}\otimes\mu_g,\vspace{2mm}\\
\displaystyle\frac{\partialnew\mathscr L_{\rm YM}}{\partialnew A}=0,\quad & \displaystyle\frac{\partialnew\mathscr L_{\rm YM}}{\partialnew\varepsilon}=-\boldsymbol\star E^{\flat_K},\quad & \displaystyle\frac{\partialnew\mathscr L_{\rm YM}}{\partialnew\beta}=-\boldsymbol\star B_A^{\flat_K}.
\end{array}\right.
\end{equation*}
\end{lemma}

\begin{proof}
By using local coordinates on $M$ and a basis of local sections of $\tilde{\mathfrak g}$, as in the proof of Lemma \ref{lemma:metricconnection}, we may locally write $A=A_i^a\,dx^{i}\otimes \hat\xi_a$ and $E=-\varepsilon=-\varepsilon_i^a\,dx^{i}\otimes \hat\xi_a$ for some local functions $A_i^a,\varepsilon_i^a\in C^\infty(M)$, $1\leq i\leq m$, $1\leq a\leq n$. Subsequently, the (reduced) curvature reads
\begin{equation*}
B_A=\left(\partial_i A_{j}^a+\Gamma_{i,b}^a\,A_j^b\right)\,dx^{i}\wedge dx^{j}\otimes \hat\xi_a,
\end{equation*}
where $\Gamma_{i,b}^a\in C^\infty(M)$, $1\leq i\leq m$, $1\leq a,b\leq n$, are the Christoffel symbols of $\nabla^A$, i.e., $\Gamma_{i,b}^a=f_{cb}^a A_i^c$.

The first part is a straightforward computation. For the second one, recall that the Riemannian metric and the Killing form are locally given by $g=g_{ij}\,dx^{i}\otimes dx^{j}$ and $K=-\delta_{ab}\,\hat\xi^a\otimes \hat\xi^b$, respectively. The Riemannian volume form is locally given by $\mu_g=\sqrt{|g|}\,d^m x\in\Omega^m(M)$ and the Hodge star operator is locally given by
\begin{equation*}
\boldsymbol\star(dx^i\otimes \hat\xi_a )=\sqrt{|g|}\,g^{ij}\,d_{(j)}^{m-1}x\otimes \hat\xi_a ,\qquad\boldsymbol\star(dx^{i}\wedge dx^{j}\otimes \hat\xi_a )=\sqrt{|g|}\,g^{ik}\,g^{jl}\,d_{(k,l)}^{m-2}x\otimes \hat\xi_a.
\end{equation*}
Therefore, we have
\begin{equation*}
E^{\boldsymbol\sharp}=-g^{ij}\,\varepsilon_j^a \,\partial_{i}\otimes \hat\xi^a,\qquad B_A^{\boldsymbol\sharp}=g^{ik}\,g^{jl}\left(\partial_i A_{j}^a+\Gamma_{i,c}^a\,A_j^c\right)\partial_{k}\wedge\partial_{l}\otimes \hat\xi^a.
\end{equation*}
In view of \eqref{eq:PhiE}, from the above relations it follows 
\begin{equation*}
\left\{\begin{array}{l}
\displaystyle\frac{\partialnew\mathscr L_{\rm YM}}{\partialnew A}= \Phi_E\circ\frac{\partial\mathscr L_{\rm YM}}{\partial A}=\Phi_E\circ 0=0,\vspace{0.1cm}\\[5mm]
\displaystyle\frac{\partialnew\mathscr L_{\rm YM}}{\partialnew\varepsilon}=\Phi_E\circ(-E^{\boldsymbol\sharp}\otimes\mu_g)=\sqrt{|g|}\,g^{ij}\,\varepsilon_j^a\,d_{(i)}^{m-1}x\otimes \hat\xi^a=-\left(\boldsymbol\star E\right)^{\flat_K},\vspace{0.1cm}\\[5mm]
\displaystyle\frac{\partialnew\mathscr L_{\rm YM}}{\partialnew\beta}=\Phi_E\circ\left(-B_A^{\boldsymbol\sharp}\otimes\mu_g\right)=-\sqrt{|g|}\,g^{ik}\,g^{jl}\left(\partial_i A_{j}^a+\Gamma_{i,c}^a\,A_j^c\right)d_{(k,l)}^{m-2}x\otimes \hat\xi^a=-\left(\boldsymbol\star B_A\right)^{\flat_K}.
\end{array}\right.
\end{equation*}
We conclude by noting that $\boldsymbol\star$ commutes with $\flat_K$.
\end{proof}

\medskip

From the previous results and by denoting $E=-\varepsilon$, we obtain the forced Lagrange--Dirac equations for the Yang--Mills Lagrangian density.

\begin{enumerate}[(i)]
    \item The Lagrange--Dirac equations \eqref{LDequations_nonabeliangauge_star} for the Lagrangian density \eqref{split_Lagrangian_density} read
    \begin{equation*}
    \left\{\begin{array}{ll}
    \dot A=-E, & \vspace{0.2cm}\\
    \displaystyle\varsigma=- E^{\boldsymbol{\sharp}}\otimes\mu_g,\qquad & \displaystyle\dot\varsigma=-\operatorname{div}^{A*}\left(B_A^{\boldsymbol\sharp}\otimes\mu_g\right)+\hat{\mathcal F}^\star,\vspace{0.2cm}\\
    \varsigma_\partial=0, & \displaystyle\dot\varsigma_\partial=-\iota_{\partial M}^*\left(\tr\left(B_A^{\boldsymbol\sharp}\otimes\mu_g\right)\right)+\hat{\mathcal F}_\partial^\star.
    \end{array}\right.
    \end{equation*}
    \item  The Lagrange--Dirac equations \eqref{LDequations_nonabeliangauge_dagger} for the Lagrangian density \eqref{split_Lagrangian_density} read
    \begin{equation*}
    \left\{\begin{array}{ll}
    \dot A=-E, & \vspace{0.2cm}\\
    \displaystyle\varsigma=-\boldsymbol\star E^{\flat_K},\qquad & \displaystyle\dot\varsigma=-{\rm d}^{A*}(\boldsymbol\star B_A^{\flat_K})+\hat{\mathcal F}^\dagger,\vspace{0.2cm}\\
    \varsigma_\partial=0, & \displaystyle\dot\varsigma_\partial=-\iota_{\partial M}^*\left(\boldsymbol\star B_A^{\flat_K}\right)+\hat{\mathcal F}_\partial^\dagger.
    \end{array}\right.
    \end{equation*}
\end{enumerate}

As for matter fields, external forces describe interior and boundary currents. By using Lemma \ref{lemma:metricconnection}, eliminating the variables $\dot A,\varsigma,\varsigma_\partial$, and denoting $\boldsymbol\star\hat{\mathcal F}^\dagger=-J^{\flat_K}$ and $\boldsymbol\star_\partial\hat{\mathcal F}_\partial^\dagger=j^{\flat_K}$, where $J:T\mathcal C(P)\to\Omega^1(M,\tilde{\mathfrak g})$ and $j:T\mathcal C(P)\to\Omega^1(\partial M,\tilde{\mathfrak g})$, the forced Lagrange--Dirac equations on $T\mathcal C(P)\oplus T^\dagger\mathcal C(P)$ read
\begin{equation}\label{eq:YMeqs}
\dot E-\delta^A B_A=-J,\qquad\boldsymbol\star_\partial\left(\iota_{\partial M}^*(\boldsymbol\star B_A)\right)=j,
\end{equation}
where $\delta^A=(-1)^{m+1}\boldsymbol\star\circ{\rm d}^A\circ\boldsymbol\star:\Omega^2(M,\tilde{\mathfrak g})\to\Omega^1(M,\tilde{\mathfrak g})$ is the codifferential of ${\rm d}^A$, as defined earlier in the general case.
Together with the equations ${\rm d}^AB_A=0$ and $\dot B_A=- {\rm d}^A E$ which follow from $B_A={\rm d}^AA$ and $E=-\dot A$, these are the \emph{Yang--Mills equations with external currents and boundary conditions} in the space+time decomposition. See \S\ref{sec:YMHeqs} for a discussion on the boundary conditions.

\paragraph{Energy and charge equations} The energy density \eqref{energy_gau} for the Yang--Mills field becomes
\begin{equation*}
\mathscr E_{\rm YM}(A,E,B_A)=\frac{1}{2}\left(\mathbf{g}(E,E)+\mathbf{g}(B_A,B_A)\right)\mu_g.
\end{equation*}
The local and global energy balances given in (ii) of Proposition \ref{prop:energybalancegauge} yield the \emph{non-Abelian Poynting theorem} and its global version:
\begin{align*}
\frac{\partial\mathscr E_{\rm YM}}{\partial t} & = - {\rm d}\mathsf S+E\bwedge\boldsymbol{\star}\, J^{\flat_K},
%={\rm d}S+(-1)^{m}\mathbf g(E,J)\,\mu_g ,
\end{align*}
and
\begin{align*}
\frac{d}{dt} \int_M \mathscr E_{\rm YM} & = \int_M E\bwedge\boldsymbol{\star}\, J^{\flat_K}+\int_{\partial M}\big(\boldsymbol\star_\partial\,j^{\flat_K}\big)\bwedge\iota_{\partial M}^*E,
\end{align*}
respectively, where $\mathsf S=E\bwedge\boldsymbol{\star}B_A^{\flat_K}\in\Omega^{m-1}(M)$ (recall Remark \ref{energy_flux_density_YM}), which gives the \emph{non-Abelian Poynting vector}, $\mathbf S=(\boldsymbol\star\mathsf S)^\sharp\in\mathfrak X(M)$.

The \emph{non-Abelian charge density} is defined as
\begin{equation}\label{eq:charge_YM}
\rho=\delta^A E.
\end{equation}
From the first equation of \eqref{eq:YMeqs}, we obtain the charge conservation law: $\dot\rho+\delta^A J=0$.

\paragraph{Spacetime formulation} For completeness, we illustrate here how the equations obtained above yield the usual spacetime formulation of the Yang--Mills equation.
The charge density and the interior current may be gathered in the \emph{color current},
\begin{equation}\label{eq:colorcurrent}
\mathsf J=J+\rho\,dt\in\Omega^1\left(\overline M,[t_0,t_1]\times\tilde{\mathfrak g}\right).
\end{equation}
As a basic form of the adjoint type on the total space, it reads $\mathcal J=\mathbf J+\boldsymbol\rho\,dt\in\overline\Omega^1\left(\overline P,\mathfrak g\right)$, where $\boldsymbol\rho\in C_G^\infty\left(\overline P,\mathfrak g\right)$ and $\mathbf J\in\overline\Omega^1\left(\overline P,\mathfrak g\right)$ correspond to $\rho$ and $J$, respectively.

\begin{corollary}\rm\label{corollary:unified_YMeqs} The equations derived above yield the Yang--Mills equations and non-Abelian charge conservation as follows:
\begin{equation*}
\delta^{\mathsf A}\mathsf B_{\mathsf A}=\mathsf J, \qquad {\rm d}^\mathsf{A}\mathsf{B}_{\mathsf{A}}=0,\qquad \delta^{\mathsf A}\mathsf J=0.
\end{equation*}
\end{corollary}

\begin{proof}
Firstly, we write
\begin{equation}\label{demo}
\delta^{\mathsf A}\mathsf B_{\mathsf A}=\varpi_1\,dt+\varpi_2\in\Omega^1(\overline M,[t_0,t_1]\times\tilde{\mathfrak g}),
\end{equation}
where $\varpi_1\in\Omega^0(\overline M,[t_0,t_1]\times\tilde{\mathfrak g})$ and $\varpi_2(t,\cdot)\in\Omega^1(M,\tilde{\mathfrak g})$ for each $t\in[t_0,t_1]$. On the other hand, by substituting the definitions of electric and magnetic fields on \eqref{eq:curvatureA} (with $\mathbf A_0=0$), we get
\begin{equation}\label{demo0}
\mathsf B_{\mathsf{A}}(t,\cdot)= B_{A(t,\cdot)} + E(t,\cdot)\wedge dt,\qquad t\in[t_0,t_1].
\end{equation} 

On the other hand, let $\alpha\in\Omega^0(M,\tilde{\mathfrak g})=C^\infty(M,\tilde{\mathfrak g})$ with compact support inside $M$, i.e., ${\rm cl}\left({\rm supp}\,\alpha\right)\subset{\rm int}M=M-\partial M$, where ${\rm cl}$ denotes the closure. Note that
\begin{align}\nonumber
\boldsymbol\eta\left(E\wedge dt,{\rm d}^A\alpha\wedge dt\right) & =
\det\begin{pmatrix}
\boldsymbol\eta(dt,dt) & \boldsymbol\eta(dt,{\rm d}^A\alpha)\\
\boldsymbol\eta(E,dt) & \boldsymbol\eta(E,{\rm d}^A\alpha)
\end{pmatrix}\\ \label{demo1}
& =
\det\begin{pmatrix}
-1 & 0\\
0 & \boldsymbol\eta(E,{\rm d}^A\alpha)
\end{pmatrix}=
-\boldsymbol\eta(E,{\rm d}^A\alpha).
\end{align}
By using that ${\rm d}^{\mathsf A}(\alpha\,dt)={\rm d}^A\alpha\wedge dt$, as well as \eqref{eq:charge_YM}, \eqref{demo}, \eqref{demo0} and \eqref{demo1}, we obtain
\begin{align*}
\int_{\overline M}\boldsymbol\eta\left(\varpi_1\,dt,\alpha\,dt\right)dt\wedge\mu_g & =\int_{\overline M}\boldsymbol\eta\left(\delta^{\mathsf A}\mathsf B_{\mathsf A},\alpha\,dt\right)dt\wedge\mu_g =\int_{\overline M}\boldsymbol\eta\left(\mathsf B_{\mathsf A},{\rm d}^{\mathsf A}(\alpha\,dt)\right)dt\wedge\mu_g\\
& =\int_{\overline M}\boldsymbol\eta\left(E\wedge dt,{\rm d}^A\alpha\wedge dt\right)dt\wedge\mu_g =-\int_{\overline M}\boldsymbol\eta\left(E,{\rm d}^A\alpha\right)dt\wedge\mu_g\\
& =-\int_{\overline M}\boldsymbol\eta\left(\delta^A E,\alpha\right)dt\wedge\mu_g =-\int_{\overline M}\boldsymbol\eta\left(\rho,\alpha\right)dt\wedge\mu_g\\
& =\int_{\overline M}\boldsymbol\eta\left(\rho\,dt,\alpha\,dt\right)dt\wedge\mu_g.
\end{align*}
Since $\alpha$ is arbitrary and $\boldsymbol\eta$ is non-degenerate, we conclude that $\varpi_1=\rho$.

Analogously, let $\beta(t,\cdot)\in\Omega^1(M,\tilde{\mathfrak g})$ for each $t\in[t_0,t_1]$ with compact support inside $\overline M$, i.e., ${\rm cl}\{(t,x)\in\overline M\mid\beta(t,x)\neq 0\}\subset{\rm int}\,\overline M=\overline M-\partial\overline M$. By using \eqref{eq:YMeqs}, \eqref{demo}, \eqref{demo0} and \eqref{demo1}, we get
\begin{align*}
\int_{\overline M}\boldsymbol\eta\left(\varpi_2,\beta\right)dt\wedge\mu_g & =\int_{\overline M}\boldsymbol\eta\left(\delta^{\mathsf A}\mathsf B_{\mathsf A},\beta\right)dt\wedge\mu_g =\int_{\overline M}\boldsymbol\eta\left(\mathsf B_{\mathsf A},{\rm d}^{\mathsf A}\beta\right)dt\wedge\mu_g\\
& =\int_{\overline M}\boldsymbol\eta\left(B_A+E\wedge dt,{\rm d}^A\beta-\dot\beta\wedge dt\right)dt\wedge\mu_g\\
&=\int_{\overline M}\left(\boldsymbol\eta(B_A,{\rm d}^A\beta)+\boldsymbol\eta(E,\dot\beta)\right)dt\wedge\mu_g\\
& =\int_{\overline M}\boldsymbol\eta\left(\delta^A B_A-\dot E,\beta\right)dt\wedge\mu_g =\int_{\overline M}\boldsymbol\eta\left(J,\beta\right)dt\wedge\mu_g.
\end{align*}
Since $\beta$ is arbitrary and $\boldsymbol\eta$ is non-degenerate, we conclude that $\varpi_2=J$. By substituting this in \eqref{demo} and using \eqref{eq:colorcurrent}, we conclude:
\begin{equation*}
\delta^{\mathsf A}\mathsf B_{\mathsf A}=\rho\,dt=J=\mathsf J.
\end{equation*}
The identity ${\rm d}^{\mathsf A}\mathsf B_{\mathsf A}=0$ is a consequence of \eqref{demo0} by recalling that ${\rm d}^A B_A=0$ and $\dot B_A=-{\rm d}^A E$; namely,
\begin{align*}
{\rm d}^{\mathsf A}\mathsf B_{\mathsf A} & ={\rm d}^{\mathsf A}(B_A+E\wedge dt) ={\rm d}^A B_A-\dot B_A\wedge dt-{\rm d}^AE\wedge dt\\
& ={\rm d}^A E\wedge dt-{\rm d}^AE\wedge dt=0.
\end{align*}
Finally, from $\delta^\mathsf{A}\mathsf{B}_\mathsf{A}=\mathsf{J}$ and $\delta^\mathsf{A}\delta^\mathsf{A}\mathsf{B}_\mathsf{A}=0$ we readily obtain that $\delta^{\mathsf A}\mathsf J=0$.
\end{proof}

\subsection{Interaction with matter fields}\label{sec:particlefieldsinteraction}

The previous setting may be extended to consider gauge fields coupled with matter fields. The matter fields considered here are particular instances of the ones considered in section \ref{sec:particlefields}; namely, we consider sections of some associated bundle. More specifically, the configuration space for non-Abelian gauge fields interacting with matter fields is given by
\begin{equation*}
Q=\mathcal C(P)\times\Omega^0(M,\tilde V),
\end{equation*}
where $\tilde V=P\times_G V\to M$ is the associated bundle defined by a representation $\varrho:G\times V\to V$, with $V$ being a vector space.

\paragraph{Metric and exterior derivative on the associated bundle}

A $\varrho$-invariant inner product on $V$, $\kappa:V\times V\to\mathbb R$, induces a bundle metric on $\tilde V\to M$, which we denote by the same symbol for simplicity, $\kappa:\tilde V\times_M\tilde V\to\mathbb R$. The musical isomorphisms induced by $\kappa$ are denoted by $\sharp_\kappa:\tilde V^*\to\tilde V$ and $\flat_\kappa:\tilde V\to\tilde V^*$, as usual. Furthermore, a fibered inner product on $\tilde V$-valued forms on $M$,
\begin{equation*}
\mathbf g_\kappa:\left(\textstyle\bigwedge^k T^*M\otimes\tilde V\right)\times_M\left(\textstyle\bigwedge^k T^*M\otimes\tilde V\right)\to\mathbb R,
\end{equation*}
may be defined as in \eqref{eq:mathbfg}. The musical isomorphisms induced by $\mathbf g_\kappa$ are given by
\begin{align*}
\boldsymbol{\sharp}_\kappa:\textstyle\bigwedge^k T^*M\otimes\tilde V\to\bigwedge^k TM\otimes\tilde V^*,\qquad & \alpha\otimes\sigma\mapsto\alpha^\sharp\otimes\sigma^{\flat_\kappa},\\
\boldsymbol{\flat}_\kappa:\textstyle\bigwedge^k TM\otimes\tilde V^*\to\bigwedge^k T^*M\otimes\tilde V,\qquad & U\otimes\eta\mapsto U^\flat\otimes\eta^{\sharp_\kappa}.
\end{align*}
At last, the Hodge star operator on $M$ is trivially extended to $\tilde V$-valued forms: for each $\alpha\in\bigwedge^k T^*M$ and $\sigma\in\tilde V$, it is defined as $\boldsymbol\star(\alpha\otimes\sigma)=(\star\alpha)\otimes\sigma$, where $\star:\bigwedge^k T^*M\to\bigwedge^{m-k} T^*M$ denotes the standard Hodge star operator.

\begin{remark}[Notation]\rm
Although this is the same construction carried out in \S\ref{sec:particlefields}, we reserve the symbols $\mathbf g$, $\boldsymbol\sharp$ and $\boldsymbol\flat$ (without subscript) for the particular case where $V=\mathfrak g$ is the Lie algebra of $G$, $\kappa=K$ is the Killing form, and $\varrho$ is the adjoint representation.
\end{remark} 

A principal connection $\mathbf A\in\Omega^1(P,\mathfrak g)$ induces a horizontal exterior derivative ${\rm d}^{\mathbf A}:\Omega^k(P,V)\to\Omega^{k+1}(P,V)$. For basic 1-forms of the adjoint type, its expression is analogous to \eqref{eq:horizontalderivativebasic1forms}; namely,
\begin{equation*}
{\rm d}^{\mathbf A}\eta(u,v)={\rm d}\eta(u,v)+\tilde\varrho(\mathbf A(u),\eta(v))-\tilde\varrho(\mathbf A(v),\eta(u)),\qquad\eta\in\overline \Omega^1(P,V),~u,v\in\mathfrak X(P),
\end{equation*}
where $\tilde\varrho:\mathfrak g\times V\to V$ denotes the infinitesimal action of $\mathfrak g$ on $V$, i.e.,
\begin{equation*}
\tilde\varrho(\sigma,v)=\left.\frac{d}{dt}\right|_{t=0}\varrho\left(\exp(\sigma t),v\right)=(d\varrho_v)_e(\sigma),\qquad\sigma\in\mathfrak g,~v\in V.
\end{equation*}
In turn, when regarded as a section $A\in\mathcal C(P)$, the principal connection induces a linear connection on $\tilde V\to M$ as well as a covariant exterior derivative on the space of $\tilde V$-valued forms on $M$, which we denote by $\nabla^A$ and ${\rm d}^A:\Omega^k(M,\tilde V)\to\Omega^{k+1}(M,\tilde V)$, respectively. The corresponding dual connection and its covariant exterior derivative are denoted by $\nabla^{A*}$ and ${\rm d}^{A*}:\Omega^k(M,\tilde V^*)\to\Omega^{k+1}(M,\tilde V^*)$, respectively. Lastly, the divergence of $\nabla^{A*}$ is denoted by ${\rm div}^{A*}:\Lambda_m^{k+1}(M,\tilde V^*)\to\Lambda_m^k(M,\tilde V^*)$. 

\begin{remark}[Notation]\rm
Note that we are using the same symbols for the linear connections, covariant exterior derivatives and divergence as for the particular case where $\tilde V=\tilde{\mathfrak g}$ is the adjoint bundle (recall \S\ref{sec:geometricsetting_gauge}). The specific meaning of these symbols will be clarified from the context.
\end{remark}

On the other hand, the infinitesimal action of $\mathfrak g$ on $V$ induces a map on the associated bundles, which we denote by the same symbol for simplicity, $\tilde\varrho:\tilde{\mathfrak g}\times_M\tilde V\to\tilde V$. Given $\varphi\in\Omega^0(M,\tilde V)$, we consider the map
\begin{equation*}
\tilde\varrho_\varphi:\Omega^0(M,\tilde{\mathfrak g})\to\Omega^0(M,\tilde V),\quad\xi\mapsto\tilde\varrho(\xi,\varphi),    
\end{equation*}
whose adjoint is denoted by $\tilde\varrho_\varphi^*:\Omega^0(M,\tilde V^*)\to\Omega^0(M,\tilde{\mathfrak g}^*)$. Namely, it satisfies $\eta\cdot\tilde\varrho_\varphi(\xi)=\tilde\varrho_\varphi^*(\eta)\cdot\xi$ for each $\xi\in\Omega^0(M,\tilde{\mathfrak g})$ and $\eta\in\Omega^0(M,\tilde V^*)$. By acting trivially on $TM$ and $T^*M$, the previous maps may be extended to
\begin{equation*}
\tilde\varrho_\varphi:\Lambda_r^s(M,\tilde{\mathfrak g})\to\Lambda_r^s(M,\tilde V),\qquad\tilde\varrho_\varphi^*:\Lambda_r^s(M,\tilde V^*)\to\Lambda_r^s(M,\tilde{\mathfrak g}^*).
\end{equation*}
The identity proven in the following Lemma will be used in the derivation of the Lagrange-Dirac equations.

\begin{lemma}\rm\label{lemma:identityadjoint}
For each $\varphi\in\Omega^0(M,\tilde V)$, $\delta A\in\Omega^1(M,\tilde{\mathfrak g})$ and $\chi\in \Lambda_m^1(M,\tilde V^*)$, we have 
\begin{equation*}
\tilde\varrho_\varphi^*(\chi)\bcdot\delta A=\delta A\bwedge\tilde\varrho_\varphi^*(\Phi_{\tilde V}(\chi)),
\end{equation*}
with $\Phi_{\tilde V}$ defined in \eqref{eq:PhiE}.
\end{lemma}

\begin{proof}
Let us use local coordinates as in Remark \ref{remark:multindex}. By linearity, it is enough to prove the result for pure elements: $\delta A=\alpha_i\,dx^i\otimes\xi$ and $\chi=U^i\,\partial_i\otimes d^m x\otimes\eta$ for some (local) functions $\alpha_i,U^i\in C^\infty(M)$, $\xi\in\Omega^0(M,\tilde{\mathfrak g})$ and $\eta\in\Omega^0(M,\tilde V^*)$. The result is now a straightforward computation using \eqref{eq:PhiE}:
\begin{align*}
\tilde\varrho_\varphi^*(\chi)\bcdot\delta A & =U^i\,\alpha_i\left(\tilde\varrho_\varphi^*(\eta)\cdot\xi\right)d^m x=\alpha_i\,U^i\left(\xi\cdot\tilde\varrho_\varphi^*(\eta)\right)\left(dx^i\wedge d_i^{m-1}x\right)\\
& =\delta A\bwedge\tilde\varrho_\varphi^*(U^i\,d_i^{m-1}x\otimes\eta)=\delta A\bwedge\tilde\varrho_\varphi^*(\Phi_{\tilde V}(\chi)).
\end{align*}
\end{proof}

\paragraph{Lagrangian density and Lagrange--Dirac equations}
A Lagrangian density for a non-Abelian gauge theory interacting with matter in the space+time decomposition is a bundle morphism,
\begin{equation*}
\mathscr L_{\rm int}:\operatorname{Conn}(P)\times_M(T^*M\otimes\tilde{\mathfrak g})\times_M\left(\textstyle\bigwedge^2 T^*M\otimes \tilde{\mathfrak g}\right)\times_M\tilde V\times_M \tilde V\times_M\left(T^*M\otimes \tilde V\right)\to\textstyle\bigwedge^m T^*M,
\end{equation*}
of the form
\begin{equation}\label{eq:lagrangiandensity_interaction}
\mathscr L_{\rm int}(A_x,\varepsilon_x,\beta_x,\varphi_x,\nu_x,\zeta_x)=\mathscr L_{\rm gau}(A_x,\varepsilon_x,\beta_x)+\mathscr L_{\rm mat}(\varphi_x,\nu_x,\zeta_x),
\end{equation}
where $\mathscr L_{\rm gau}$ and $\mathscr L_{\rm mat}$ are as in \eqref{eq:lagrangiandensitynonabelian} and \eqref{Lagrangian_particle} (or \eqref{Lagrangian_density_bundle} with $k=0$), respectively. The corresponding Lagrangian is the map $L_{\rm int}:TQ\to\mathbb R$ given by
\begin{equation}\label{eq:lagrangianinteraction}
\begin{aligned}
L_{\rm int}(A,\varepsilon,\varphi,\nu) & =\int_M\mathscr L_{\rm int}(A,\varepsilon,B_A,\varphi,\nu,{\rm d}^A\varphi)\\
& =\int_M\left(\mathscr L_{\rm gau}(A,\varepsilon,B_A)+\mathscr L_{\rm mat}\left(\varphi,\nu,{\rm d}^A\varphi\right)\right)
\end{aligned}
\end{equation}
for each $(A,\varepsilon,\varphi,\nu)\in TQ\simeq\mathcal C(P)\times\Omega^1(M,\tilde{\mathfrak g})\times T\Omega^0(M, \tilde V)$. The interaction arises from the dependence on $A$ in the exterior covariant derivative of the field in $\mathscr{L}_{\rm mat}$.

\begin{lemma}\rm
For each $(A,\varepsilon,\varphi,\nu)\in TQ$, the fiber derivatives of the Lagrangian \eqref{eq:lagrangianinteraction} are given as follows (we omit the arguments for brevity):
\begin{enumerate}[(i)]
    \item When working with $T^\star Q$, they read
    \begin{align*}
    \frac{\delta L_{\rm int}}{\delta A} & =\left(\frac{\partial\mathscr L_{\rm gau}}{\partial A}+\operatorname{div}^{A*}\left(\frac{\partial\mathscr L_{\rm gau}}{\partial\beta}\right)+\tilde\varrho_\varphi^*\left(\frac{\partial\mathscr L_{\rm mat}}{\partial\zeta}\right),~\iota_{\partial M}^*\left(\tr \frac{\partial\mathscr L_{\rm gau}}{\partial\beta}\right)\right),\\
    \frac{\delta L_{\rm int}}{\delta\varepsilon} & =\left(\frac{\partial\mathscr L_{\rm gau}}{\partial\varepsilon},~0\right),\\
    \frac{\delta L_{\rm int}}{\delta\varphi} & =\left(\frac{\partial\mathscr L_{\rm mat}}{\partial\varphi}-\operatorname{div}^{A*} \left( \frac{\partial\mathscr L_{\rm mat}}{\partial\zeta}\right) ,~\iota_{\partial M}^*\left(\tr\frac{\partial\mathscr L_{\rm mat}}{\partial\zeta}\right)\right),\\
    \frac{\delta L_{\rm int}}{\delta\nu } & =\left(\frac{\partial\mathscr L_{\rm mat}}{\partial \nu },~0\right).
    \end{align*}
    \item When working with $T^\dagger Q$, they read
    \begin{align*}
    \frac{\delta L_{\rm int}}{\delta A} & =\left(\frac{\partialnew\mathscr L_{\rm gau}}{\partialnew A}+{\rm d}^{A*}\left(\frac{\partialnew\mathscr L_{\rm gau}}{\partialnew\beta}\right)+\tilde\varrho_\varphi^*\left(\frac{\partialnew\mathscr L_{\rm mat}}{\partialnew\zeta}\right) ,~\iota_{\partial M}^*\left(\frac{\partialnew\mathscr L_{\rm gau}}{\partialnew\beta}\right)\right),\\
    \frac{\delta L_{\rm int}}{\delta\varepsilon} & =\left(\frac{\partialnew\mathscr L_{\rm gau}}{\partialnew\varepsilon},~0\right),\\
    \frac{\delta L_{\rm int}}{\delta\varphi} & =\left(\frac{\partialnew\mathscr L_{\rm mat}}{\partialnew\varphi}-{\rm d}^{A*} \left( \frac{\partialnew\mathscr L_{\rm mat}}{\partialnew\zeta}\right) ,~\iota_{\partial M}^*\left(\frac{\partialnew\mathscr L_{\rm mat}}{\partialnew\zeta}\right)\right),\\
    \frac{\delta L_{\rm int}}{\delta\nu } & =\left(\frac{\partialnew\mathscr L_{\rm mat}}{\partialnew \nu },~0\right).
    \end{align*}
\end{enumerate}
\end{lemma}

\begin{proof}
The result is a straightforward consequence of Lemmas \ref{lemma:partialderivativevectorbundlekforms} and \ref{lemma:partialderivativegauge}, as well as the following computation: let $\delta A\in\Omega^1(M,\tilde{\mathfrak g})$, and denote by $\delta\mathbf A\in\overline\Omega^1(P,\mathfrak g)$ and $\boldsymbol\varphi\in\overline\Omega^0(P, V)= C_G^\infty(P, V)$ the basic 1-form of the adjoint type and the $\varrho$-invariant function corresponding to $\delta A$ and $\varphi$, respectively. We have
\begin{align*}
\left.\frac{d}{d\epsilon}\right|_{\epsilon=0}{\rm d}^{\mathbf A+\epsilon\delta\mathbf A}\boldsymbol\varphi & =\left.\frac{d}{d\epsilon}\right|_{\epsilon=0}{\rm d}\boldsymbol\varphi+ \tilde\varrho(\mathbf A+\epsilon\delta\mathbf A,\boldsymbol\varphi)=\tilde\varrho(\delta\mathbf A,\boldsymbol\varphi),
\end{align*}
where we have used a modification of the formula for ${\rm d}^{\mathbf A}\boldsymbol\varphi$ given in \cite[\S 2.3]{GBRa2008} to account for a general $\varrho$ instead of the adjoint representation. Hence, $d/d\epsilon|_{\epsilon=0}{\rm d}^{A+\epsilon\delta A}\varphi= \tilde\varrho_{\varphi}(\delta A)$ and, thus,
\begin{align*}
\left.\frac{d}{d\epsilon}\right|_{\epsilon=0}\int_M\mathscr L_{\rm mat}(\varphi,\nu,{\rm d}^{A+\epsilon\delta A}\varphi)=\int_M\frac{\partial\mathscr L_{\rm mat}}{\partial\zeta}\bcdot \tilde\varrho_\varphi(\delta A)=\int_M \tilde\varrho_\varphi^*\left(\frac{\partial\mathscr L_{\rm mat}}{\partial\zeta}\right)\bcdot\delta A.
\end{align*}
We conclude by using Lemma \ref{lemma:identityadjoint}.
\end{proof}

\medskip

By gathering the previous lemma with Propositions \ref{proposition:LDVectBundValued_kforms} and \ref{proposition:LDnonabelian}, we obtain the forced Lagrange--Dirac equations for non-Abelian gauge theories interacting with matter fields.

\begin{proposition}\rm
Associated to each choice of the restricted dual space and denoting $E=-\varepsilon$, the following statements hold:
\begin{enumerate}[(i)]
    \item The forced Lagrange--Dirac equations for a curve
    \begin{equation*}
    (A,\varepsilon,\varsigma,\varsigma_\partial,\varphi,\nu,\alpha,\alpha_\partial):[t_0,t_1]\to TQ\oplus T^\star Q
    \end{equation*}
    are given by
    \begin{equation}\label{LDequations_interaction_star}
    \left\{\begin{array}{ll}
    \dot A=-E,& \dot\varphi=\nu,\vspace{4mm}\\
    \displaystyle\varsigma=\frac{\partial\mathscr L_{\rm gau}}{\partial\varepsilon},\; & \displaystyle\dot\varsigma=\frac{\partial\mathscr L_{\rm gau}}{\partial A}+\operatorname{div}^{A*} \left(\frac{\partial\mathscr L_{\rm gau}}{\partial\beta}\right)+ \tilde\varrho_\varphi^*\left(\frac{\partial\mathscr L_{\rm mat}}{\partial\zeta}\right) +\hat{\mathcal F}^{\star},\vspace{2mm}\\
    \varsigma_\partial=0, & \displaystyle\dot\varsigma_\partial=\iota_{\partial M}^*\left(\tr \frac{\partial\mathscr L_{\rm gau}}{\partial\beta}\right)+\hat{\mathcal F}^{\star}_{\partial},\vspace{2mm}\\
    \displaystyle\alpha=\frac{\partial\mathscr L_{\rm mat}}{\partial\nu},\; & \displaystyle\dot\alpha=\frac{\partial\mathscr L_{\rm mat}}{\partial\varphi}-\operatorname{div}^{A*} \left( \frac{\partial\mathscr L_{\rm mat}}{\partial\zeta}\right)+\mathcal F^{\star},\vspace{2mm}\\
    \alpha_\partial=0, & \displaystyle\dot\alpha_\partial=\iota_{\partial M}^*\left(\tr \frac{\partial\mathscr L_{\rm mat}}{\partial\zeta}\right)+\mathcal F^{\star}_{\partial}.
    \end{array}\right.
    \end{equation}
    \item The forced Lagrange--Dirac equations for a curve
    \begin{equation*}
    (A,\varepsilon,\varsigma,\varsigma_\partial,\varphi,\nu,\alpha,\alpha_\partial):[t_0,t_1]\to TQ\oplus T^\dagger Q
    \end{equation*}
    are given by
    \begin{equation}\label{LDequations_interaction_dagger}
    \left\{\begin{array}{ll}
    \dot A=-E,& \dot\varphi=\nu\vspace{0.4cm}\\
    \displaystyle\varsigma=\frac{\partialnew\mathscr L_{\rm gau}}{\partialnew\varepsilon},\; & \displaystyle\dot\varsigma=\frac{\partialnew\mathscr L_{\rm gau}}{\partialnew A}+{\rm d}^{A*}\left( \frac{\partialnew\mathscr L_{\rm gau}}{\partialnew\beta}\right)+ \tilde\varrho_\varphi^*\left(\frac{\partialnew\mathscr L_{\rm mat}}{\partialnew\zeta}\right)+\hat{\mathcal F}^\dagger,\vspace{0.2cm}\\
    \varsigma_\partial=0, & \displaystyle\dot\varsigma_\partial=\iota_{\partial M}^*\left(\frac{\partialnew\mathscr L_{\rm gau}}{\partialnew\beta}\right)+\hat{\mathcal F}_\partial^\dagger,\vspace{2mm}\\
    \displaystyle\alpha=\frac{\partialnew\mathscr L_{\rm mat}}{\partialnew\nu},\; & \displaystyle\dot\alpha=\frac{\partialnew\mathscr L_{\rm mat}}{\partialnew\varphi}-{\rm d}^{A*} \left( \frac{\partialnew\mathscr L_{\rm mat}}{\partialnew\zeta}\right)+\mathcal F^\dagger,\vspace{0.2cm}\\
    \alpha_\partial=0, & \displaystyle\dot\alpha_\partial=\iota_{\partial M}^*\left(\frac{\partialnew\mathscr L_{\rm mat}}{\partialnew\zeta}\right)+\mathcal F^\dagger_{\partial}.
    \end{array}\right.
    \end{equation}
\end{enumerate}
\end{proposition}

\begin{remark}[Lagrange--d'Alembert form]\rm As before, by eliminating the variables $\varsigma(t)$, $\varsigma_\partial(t)$, $\varepsilon(t)$ and $\alpha(t)$, $\alpha_\partial(t)$, $\nu(t)$ in 
\eqref{LDequations_interaction_star} and \eqref{LDequations_interaction_dagger} we get the equations in \textit{Lagrange--d'Alembert form} as follows
\begin{equation*}%\label{LdA_connection_star}
\left\{\begin{array}{l}
\displaystyle\frac{\partial }{\partial t}\frac{\partial\mathscr L_{\rm gau}}{\partial\varepsilon}=\frac{\partial\mathscr L_{\rm gau}}{\partial A}+\operatorname{div}^{ A_*}\left(\frac{\partial\mathscr L_{\rm gau}}{\partial\beta}\right) + \tilde\varrho_\varphi^* \left(\frac{\partial\mathscr L_{\rm mat}}{\partial\zeta}\right)+\hat{\mathcal{F}}^\star,\vspace{2mm}\\
\displaystyle\frac{\partial }{\partial t}\frac{\partial\mathscr L_{\rm mat}}{\partial\nu}=\frac{\partial\mathscr L_{\rm mat}}{\partial \varphi}-\operatorname{div}^{ A_*}\left(\frac{\partial\mathscr L_{\rm mat}}{\partial\zeta}\right)+\mathcal{F}^\star,\vspace{2mm}\\
\displaystyle\hat{\mathcal{F}}_\partial^\star =-\iota_{\partial M}^*\left(\tr\frac{\partial\mathscr L_{\rm gau}}{\partial\beta}\right), \qquad \mathcal{F}_\partial^\star =-\iota_{\partial M}^*\left(\tr\frac{\partial\mathscr L_{\rm mat}}{\partial\zeta}\right),
\end{array}\right.
\end{equation*}
and
\begin{equation*}%\label{LdA_connection_dagger}
\left\{\begin{array}{l}
\displaystyle\frac{\partial }{\partial t}\frac{\partialnew\mathscr L_{\rm gau}}{\partialnew\varepsilon}=\frac{\partialnew\mathscr L_{\rm gau}}{\partialnew A}+{\rm d}^{A_*} \left( \frac{\partialnew\mathscr L_{\rm gau}}{\partialnew\beta} \right)
+ \tilde\varrho_\varphi^*\left( \frac{\partialnew\mathscr L_{\rm mat}}{\partialnew\zeta} \right) +\hat{\mathcal{F}}^{\dagger},\vspace{2mm}\\
\displaystyle\frac{\partial }{\partial t}\frac{\partialnew\mathscr L_{\rm mat}}{\partialnew\nu}=\frac{\partialnew\mathscr L_{\rm mat}}{\partialnew \varphi}-{\rm d}^{A_*} \left( \frac{\partialnew\mathscr L_{\rm mat}}{\partialnew\zeta} \right)
+\mathcal{F}^{\dagger},\vspace{2mm}\\
\displaystyle\hat{\mathcal{F}}^{\dagger}_\partial=-\iota_{\partial M}^*\left(\frac{\partialnew\mathscr L_{\rm gau}}{\partialnew\beta}\right),\qquad \mathcal{F}^{\dagger}_\partial=-\iota_{\partial M}^*\left(\frac{\partialnew\mathscr L_{\rm mat}}{\partialnew\zeta}\right).
\end{array}\right.
\end{equation*}
\end{remark}

\paragraph{Energy and interaction} The total energy density is given by the sum $\mathscr{E}_{\rm tot}= \mathscr{E}_{\rm gau}+ \mathscr{E}_{\rm mat}$, which represents the energy contributions from the gauge field and the matter field, see \eqref{eq:lagrangiandensity_interaction} and \eqref{eq:lagrangianinteraction}. 
It is instructive to first examine the local energy balance for each component separately. Focusing on the description of the restricted dual given by $T^\dagger Q$ we obtain:
\[
\frac{\partial }{\partial t}\mathscr{E}_{\rm gau}=-{\rm d}\left(\varepsilon\bwedge\frac{\partialnew\mathscr L_{\rm gau}}{\partialnew\beta}\right)+ \dot A \bwedge  \tilde\varrho_\varphi^*\left( \frac{\partialnew\mathscr L_{\rm mat}}{\partialnew\zeta} \right)+\varepsilon\bwedge\hat{\mathcal F}^\dagger
\]
\[
\frac{\partial }{\partial t}\mathscr{E}_{\rm mat}=-{\rm d}\left(\nu\bwedge\frac{\partialnew\mathscr L_{\rm mat}}{\partialnew\zeta}\right) - \tilde{\varrho}_\varphi(\dot A) \bwedge \frac{\partialnew\mathscr L_{\rm mat}}{\partialnew\zeta}+\nu\bwedge\mathcal F^\dagger.
\]
These expression highlight the presence of the interaction terms (specifically, the second terms on the right-hand sides), which cancel out in the total energy balance equation for $\mathscr{E}_{\rm tot}$. The local and global energy balances involve the total energy flux density, described either as the vector field density $S_{\rm tot}=\frac{\partial \mathscr{L}_{\rm gau}}{\partial \beta}\bcdot \dot A + \frac{\partial \mathscr{L}_{\rm mat}}{\partial \zeta}\bcdot \dot \varphi$, or as the $(m-1)$-form $\mathsf{S}_{\rm tot}=\dot A\bwedge \frac{\partialnew\mathscr{L}_{\rm gau}}{\partialnew \beta}+\dot\varphi\bwedge \frac{\partialnew\mathscr{L}_{\rm mat}}{\partialnew \zeta}$, as well as the spatially distributed and boundary contributions arising from both the gauge and matter sectors, see Propositions \ref{prop:energybalancevectorbundle} and \ref{prop:energybalancegauge}. These will be detailed further in the context of the Yang--Mills--Higgs system below.

\subsection{Yang--Mills--Higgs equations}\label{sec:YMHeqs}

We give here an important application of the interaction of the matter and gauge fields described previously. We consider the geometric setting introduced in Section \ref{sec:YMlagrangiandensity}.

\begin{definition}\rm
The \emph{Yang--Mills--Higgs Lagrangian density} is the Lagrangian density \eqref{eq:lagrangiandensity_interaction} with
\begin{enumerate}[(i)]
    \item $\mathscr L_{\rm gau}(A_x,\varepsilon_x,\beta_x)=\mathscr L_{\rm YM}(A_x,\varepsilon_x,\beta_x)=\displaystyle\frac{1}{2}\left(\mathbf g\left(\varepsilon_x,\varepsilon_x\right)-\mathbf g\left(\beta_x,\beta_x\right)\right)\mu_g$ (recall \eqref{split_Lagrangian_density}).
    \item $\mathscr L_{\rm mat}(\varphi_x,\nu_x,\zeta_x)=\left(\displaystyle\frac{1}{2}\mathbf g_{ \kappa}(\nu_x,\nu_x)-\frac{1}{2}\mathbf g_{\kappa}(\zeta_x,\zeta_x)-\mathbf V(\varphi_x)\right)\mu_g$ for some potential $\mathbf V:\tilde V\to\mathbb R$.
\end{enumerate}
\end{definition}

The choice $\mathbf V(\varphi_x)=m\,\mathbf g_{ \kappa}(\varphi_x,\varphi_x)$ yields the \emph{Klein--Gordon field} of mass $m>0$, whilst $\mathbf V(\varphi_x)=\lambda\,\mathbf g_{ \kappa}(\varphi_x,\varphi_x)^2-\mu_H\,\mathbf g_{\kappa}(\varphi_x,\varphi_x)$ yields the \emph{Higgs field} of mass $m_H=\sqrt{2\,\mu_H}>0$.

The partial derivative of the potential is defined in the same vein as the functional derivatives of the Lagrangian density. More specifically, it is the bundle morphism $\partial\mathbf V/\partial\varphi:\tilde V\to \tilde V^*$ given by
\begin{equation*}
\frac{\partial\mathbf V}{\partial\varphi}(\varphi_x)\cdot \delta\varphi_x
=\left.\frac{d}{d\epsilon}\right|_{\epsilon=0}\mathbf V(\varphi_x+\epsilon\delta\varphi_x),\qquad\varphi_x,\delta\varphi_x\in\tilde V_x,~x\in M.
\end{equation*}
Moreover, we denote $\operatorname{grad}_{\kappa}\mathbf V=\sharp_{\kappa}\circ\partial\mathbf V/\partial\varphi: \tilde V\to \tilde V$. The following result is straightforward by noting that the Lagrangian density for the matter field may be equivalently written as $\mathscr L_{\rm mat}(\varphi_x,\nu_x,\zeta_x)=\displaystyle\frac{1}{2}\nu_x^{\flat_\kappa}\wedge\boldsymbol\star\nu_x-\frac{1}{2}\zeta_x^{\flat_\kappa}\wedge\boldsymbol\star\zeta_x-\star\mathbf V(\varphi_x)$.

\begin{lemma}\rm
For each $(\varphi,\dot\varphi)\in T\Omega^0(M, \tilde V)$, by denoting $\zeta={\rm d}^A\varphi\in\Omega^1(M,\tilde V)$, the partial derivatives of the Lagrangian density for a matter field are given by
\begin{align*}
& \frac{\partial\mathscr L_{\rm mat}}{\partial\varphi}=-\frac{\partial\mathbf V}{\partial\varphi}\otimes\mu_g, && \frac{\partial\mathscr L_{\rm mat}}{\partial\nu}=\dot\varphi^{\boldsymbol\sharp_{\kappa}}\otimes\mu_g, && \frac{\partial\mathscr L_{\rm mat}}{\partial\zeta}=-\left({\rm d}^A\varphi\right)^{\boldsymbol\sharp_{\kappa}}\otimes\mu_g,\\
& \frac{\partialnew\mathscr L_{\rm mat}}{\partialnew\varphi}=-\boldsymbol\star\frac{\partial\mathbf V}{\partial\varphi}, && \frac{\partialnew\mathscr L_{\rm mat}}{\partialnew\nu}=\boldsymbol\star\dot\varphi^{\flat_{\kappa}}, && \frac{\partialnew\mathscr L_{\rm mat}}{\partialnew\zeta}=-\boldsymbol\star\left({\rm d}^A\varphi\right)^{\flat_{\kappa}}.
\end{align*}
\end{lemma}

The forced Lagrange--Dirac equations \eqref{LDequations_interaction_star} for the Yang--Mills--Higgs Lagrangian density read
\begin{equation*}
\left\{\begin{array}{ll}
\dot A=-E,& \dot\varphi=\nu,\vspace{4mm}\\
\displaystyle\varsigma=-E^{\boldsymbol\sharp}\otimes\mu_g,\; & \displaystyle\dot\varsigma=-\operatorname{div}^{A*} \left(B_A^{\boldsymbol\sharp}\otimes\mu_g\right)-\tilde\varrho_\varphi^*\left(({\rm d}^A\varphi)^{\boldsymbol\sharp_\kappa}\otimes\mu_g\right) +\hat{\mathcal F}^{\star},\vspace{2mm}\\
\varsigma_\partial=0, & \displaystyle\dot\varsigma_\partial=-\iota_{\partial M}^*\left(\tr\left(B_A^{\boldsymbol\sharp}\otimes\mu_g\right)\right)+\hat{\mathcal F}^{\star}_{\partial},\vspace{2mm}\\
\displaystyle\alpha=\dot\varphi^{\boldsymbol\sharp}\otimes\mu_g,\; & \displaystyle\dot\alpha= -\frac{\partial\mathbf V}{\partial\varphi}\otimes\mu_g+\operatorname{div}^{A*} \left(\left({\rm d}^A\varphi\right)^{\boldsymbol\sharp_{\kappa}}\otimes\mu_g\right)+\mathcal F^{\star},\vspace{2mm}\\
\alpha_\partial=0, & \displaystyle\dot\alpha_\partial=-\iota_{\partial M}^*\left(\tr \left({\rm d}^A\varphi\right)^{\boldsymbol\sharp_{\kappa}}\otimes\mu_g\right)+\mathcal F^{\star}_{\partial}.
\end{array}\right.
\end{equation*}
Analogously, the forced Lagrange--Dirac equations \eqref{LDequations_interaction_dagger} for the Yang--Mills--Higgs Lagrangian density read
\begin{equation*}
\left\{\begin{array}{ll}
\dot A=-E,& \dot\varphi=\nu\vspace{0.4cm}\\
\displaystyle\varsigma=-\boldsymbol\star E^{\flat_K},\; & \displaystyle\dot\varsigma=-{\rm d}^{A*}\left(\boldsymbol\star B_A^{\flat_K}\right)-\tilde\varrho_\varphi^*\left(\boldsymbol\star\left({\rm d}^A\varphi\right)^{\flat_{\kappa}}\right)+\hat{\mathcal F}^\dagger,\vspace{0.2cm}\\
\varsigma_\partial=0, & \displaystyle\dot\varsigma_\partial=-\,\iota_{\partial M}^*\left(\boldsymbol\star B_A^{\flat_K}\right)+\hat{\mathcal F}_\partial^\dagger,\vspace{2mm}\\
\displaystyle\alpha=\boldsymbol\star\dot\varphi^{\flat_{\kappa}},\; & \displaystyle\dot\alpha= -\boldsymbol\star\frac{\partial\mathbf V}{\partial\varphi}+{\rm d}^{A*}\left(\boldsymbol\star\left({\rm d}^A\varphi\right)^{\flat_{\kappa}}\right)+\mathcal F^\dagger,\vspace{0.2cm}\\
\alpha_\partial=0, & \displaystyle\dot\alpha_\partial=-\iota_{\partial M}^*\left(\boldsymbol\star\left({\rm d}^A\varphi\right)^{\flat_{\kappa}}\right)+\mathcal F^\dagger_{\partial}.
\end{array}\right.
\end{equation*}

For the latter equations, by using Lemma \ref{lemma:metricconnection}, eliminating the variables $\dot A,\varsigma,\varsigma_\partial,\nu,\alpha$ and $\alpha_\partial$, and denoting by $J:T\mathcal C(P)\to\Omega^1(M,\tilde{\mathfrak g})$, $j:T\mathcal C(P)\to\Omega^1(\partial M,\tilde{\mathfrak g})$, $\beth:T\Omega^0(M,\tilde{\mathfrak g})\to\Omega^0(M,\tilde{\mathfrak g})$ and $\gimel:T\Omega^0(M,\tilde V)\to\Omega^0(\partial M, \tilde V)$ the maps defined as $\boldsymbol\star\hat{\mathcal F}^\dagger=-J^{\flat_K}$, $\boldsymbol\star_\partial\hat{\mathcal F}_\partial^\dagger=j^{\flat_K}$, $\mathcal F^\dagger=\boldsymbol\star\beth^{\flat_{\kappa}}$ and $\mathcal F_\partial^\dagger=\boldsymbol\star_\partial\gimel^{\flat_{\kappa}}$, we obtain
\begin{equation}\label{eq:YMH_eqs}
\left\{\begin{array}{l}
\displaystyle\dot E-\delta^A B_A+\tilde\varrho_\varphi^*\left(({\rm d}^A\varphi)^{\flat_\kappa}\right)^{\sharp_K}=-J,\vspace{0.2cm}\\
\displaystyle\boldsymbol\star_\partial\left(\iota_{\partial M}^*\left(\boldsymbol\star B_A\right)\right)=j,\vspace{2mm}\\
\end{array}\right.\qquad\left\{\begin{array}{l}
\displaystyle\ddot\varphi+\delta^A\left({\rm d}^A\varphi\right)+\operatorname{grad}_{\kappa}\mathbf V=\beth,\vspace{0.2cm}\\
\displaystyle\boldsymbol\star_\partial\left(\iota_{\partial M}^*\left(\boldsymbol\star{\rm d}^A\varphi\right)\right)=\gimel.
\end{array}\right.
\end{equation}

Together with the equations ${\rm d}^AB_A=0$ and $\dot B_A=- {\rm d}^AE$, which follow from $B_A= {\rm d}^AA$ and $E=-\dot A$, these are the \textit{Yang--Mills--Higgs equations with external currents and boundary conditions} in the space+time decomposition.

\begin{remark}\rm
In order to obtain the second interior equation of \eqref{eq:YMH_eqs}, we have used that
\begin{align*}
\boldsymbol\star\left({\rm d}^{A*}\left(\boldsymbol\star\left({\rm d}^A\varphi\right)^{\flat_\kappa}\right)\right) & =\boldsymbol\star\left({\rm d}^{A}\left(\boldsymbol\star\left({\rm d}^A\varphi\right)\right)\right)^{\flat_\kappa}=\left(\delta^A\left({\rm d}^A\varphi\right)\right)^{\flat_\kappa}.
\end{align*}
This is true by assuming that $\nabla^A$ is a metric connection with respect to $\kappa$, whence $\flat_\kappa\circ{\rm d}^A={\rm d}^{A*}\circ\flat_\kappa$. If that is not the case, such equation would read:
\begin{equation*}
\displaystyle\ddot\varphi+\left(\delta^{A*}\left({\rm d}^A\varphi\right)^{\flat_\kappa}\right)^{\sharp_\kappa}+\operatorname{grad}_\kappa\mathbf V=\beth.
\end{equation*}
\end{remark}

An equivalent expression for the previous boundary conditions may be shown using coordinates:
\begin{equation*}
\iota_{\partial M}^*(i_n\,B_A)=(-1)^m\frac{|g|}{|g_\partial|}\,j,\qquad\iota_{\partial M}^*(i_n\,{\rm d}^A\varphi)=\frac{|g|}{|g_\partial|}\,\gimel.
\end{equation*}
These constitute a generalization of the boundary conditions discussed in \cite[Eq. (1.2)]{AiSoZh2019} to include boundary currents (with notations of that paper, we have $\mathcal K=\mathcal F$ and, thus, $T_\varphi\mathcal K^v=\{0\}$). For an alternative approach to Yang--Mills--Higgs theories with boundary conditions see \cite{Ta1997}.

Lastly, let us compute the energy balance. From \eqref{eq:lagrangiandensity_interaction}, it is clear that the energy density for the Yang--Mills--Higgs field is the sum of the energy densities of the Yang--Mills and the matter fields, i.e.,
\begin{align*}
\mathscr E_{\rm int} & =\mathscr E_{\rm YM}+\mathscr E_{\rm mat}=\left(\frac{1}{2}\mathbf{g}(E,E)+\frac{1}{2}\mathbf{g}(B_A,B_A)+\frac{1}{2}\mathbf g_{\kappa}(\dot\varphi,\dot\varphi)+\frac{1}{2}\mathbf g_{\kappa}({\rm d}^A\varphi,{\rm d}^A\varphi)+\mathbf V(\varphi)\right)\mu_g\\
& =\frac{1}{2}\left(E^{\flat_K}\bwedge\boldsymbol\star E+B_A^{\flat_K}\bwedge\boldsymbol\star B_A+\dot\varphi^{\flat_{\kappa}}\bwedge\boldsymbol\star\dot\varphi+({\rm d}^A\varphi)^{\flat_{\kappa}}\bwedge\boldsymbol\star{\rm d}^A\varphi\right)+\mathbf V(\varphi)\mu_g.
\end{align*}
Therefore, the local and global energy balances are obtained by adding the corresponding balances; namely,
\begin{align*}
\frac{\partial\mathscr E_{\rm tot}}{\partial t}=  -{\rm d}\mathsf S_{\rm gau}+E\bwedge\boldsymbol{\star}\, J^{\flat_K}-{\rm d}\mathsf S_{\rm mat}+\dot\varphi\bwedge\boldsymbol\star\beth^{\flat_{\kappa}},
\end{align*}
and
\begin{align*}
\frac{d}{dt} \int_M \mathscr{E}_{\rm tot}= & \underbrace{\int_M\big(E\bwedge\boldsymbol{\star}\, J^{\flat_K}+\dot\varphi\bwedge\boldsymbol\star\beth^{\flat_{\kappa}}\big)}_{\text{spatially distributed contribution}}\\
& +\underbrace{\int_{\partial M}\big(\big(\boldsymbol\star_\partial\,j^{\flat_K}\big)\bwedge\iota_{\partial M}^*E+\big(\iota_{\partial M}^*\dot\varphi\big)\bwedge\boldsymbol\star_\partial\gimel^{\flat_{\kappa}}\big)}_{\text{boundary contribution}}.
\end{align*}
where $\mathsf S_{\rm gau}=\mathsf S_{\rm YM}=E\bwedge\boldsymbol\star B_A^{\flat_K}$ and $\mathsf S_{\rm mat}=-\dot\varphi\bwedge\boldsymbol\star\left({\rm d}^A\varphi\right)^{\flat_{\kappa}}$.

As for the pure Yang--Mills theory, the \emph{charge density} is given by $\rho=\delta^A E$. From the first interior equation of \eqref{eq:YMH_eqs}, we obtain the charge conservation in presence of a Higgs field:
\begin{equation*}
\dot\rho+\delta^A \left(J-\tilde\varrho_\varphi^*\big(({\rm d}^A\varphi)^{\flat_\kappa}\big)^{\sharp_K}\right)=0.
\end{equation*}

\section{Conclusion}

In this paper, we developed a geometric formulation for systems that exchanges energy with their surroundings via energy flow, using the framework of infinite-dimensional Dirac structures. Our primary focus was on field theories involving bundle-valued $k$-forms, gauge field theories, and their interactions within a space-time decomposition. This work was built on the geometric framework introduced in Part I of this paper (see \cite{GBRAYo2025I}). The aim was to establish the foundational geometric aspects of the theory within a general setting that included potentially nontrivial vector and principal bundles, thereby highlighting the wide range of systems this approach can accommodate.
Compared to previously proposed approaches for systems with boundary energy flow, the Lagrange–Dirac framework presented here fulfills several key consistency requirements:
\begin{itemize}
\item[(1)] The infinite-dimensional Dirac structures employed are associated with a canonical symplectic form on the system's infinite-dimensional phase space, extending the classical symplectic structure $ dq \wedge dp$ from finite-dimensional mechanics.
\item[(2)] The evolution equations of the Lagrange–Dirac system arise as critical conditions of a variational principle. 
\item[(3)] This variational principle is a direct extension of Hamilton’s principle, $\delta \!\int\! L(q, \dot q) dt=0$, to the infinite-dimensional setting, naturally incorporating boundary contributions.
\item[(4)] The geometric framework does not impose specific restrictions on the form or regularity of the Lagrangian density. 
\end{itemize}

These properties will be further explored and exploited in future work, particularly in relation to interconnections, symmetry reduction, and modeling.

\section*{Acknowledgements}
F.G.-B. is partially supported by a start-up grant from the Nanyang Technological University. H.Y. is partially supported by JST CREST (JPMJCR24Q5), JSPS Grant-in-Aid for Scientific Research (22K03443) and Waseda University Grants for Special Research Projects (2025C-095).

\bibliographystyle{plainnat}
\bibliography{biblio}

\end{document}